\documentclass[10pt,twoside]{article}

\usepackage{geometry}
\usepackage{color}

\usepackage[numbers,sort&compress]{natbib}
\usepackage{graphicx,latexsym,euscript,makeidx,color,bm}
\usepackage{amsmath,amsfonts,amssymb,amsthm,thmtools,mathtools,mathrsfs,enumerate}
\usepackage[colorlinks,linkcolor=blue,anchorcolor=green,citecolor=red]{hyperref}
\usepackage[utf8]{inputenc}
\usepackage[T1]{fontenc}

\def\5n{\negthinspace \negthinspace \negthinspace \negthinspace \negthinspace }
\def\4n{\negthinspace \negthinspace \negthinspace \negthinspace }
\def\3n{\negthinspace \negthinspace \negthinspace }
\def\2n{\negthinspace \negthinspace }
\def\1n{\negthinspace }

\def\ms{\medskip}

\def\no{\noindent}        \def\q{\quad}

     \def\le{\leqslant}                
\def\leq{\leqslant}       \def\geq{\geqslant}
          \def\[{\Big[}
           \def\]{\Big]}

\theoremstyle{plain}

\theoremstyle{plain}

\makeatletter

\@addtoreset{equation}{section}
\makeatother

\newtheorem{theorem}{Theorem}[section]
\newtheorem{definition}[theorem]{Definition}
\newtheorem{proposition}[theorem]{Proposition}
\newtheorem{corollary}[theorem]{Corollary}
\newtheorem{lemma}[theorem]{Lemma}
\newtheorem{remark}[theorem]{Remark}
\newtheorem{example}[theorem]{Example}

\newtheorem{assumption}{Assumption}

\makeatletter

\@addtoreset{equation}{section}
\makeatother

\allowdisplaybreaks[4]

\begin{document}
	\title{
		\Large \bf Sharp propagation of chaos for mean-field backward stochastic differential equations\thanks{Ying Hu is partially supported by the Lebesgue Center of Mathematics "Investissements d'avenir" program (ANR-11-LABX-0020-01).
			Jiaqiang Wen is supported by National Natural Science Foundation of China (No. 12571478), the Guangdong Basic and Applied Basic Research Foundation (No. 2025B1515020091), and the Shenzhen Fundamental Research General Program (No. JCYJ20230807093309021).
		}
	}
	\author{Shuxian Gao\thanks{
			Department of Mathematics,
			Southern University of Science and Technology, Shenzhen, 518055, China
			(Email: {\tt 12431007@mail.sustech.edu.cn}).}~,~~~~
	Ying Hu\thanks{Univ. Rennes, CNRS, IRMAR - UMR 6625, F-35000 Rennes, France
	(Email: {\tt ying.hu@univ-rennes.fr}).}~,~~~~
		Jiaqiang Wen\thanks{Department of Mathematics and SUSTech International Center for Mathematics,
			Southern University of Science and Technology, Shenzhen, 518055, China
			(Email: {\tt wenjq@sustech.edu.cn}).}
	}
	
	\date{}
	\maketitle

\no\bf Abstract. \rm
We study quantitative propagation of chaos for  mean-field forward-backward stochastic differential equations with interactions through the empirical measures of the forward states, backward values, and diagonal martingale integrands.
Our main contribution is to introduce an intermediate-particle comparison method that transfers forward propagation of chaos estimates to the backward system while controlling the full matrix of martingale integrands. 
Using this method, we establish sharp rates for the backward values and diagonal martingale integrands, with order $m^2/n^2$ in squared Wasserstein distance and order $n^{-1}$ in smooth weak metrics for each fixed $m$.
Explicit examples  provide  matching lower bounds, showing the optimality of these rates.
We also establish coupling estimates, including a squared Wasserstein bound of order $m/n$ with dimension-independent constants under suitable assumptions.

	\ms
	
	\no\bf Key words:  \rm 
	propagation of chaos; mean-field backward stochastic differential equations; interacting particles; Wasserstein distance; weak approximation; master equation.
	
	\rm

	\ms
	

\noindent\textbf{2020 Mathematics Subject Classification:} 
\rm 
Primary 60H10;  Secondary 60H30, 60F17.

	\section{Introduction}\label{sec1}

Propagation of chaos quantifies the convergence of the joint law of a fixed number of particles in a large interacting system to that of independent copies of a nonlinear limiting process. In the backward setting, additional difficulties arise from the martingale representation. Although the $i$-th forward equation is driven only by $W^i$, the $i$-th backward component is adapted to the full product filtration, and its martingale representation generally involves the full row of integrands $(Z^{i,j,n})_{j=1}^n$. Moreover, when the generator depends on the empirical measure of the diagonal integrands $Z^{i,i,n}$, approximation errors in these integrands feed back into the interaction.

The aim of this paper is to establish sharp propagation of chaos rates for the interacting backward components by transferring quantitative estimates from the forward particle system.
This is not straightforward, since the entropy and hierarchy methods developed for forward particle systems do not directly control the backward martingale integrands. To overcome this difficulty, we introduce an intermediate-particle comparison method based on evaluating a classical decoupling field along the empirical forward system. This method separates the intrinsic backward consistency error from the forward approximation error while controlling the full matrix of martingale integrands. Combined with suitable forward estimates, it yields sharp  Wasserstein rates under a first-order cancellation condition and sharp weak rates without this condition.

	\subsection{The particle system and its limit}\label{sec1.1}
	Fix \(T>0\) and integers \(d,k\geq1\). Let \(X_0\) be an \(\mathbb R^d\)-valued random variable and \(W\) an independent \(d\)-dimensional Brownian motion. On a complete probability space \((\Omega,\mathscr F,\mathbb P)\), let \(\mathbb F=(\mathscr F_t)_{0\leq t\leq T}\) be the usual augmentation of the filtration generated by \(X_0\) and \(W\). We consider jointly measurable coefficients
	\begin{align*}
		b&:[0,T]\times\Omega\times\mathbb R^d\times\mathcal P(\mathbb R^d)\longmapsto \mathbb R^d,\\
		\sigma&:[0,T]\times\Omega\times\mathbb R^d\times\mathcal P(\mathbb R^d)\longmapsto\mathbb R^{d\times d},\\
		f&:[0,T]\times\Omega\times\mathbb R^d\times\mathcal P(\mathbb R^d)
		\times\mathbb R^k\times\mathcal P(\mathbb R^k)
		\times\mathbb R^{k\times d}\times\mathcal P(\mathbb R^{k\times d})
	\longmapsto\mathbb R^k,\\
		h&:\Omega\times\mathbb R^d\times\mathcal P(\mathbb R^d)\longmapsto\mathbb R^k.
	\end{align*}
	The coefficients \(b,\sigma,f\) are measurable with respect to the progressive sigma-field in \(t,\omega\) and the Borel sigma-fields in their remaining variables; \(h\) is measurable with respect to the corresponding product of \(\mathscr F_T\) and Borel sigma-fields. Their quantitative assumptions are stated in \autoref{lipp}.
	All measure arguments are understood in the finite-moment spaces specified in the relevant assumptions; in \autoref{sec4} and \autoref{sec5}, these are the corresponding spaces \(\mathcal{P}_2\).
	The general coupling theory in \autoref{sec3} allows progressively measurable random coefficients, 
	whereas \autoref{sec4} and \autoref{sec5} specialize
	to the Markovian setting required for the decoupling-field arguments.
	
	For each \(n\), consider the \(n\)-fold product probability space, with the usual augmentation of the product filtration \(\mathbb F^n\). Write \(\omega=(\omega_1,\ldots,\omega_n)\) and define the coordinate copies by
	\begin{align*}
		b^i(t,\omega,x,\theta)&=b(t,\omega_i,x,\theta),\qquad
		h^i(\omega,x,\theta)=h(\omega_i,x,\theta),
	\end{align*}
	with the same convention for \(\sigma^i\) and \(f^i\). 
	Thus, the random coefficients, when present, are independently replicated across particles together with $(X_0,W)$.
	   The interacting system is
	\begin{align}
		X_t^{i,n}
		&=X_0^{i,n}+\int_0^t b^i(s,X_s^{i,n},\theta_s^n) ds
		+\int_0^t\sigma^i(s,X_s^{i,n},\theta_s^n) dW_s^i,
		\label{x1}\\
		Y_t^{i,n}
		&=h^i(X_T^{i,n},\theta_T^n)
		+\int_t^T f^i(s,X_s^{i,n},\theta_s^n,Y_s^{i,n},\mu_s^n,
		Z_s^{i,i,n},\nu_s^n) ds
		-\sum_{j=1}^n\int_t^T Z_s^{i,j,n} dW_s^j,
		\label{y1}
	\end{align}
	for \(1\leq i\leq n\), where
	\begin{align*}
		\theta_s^n=\frac1n\sum_{i=1}^n\delta_{X_s^{i,n}},\qquad
		\mu_s^n=\frac1n\sum_{i=1}^n\delta_{Y_s^{i,n}},\qquad
		\nu_s^n=\frac1n\sum_{i=1}^n\delta_{Z_s^{i,i,n}}.
	\end{align*}
Here $\delta_x$ denotes the Dirac measure at $x$, and $Z^{i,j,n}$ is $\mathbb{R}^{k\times d}$-valued. The forward system is autonomous. The generator of the $i$-th backward equation depends on the forward state, the backward value, the diagonal martingale integrand $Z^{i,i,n}$, and the empirical measures $\theta^n$, $\mu^n$, and $\nu^n$. Although only $Z^{i,i,n}$ appears as the individual martingale integrand argument of the generator, the martingale representation involves the full row $(Z^{i,j,n})_{j=1}^n$. In particular, independence of the driving Brownian motions does not imply that the off-diagonal integrands vanish. We omit the coordinate superscript on the coefficients and the particle-number superscript $n$ whenever no ambiguity arises. Random coefficients in the $i$-th particle equation are always understood as their $i$-th coordinate copies.

	The limiting system is
	\begin{align}
		X_t&=X_0+\int_0^t b(s,X_s,\mathbb P_{X_s}) ds
		+\int_0^t\sigma(s,X_s,\mathbb P_{X_s}) dW_s,
		\label{x1n}\\
		Y_t&=h(X_T,\mathbb P_{X_T})
		+\int_t^T f(s,X_s,\mathbb P_{X_s},Y_s,\mathbb P_{Y_s},Z_s,\mathbb P_{Z_s}) ds
		-\int_t^T Z_s dW_s.
		\label{y1n}
	\end{align}
Here	we write \(\mathbb P_V\) for the law of a random variable \(V\). All law
	 arguments in \eqref{x1n}--\eqref{y1n} are unconditional. 
	The laws of the martingale integrands are understood for almost every $t\in[0,T]$. For the Wasserstein estimates, $X$ and $Y$ are viewed as random variables taking values in continuous path spaces, whereas $Z$ takes values in $L^2([0,T];\mathbb R^{k\times d})$.

	\subsection{Main results and new method}\label{sec:overview}

		We establish sharp propagation of chaos rates for interacting mean-field BSDEs in Wasserstein distance and smooth weak metrics.
		To obtain these results, we introduce an intermediate-particle comparison method that handles interactions involving the empirical laws of both backward components and controls the full matrix of martingale integrands.
		This method separates the intrinsic backward consistency error from the forward approximation error, allowing suitable forward estimates to be transferred to the backward system.

To state our results, we introduce the following notation. For $1\leq m\leq n$, set
\begin{align*}
	\mathbf X^{m,n}&=(X^{1,n},\ldots,X^{m,n}),\qquad
	\mathbf Y^{m,n}=(Y^{1,n},\ldots,Y^{m,n}),\q\
	\mathbf Z^{m,n}_{\mathrm{diag}}=(Z^{1,1,n},\ldots,Z^{m,m,n}).
\end{align*}
For the Wasserstein estimates below, we use the supremum norm for $X$ and $Y$ and the $L^2$ norm for $Z$, with Euclidean norms across particle coordinates.

\paragraph{Coupling estimates.}
In \autoref{sec3}, we establish well-posedness for both the limiting and particle systems and derive quantitative estimates by synchronous coupling under \autoref{lipp}.
For \(1<p<2\), \(q=2\), and a diffusion coefficient independent of the measure argument, \autoref{coup} gives
\begin{align*}
	&\mathcal W_{p,\|\cdot\|_\infty}^p
	(\mathbb P_{\mathbf X^{m,n}},\mathbb P_X^{\otimes m})
	+\mathcal W_{p,\|\cdot\|_\infty}^p
	(\mathbb P_{\mathbf Y^{m,n}},\mathbb P_Y^{\otimes m})
	+\mathcal W_{p,L^2}^p
	(\mathbb P_{\mathbf Z^{m,n}_{\mathrm{diag}}},\mathbb P_Z^{\otimes m})
	\leq Cm\bar\lambda_{p,r,n}.
\end{align*}
Here \(\bar\lambda_{p,r,n}\) is determined by the empirical Wasserstein rates in dimensions \(d\), \(k\), and \(kd\), with the admissible range of \(r\) specified in \autoref{xyzcon}.
For interactions through finitely many Lipschitz statistics (\autoref{nodimen}), \autoref{nodi} yields a squared Wasserstein bound of order \(m/n\), allowing measure dependence in the diffusion coefficient. The corresponding constant is independent of \(d\) and \(k\), provided that the dimensions of the statistics, coefficient bounds, and moment bounds are uniformly controlled.

	\paragraph{Sharp Wasserstein estimates.}

In \autoref{sec4}, we establish the sharp Wasserstein rate in the deterministic Markovian setting.
	We assume that the diffusion coefficient is a constant invertible matrix, that the associated master equation admits a classical decoupling field \(U\) satisfying \autoref{asp4.1}, and that the finite-dimensional structure in \autoref{asp4.2} holds.
	In particular,
	\begin{align*}
		U(t,x,\mu)&=G(t,x,\Phi_t(\mu)),\q\
		D_aG(t,x,\Phi_t(\theta_t))=0,\q\
		\theta_t=\mathbb P_{X_t},\q\
		\Phi_t(\mu)=\int_{\mathbb{R}^d}\phi_U(t,v)\mu(dv).
	\end{align*}
	Together with the forward assumptions in \autoref{asp4.3}, these conditions yield, by \autoref{sharpuperr}, 
	\begin{align*}
		&\mathcal W_{2,\|\cdot\|_\infty}^2
		(\mathbb P_{\mathbf Y^{m,n}},\mathbb P_Y^{\otimes m})
		+\mathcal W_{2,L^2}^2
		(\mathbb P_{\mathbf Z^{m,n}_{\mathrm{diag}}},\mathbb P_Z^{\otimes m})
		\leq C\frac{m^2}{n^2},\q\ 1\leq m\leq n.
	\end{align*}
	Here \(C\) is independent of \(m\) and \(n\), but may depend on \(d\) and \(k\).
	\autoref{lowerfory} and \autoref{lowerforz} provide matching lower bounds for the backward values and diagonal martingale integrands, respectively.
More generally, \autoref{sharpuperr} provides a transfer estimate
for the joint law of $(X,Y,Z)$, bounding the squared Wasserstein
error by the forward approximation error and an intrinsic backward
consistency error of order $m/n^2$.
	The result also controls the off-diagonal martingale integrands and quantifies the effect of approximate first-order cancellation.
	
	The first-order cancellation condition plays an important role in obtaining the sharp Wasserstein rate.
	Even when the forward particles are independent, dependence of the terminal condition on the empirical measure may induce fluctuations of order \(1/\sqrt{n}\) in the backward values, as illustrated in \autoref{ex}.
	Thus, smoothness alone does not imply the sharp Wasserstein rate.
	The cancellation condition eliminates the linear term in the expansion of the decoupling field around the limiting law.
	However, this condition is sufficient but not necessary for obtaining the sharp rate.
	\autoref{exasp4.2} shows that the cancellation condition
	can hold even in the presence of feedback through
	the laws of both backward components.

	\paragraph{Sharp weak estimates.}
	In \autoref{sec5}, we establish sharp weak propagation of chaos estimates without requiring the first-order cancellation condition.
	The diffusion coefficient \(\sigma(t,x,\mu)\) can depend on time, state and law, and may be degenerate.
	Under the stated regularity and forward weak-error assumptions, \autoref{5maintotal} yields, for every fixed \(m\),
	\begin{align*}
		&\sup_{t\in[0,T]}d_{3,1}^{Y,m}
		\big(\mathbb P_{(Y_t^{1,n},\ldots,Y_t^{m,n})},
		\mathbb P_{Y_t}^{\otimes m}\big)
		+\int_0^T d_{3,1}^{Z,m}
		\big(\mathbb P_{(Z_t^{1,1,n},\ldots,Z_t^{m,m,n})},
		\mathbb P_{Z_t}^{\otimes m}\big)dt
		\leq \frac{C_m}{n}.
	\end{align*}
	Here \(d_{3,1}^{Y,m}\) and \(d_{3,1}^{Z,m}\) are defined using the unit balls of the \(C_b^{3,1}\) test-function spaces introduced in \autoref{sec5}, and the constant \(C_m\)  depends on \(m\).
	The estimate is uniform in time for the backward values and integrated over time for the diagonal martingale integrands.
The result does not cover arbitrary path functionals or provide a weak error bound of order $n^{-1}$ for $Z$ at each fixed time.
	Finally, \autoref{nondege} gives matching lower bounds of order \(1/n\) for both components, while the corresponding Wasserstein errors are of order \(1/\sqrt{n}\).

	\paragraph{The intermediate-particle comparison method.}
	Both sharp results rely on the same intermediate-particle construction,  obtained by evaluating the limiting decoupling field along the empirical forward system:
	\begin{align*}
		\widetilde Y_t^i=U(t,X_t^{i,n},\theta_t^n).
	\end{align*}
The corresponding martingale integrands are given by
	\begin{align*}
		\widetilde Z_t^{i,j}
		=\mathbf1_{\{i=j\}}D_xU(t,X_t^{i,n},\theta_t^n)\sigma_t^i
		+\frac1n D_\mu U(t,X_t^{i,n},\theta_t^n)(X_t^{j,n})\sigma_t^j,
	\end{align*}
		where  \(\sigma_t^j=\sigma(t,X_t^{j,n},\theta_t^n)\).
	Under the corresponding assumptions, applying It\^o's formula to the empirical measure lift of \(U\) yields a consistency error of order \(n^{-1}\) in \(L^2(dt\otimes d\mathbb P)\).
	By  stability estimates of BSDEs, we obtain a mean-square estimate of order \(m/n^2\) between the interacting and intermediate particles; see \autoref{empirical} and \autoref{empirical2}.
	The measure derivative term in \(\widetilde Z\), including the off-diagonal components, is essential for controlling the full matrix of martingale integrands.

	For the sharp Wasserstein estimate, we replace \(\theta_t^n\) by the limiting law \(\theta_t\) in the decoupling field.
	The first-order cancellation condition, together with a fourth-moment bound for the empirical statistics, yields an additional mean-square error of order \(m/n^2\).
	We then transfer the sharp forward Wasserstein estimate of Arnese and Lacker \cite{man_sha_26} to the backward components through the decoupling field.

	For the weak estimate, \autoref{abstract-transfer} separates the backward consistency error from the forward weak approximation error.
	The latter is controlled by applying forward weak error estimates to test functionals defined through the decoupling field.
	Exchangeability and a correction for the difference between sampling with and without replacement then yield the estimate for each fixed \(m\).
	The required forward estimates can be obtained from the results of Chassagneux, Szpruch and Tse \cite{chassagen_22_aap} or Frikha and Song \cite{frikhaandsong_26} in the smooth setting.

	\subsection{Related literature}\label{sec2}
	For forward particle systems, synchronous coupling combined with estimates for empirical measures provides a standard approach to quantitative propagation of chaos (see \cite{sz_topic_book,malrieu_03,four_on_2015}).
Another approach to quantitative propagation of chaos is based on entropy methods.
For instance, Jabin and Wang \cite{jabin_wang_2018} derive relative entropy estimates for singular interaction kernels, while Bresch, Jabin and Wang \cite{bresch_jabin_wang_2023} develop a modulated free-energy approach for singular attractive interactions.
Using a relative entropy hierarchy, Lacker \cite{dan_hie_23} obtains sharp relative entropy bounds of order $m^2/n^2$ for $m$-particle marginals, assuming a transport-type inequality and suitable initial entropy bounds.
Lacker and Le Flem \cite{dan_sha_23} combine this hierarchy with uniform-in-time log-Sobolev inequalities to obtain bounds of the same order uniformly in time when the noise is sufficiently strong. Grass, Guillin and Poquet \cite{grass_sha_25} establish sharp entropy estimates for McKean--Vlasov SDEs with nonconstant diffusion coefficients under uniform ellipticity and a small-oscillation condition.
Arnese and Lacker \cite{man_sha_26} extend the entropy hierarchy approach to smooth drift coefficients with nonlinear dependence on the measure and obtain sharp Wasserstein estimates on path space over finite time intervals
(see also \cite{wang_sharp_26,monmarche_ren_wang_24,ren_wang_26}
for related sharp chaos estimates).
Moreover, 
%
	Mischler, Mouhot and Wennberg \cite{mischler_mouhot_wennberg_15} develop a general framework for quantitative propagation of chaos for drift, diffusion and jump processes.
A complementary PDE approach is developed by Chaudru de Raynal and Frikha \cite{CdRF21}, who use backward Kolmogorov equations on the Wasserstein space to establish quantitative propagation of chaos estimates for trajectories, semigroups, and densities of McKean--Vlasov SDEs.
They also obtain a first-order expansion for the one-particle density error.
Tse \cite{tse_regular_21} establishes higher-order regularity
of nonlinear Fokker--Planck flows with respect to the initial measure.
Meanwhile,  a related line of work concerns weak approximation of mean-field particle systems.
Chassagneux, Szpruch and Tse \cite{chassagen_22_aap} derive weak expansions using differential calculus on the space of probability measures.
Frikha and Song \cite{frikhaandsong_26} establish a weak error bound of order \(n^{-1}+\Delta t\) for Euler particle approximations in the smooth setting, where \(\Delta t\) denotes the time step.
Neither framework requires uniform ellipticity.
Related results include the analysis of particle flows by Haji-Ali, Hoel and Tempone \cite{haji_weak_25} and the uniform-in-time weak estimates on the torus obtained by Delarue and Tse \cite{del_tse_25}.

For backward systems, the classical theory of BSDEs was initiated by Pardoux and Peng \cite{pardoux_peng_90}. 
Since then, BSDEs have found a lot of applications in nonlinear partial differential equations (see Zhang \cite{zhang_bsde_17}), mathematical finance (see El Karoui, Peng, and Quenez  \cite{el_karoui_peng_quenez_97}), and stochastic optimal control (see Yong and Zhou \cite{yong-99}), to mention a few.
Meanwhile, mean-field models have found important applications in economics, finance, and game theory, particularly through the theory of mean field games developed by Lasry and Lions \cite{lasry_lions_07}. This naturally raises a fundamental question: whether mean-field BSDEs can be approximated by finite systems of backward equations.
For this problem, Buckdahn, Djehiche, Li, and Peng \cite{buck_mea_09} studied mean-field BSDEs through approximations based on independent copies and established a convergence rate of $1/\sqrt{n}$.
Lauri\`ere and Tangpi \cite{la_back_22} establish quantitative propagation of chaos estimates for interacting mean-field BSDEs, including a convergence rate of $1/\sqrt{n}$.
Hao, Hu, Tang, and Wen \cite{hao_qua_25} studied scalar quadratic mean-field BSDEs and obtained a convergence rate slower than $1/\sqrt{n}$.
It is worth noting that all the above results for this problem are based on coupling estimates and assume that the generator is independent of the  empirical measures of the martingale integrands.
Other related works consider mean-field BSDEs arising from investment games, propagation of chaos and stability for BSDEs driven by general square-integrable martingales, and conditional backward propagation of chaos in the presence of common noise, and optimal control problem,
we refer the readers to  \cite{amini_cao_sulem_25,tangpi_zhou_22,moreau_22,papapantoleon_24,papapantoleon_stability_25,buckdahn_li_peng_09,hernandez_tangpi_25}, and the references therein.

	The regularity theory for classical decoupling fields and master equations is developed in \cite{buckdahn_li_peng_rainer_17,li_18,chassagneux_crisan_delarue_22,carmona18_2}.
	Empirical evaluation of a classical solution on the Wasserstein
	space, combined with BSDE estimates, also appears in
	Germain, Pham and Warin \cite{germain_pham_warin_22}
	for particle approximations of PDEs
	(see also
	\cite{cardaliaguet_delarue_lasry_lions_19,delarue_lacker_ramanan_19,delarue_lacker_ramanan_20,cardaliaguet_jackson_mimikos_souganidis_26}
	for related convergence results in mean field games
	and control).
	Our results concern individual interacting backward particles
	whose generators depend on the empirical laws of both
	the backward values and the diagonal martingale integrands.
	The intermediate-particle comparison method
	 controls the full matrix
	of martingale integrands and separates a mean-square backward
	consistency error of order $m/n^2$ from the forward approximation
	error.
	This yields sharp  rates  in Wasserstein distance and in smooth weak metrics.

	 \paragraph{Organization of the paper.}
	 \autoref{section1.2} introduces the notation and measure derivatives used throughout the paper.
	 \autoref{sec3} establishes well-posedness and coupling estimates.
	 \autoref{sec4} and \autoref{sec5} establish sharp Wasserstein and smooth weak estimates, respectively.
\autoref{proofoflemma} contains the proofs of the regularity
results for the decoupling field and the time-dependent
McKean--Vlasov semigroup.

\section{Preliminaries}\label{section1.2}
We equip Euclidean spaces with their usual inner products and norms, and matrices with the Frobenius norm.
For a Polish metric space \((E,d_E)\), let \(\mathcal P(E)\) denote the set of Borel probability measures on \(E\), and \(\mathcal P_p(E)\) the subset of measures with finite \(p\)th moment, \(p\geq1\).
 For \(\mu,\nu\in\mathcal P_p(E)\), write \(\Pi(\mu,\nu)\) for their couplings and set
\begin{align*}
	\mathcal W_{p,d_E}(\mu,\nu)
	=\inf_{\pi\in\Pi(\mu,\nu)}
	\left(\int_{E\times E}d_E(x,y)^p\,\pi(dx,dy)\right)^{1/p}.
\end{align*}
We omit \(d_E\) for Euclidean spaces. For vector-valued continuous paths and square-integrable functions, respectively, the norms are
\begin{align*}
	\|x\|_\infty=\sup_{0\leq t\leq T}|x_t|\q\ \hbox{and} \q\ 
	\|z\|_{L^2}=\left(\int_0^T|z_t|^2\,dt\right)^{1/2}.
\end{align*}
For \(m\)-particle paths, \(|x_t|\) denotes the Euclidean norm across all particle coordinates, and the same convention applies to \(z_t\).
A law on \(E^m\) is exchangeable if it is invariant under coordinate permutations.

For the transport inequality used in \autoref{sec4}, we define the relative entropy by
\begin{align*}
	H(\mu\mid\nu)=
	\begin{cases}
		\displaystyle\int_E\log\frac{d\mu}{d\nu}\,d\mu,&\mu\ll\nu,\\
		+\infty,&\text{otherwise}.
	\end{cases}
\end{align*}
A measure \(\nu\in\mathcal P_1(E)\) satisfies \(T_1(C)\) if
\(\mathcal W_{1,d_E}^2(\mu,\nu)\leq2C H(\mu\mid\nu)\) for every
\(\mu\in\mathcal P_1(E)\); see Villani \cite{vill_op_09}.

For \(p\geq1\) and a finite-dimensional Euclidean space \(\mathbb H\), let \(L^p_{\mathscr F_T}(\Omega;\mathbb H)\) denote the space of \(\mathscr F_T\)-measurable random variables with finite \(p\)th moment.
We denote by \(S^p_{\mathbb F}([t,T];\mathbb H)\) and \(L^p_{\mathbb F}([t,T];\mathbb H)\) the spaces of continuous adapted processes and progressively measurable processes, respectively, equipped with the norms 
\begin{align*}
	\|V\|_{S^p_{\mathbb F}}
	&=\left(\mathbb E \left[\sup_{t\leq s\leq T}|V_s|^p \right]\right)^{1/p}\q\ 
	\hbox{and} \q\ 
	\|V\|_{L^p_{\mathbb F}}
	=\left[\mathbb E\left(\int_t^T|V_s|^2\,ds\right)^{p/2}\right]^{1/p}.
\end{align*}
The same notation applies to the product filtration. Constants denoted by \(C\) may change from line to line; their relevant dependencies are specified in each result.

We use linear functional and Lions derivatives as follows; see \cite{carmona18_1,chassagen_22_aap}. Vector- and matrix-valued derivatives are understood componentwise.
\begin{definition}[First-order measure derivatives]\label{def2,1}\rm
	A continuous map \(V:\mathcal P_2(\mathbb R^d)\to\mathbb R\) has a linear functional derivative if there exists a jointly continuous map \(\frac{\partial V}{\partial m}(\mu,v)\), with at most quadratic growth in \(v\), locally uniformly in \(\mu\), such that
	\begin{align*}
		V(\mu')-V(\mu)
		=\int_0^1\int_{\mathbb R^d}
		\frac{\partial V}{\partial m}((1-\lambda)\mu+\lambda\mu',v)
		\,(\mu'-\mu)(dv)\,d\lambda.
	\end{align*}
	We choose the normalization \(\frac{\partial V}{\partial m}(\mu,0)=0\).
	To define the Lions derivative, consider an auxiliary atomless probability space rich enough to realize every law in \(\mathcal P_2(\mathbb R^d)\), and define the lift
	\begin{align*}
		\widetilde V:L^2(\Omega;\mathbb R^d)\longrightarrow\mathbb R,
		\q\  \widetilde V(\xi)=V(\mathbb P_\xi).
	\end{align*}
	The map \(V\) is Lions differentiable at \(\mu\) if its lift is
	Fr\'echet differentiable at a random variable \(\xi\) with law \(\mu\).
	Its derivative is represented by a map \(D_\mu V(\mu,\cdot)\),
	defined \(\mu\)-almost everywhere, such that
	\begin{align*}
		D\widetilde V(\xi)[\eta]
		=\mathbb E\left[D_\mu V(\mu,\xi)\cdot\eta\right],
		\q\   \eta\in L^2(\Omega;\mathbb R^d).
	\end{align*}
	Whenever joint continuity is required, we work with the continuous versions specified in the corresponding assumptions.
	Under the spatial differentiability and regularity needed for the identification,
	\begin{align*}
		D_\mu V(\mu,v)=\nabla_v\frac{\partial V}{\partial m}(\mu,v).
	\end{align*}
\end{definition}
\begin{definition}[Second-order measure derivatives]\label{def2.2}\rm 
	For fixed \(v\), the second Lions derivative is the Lions derivative of the map \(\mu \longmapsto D_\mu V(\mu,v)\),
	taken componentwise:
	\begin{align*}
		D_\mu^2 V(\mu,v,v')
		=D_\mu \Big( D_\mu V(\cdot,v)\Big)(\mu,v').
	\end{align*}
	When the indicated derivatives exist and differentiation may be interchanged,
	\begin{align*}
		D_\mu^2V(\mu,v,v')
		=\nabla_{v'}\nabla_v\frac{\partial^2V}{\partial m^2}(\mu,v,v').
	\end{align*}
	The derivative \(D_vD_\mu V\) differentiates the spatial argument of the first Lions derivative and is distinct from \(D_\mu^2V\).
\end{definition}
\begin{remark}\label{rm}\rm
	Higher-order linear functional derivatives are defined recursively, with auxiliary spatial variables held fixed. Under the corresponding regularity and interchange conditions,
	\begin{align*}
		D_\mu^jV(\mu,v_1,\ldots,v_j)
		=\nabla_{v_j}\cdots\nabla_{v_1}
		\frac{\partial^jV}{\partial m^j}(\mu,v_1,\ldots,v_j).
	\end{align*}
	The regularity and bounds required for these derivatives are specified in \autoref{sec4} and \autoref{sec5}.
\end{remark}

	\section{Well-posedness and coupling estimates}\label{sec3}
	To establish the well-posedness of the limiting
		 FBSDE system  \eqref{x1n}--\eqref{y1n},
	we make use of the following assumption:
	
	\begin{assumption}\label{lipp}\sl 
		There exist   constants \(p\in \left(1,2\right]\), \(q\geq p\), \(\beta \in \mathbb{R}\),    \(K_q\geq 0,\)
		\(L_b,L_\sigma, L_h, L_f\geq 0,\) and \(\gamma =\max\{L_b,L_\sigma,L_h,L_f\}\) such that,   for
		any \((t,x,y,z,\theta,\mu,\nu,
		x',y',z',\theta',\mu',\nu')\in [0,T]\times \Big(\mathbb{R}^d \times \mathbb{R}^k \times \mathbb{R}^{k\times d} \times \mathcal{P}_p(\mathbb{R}^d) \times
		\mathcal{P}_p(\mathbb{R}^k) \times \mathcal{P}_p(\mathbb{R}^{k\times d}) \Big)^2\),  the following hold. 
		The continuity, monotonicity and Lipschitz conditions
		below hold outside a common \(dt\otimes d\mathbb P\)-null set,
		for all displayed spatial and measure arguments; the terminal Lipschitz
		condition holds outside a common \(\mathbb P\)-null set.
		\begin{enumerate}[(i)]
		\item 
	\begin{align*}
		 	K_q=&\ \mathbb{E}\left[|X_0|^q
		 +|h(0,\delta_0)|^q+\Bigg(\int_0^T|b(t,0,\delta_0)|dt\Bigg)^q
		 +\Bigg(\int_0^T|\sigma(t,0,\delta_0)|^2dt\Bigg)^{\frac{q}{2}}\right.\\
	&\left.	 +\Bigg(\int_0^T|f(t,0,\delta_0,0,\delta_0,0,\delta_0)|dt\Bigg)^q \right]<\infty.
	\end{align*}
	\item 
	\( y\longmapsto f(t,x,\theta,y,\mu,z,\nu) \) is continuous.
	\item 
	For every \(M>0,\)
	\begin{align*}
	  \mathbb{E}\int_0^T \Big( 
	 \sup\limits_{|y|\leq M} \big|f(t,0,\delta_0,y,\delta_0,0,\delta_0)-f(t,0,\delta_0,0,\delta_0,0,\delta_0)\big| \Big) dt <\infty.
	\end{align*}
	\item 
	\(	\Big\langle y-y', \big(  f(t,x,\theta,y,\mu,z,\nu)-f(t,x,\theta,y',\mu,z,\nu)\big) \Big\rangle
	\leq \beta |y-y'|^2. \)
	\item 
	\(  |h(x,\theta)-h(x',\theta')|\leq 
	L_h \Big(|x-x'|+\mathcal{W}_p(\theta,\theta')\Big).\)
	\item
	\(   |b(t,x,\theta)-b(t,x',\theta')| 
	\leq L_b \Big(|x-x'|+\mathcal{W}_p(\theta,\theta') \Big).\)
	\item 
	\(  |\sigma(t,x,\theta)-\sigma(t,x',\theta')| 
	\leq L_\sigma \Big(|x-x'|+\mathcal{W}_p(\theta,\theta') \Big).\)
	\item 
	\( 	|f(t,x,\theta,y,\mu,z,\nu)-f(t,x',\theta',y,\mu',z',\nu')|
 	 \leq L_f \Big(|x-x'|+\mathcal{W}_p(\theta,\theta')+\mathcal{W}_p(\mu,\mu')+|z-z'|+\mathcal{W}_p(\nu,\nu')\Big).\)
		\end{enumerate}
	\end{assumption}
	
		\begin{proposition}\label{exanduni}\sl 
		Assume that \autoref{lipp} holds. Then
		the limiting FBSDE system \eqref{x1n}--\eqref{y1n} admits  
		a unique solution  \((X,Y,Z)\in 
		S_{\mathbb{F}}^{p}([0,T];\mathbb{R}^d)\times 	S_{\mathbb{F}}^{p}([0,T];\mathbb{R}^k)\times 	L_{\mathbb{F}}^{p}([0,T];\mathbb{R}^{k\times d})\).
		If   \(q=2,\) then   \((X,Y,Z)\in 
		S_{\mathbb{F}}^{2}([0,T];\mathbb{R}^d)\times 	S_{\mathbb{F}}^{2}([0,T];\mathbb{R}^k)\times 	L_{\mathbb{F}}^{2}([0,T];\mathbb{R}^{k\times d})\).
	\end{proposition}
	\begin{proof}
			A Picard iteration, together with the Burkholder--Davis--Gundy inequality and Gronwall's lemma, gives a unique solution
		\(X\in S^p_{\mathbb{F}}([0,T];\mathbb{R}^d)\).
		Fix  \(( U,V)\in  	S_{\mathbb{F}}^{p}([0,T];\mathbb{R}^k)\times 	L_{\mathbb{F}}^{p}([0,T];\mathbb{R}^{k\times d})\),
		 and define
		\begin{align*}
		 f^{U,V,X}(t,y,z)&=f(t,X_t,\mathbb{P}_{X_t},y,\mathbb{P}_{U_t},z,\mathbb{P}_{V_t})\q\ \hbox{and} \q\ 
		 \psi_M(t)=\sup\limits_{|y|\leq M}|f^{U,V,X}(t,y,0)-f^{U,V,X}(t,0,0)|.
		\end{align*}
	By \autoref{lipp},
		\begin{align*}
|f^{U,V,X}(t,0,0)|&\leq
|f(t,0,\delta_0,0,\delta_0,0,\delta_0)|
+L_f\Big(|X_t| +\mathcal{W}_p(\mathbb{P}_{X_t},\delta_0)
+\mathcal{W}_p(\mathbb{P}_{U_t},\delta_0)
+\mathcal{W}_p(\mathbb{P}_{V_t},\delta_0)\Big),\\
|h(X_T,\mathbb{P}_{X_T})|
&\leq
|h(0,\delta_0)|
+L_h\Big(|X_T| +\mathcal{W}_p(\mathbb{P}_{X_T},\delta_0)\Big),\\
\psi_M(t)
&\leq \sup\limits_{|y|\leq M}
|f(t,0,\delta_0,y,\delta_0,0,\delta_0)-f(t,0,\delta_0,0,\delta_0,0,\delta_0)|\\
&\ +2L_f\Big(|X_t| +\mathcal{W}_p(\mathbb{P}_{X_t},\delta_0)
+\mathcal{W}_p(\mathbb{P}_{U_t},\delta_0)
+\mathcal{W}_p(\mathbb{P}_{V_t},\delta_0)\Big).
		\end{align*}
			Since \(p\leq2\), the Cauchy--Schwarz and   Minkowski inequalities give
		\begin{align*}
	 \Bigg(\int_0^T \mathcal{W}_p(\mathbb{P}_{V_t},\delta_0)dt\Bigg)^p
	 \leq T^{\frac{p}{2}} \mathbb{E}\left[
	 \Bigg(\int_0^T |V_t|^2dt
	 \Bigg)^{\frac{p}{2}}\right].
		\end{align*}
The terms involving \(U\) and \(X\) are controlled by their \(S^p_{\mathbb F}\)-norms.
 These estimates, together with \autoref{lipp}, imply that conditions \textbf{(H1)}--\textbf{(H5)} in \cite{bran_lp_03} hold.
By Briand et al. \cite[Theorem 4.2]{bran_lp_03},   there exists a unique solution \((Y,Z)\in  	S_{\mathbb{F}}^{p}([0,T];\mathbb{R}^k)\times 	L_{\mathbb{F}}^{p}([0,T];\mathbb{R}^{k\times d})\) satisfying
		\begin{align}\label{uvx}
			&Y_t=h(X_T,\mathbb{P}_{X_T})+\int_t^T f^{U,V,X}(s, Y_s,  Z_s)ds
			-\int_t^T Z_sdW_s,\q\ 0\leq t\leq T.
		\end{align}
		This defines a map
		 \( 
		\Phi(U,V)=(Y,Z)\)  on
		\(  S_{\mathbb{F}}^{p}([0,T];\mathbb{R}^{k})\times 	L_{\mathbb{F}}^{p}([0,T];\mathbb{R}^{k\times d})\). 
		For \(0<\delta\leq T\), 
		consider two pairs \((U,V)\) and \((U',V')\) on \([T-\delta,T]\), 
		and denote the corresponding solutions by \((Y,Z)\) and \((Y',Z')\).
		By  the   \(L^p\)-stability estimate in
		 Briand et al. \cite[Proposition 3.2]{bran_lp_03}, 
	we obtain
		\begin{align*}
			\mathbb{E}\left[ \sup\limits_{t\in [T-\delta,T]}|Y_t-Y_t'|^p +\Bigg(
			\int_{T-\delta}^T |Z_s-Z_s'|^2ds\Bigg)^{\frac{p}{2}}\right]
			&\leq  C
			\mathbb{E}\left[  \Bigg(
			\int_{T-\delta}^T \Big( \mathcal{W}_p(\mathbb{P}_{U_s},\mathbb{P}_{U_s'})
			+\mathcal{W}_p(\mathbb{P}_{V_s},\mathbb{P}_{V_s'})\Big)
			ds\Bigg)^p\right],
		\end{align*}
		where \(C>0\) depends only on \(p,\beta,L_f\) and \(T\).
		Using \(p\leq2\), the Cauchy--Schwarz and Minkowski inequalities give
		\begin{align*}
			&\ \mathbb{E}\left[ \sup\limits_{t\in [T-\delta,T]}|Y_t-Y_t'|^p +\Bigg(
			\int_{T-\delta}^T |Z_s-Z_s'|^2ds\Bigg)^{\frac{p}{2}}\right]\\
			&\leq  C \Bigg\{ 
			\delta^{p}  \sup\limits_{t\in [T-\delta,T]} \mathcal{W}^p_p(\mathbb{P}_{U_t},\mathbb{P}_{U_t'})    
			+\delta^{\frac{p}{2}}\mathbb{E}\left[   \Bigg(
			\int_{T-\delta}^T  \mathcal{W}_p^2(\mathbb{P}_{V_s},\mathbb{P}_{V_s'})   ds\Bigg)^{\frac{p}{2}}\right]\Bigg\}\\
				&\leq  C \Bigg\{ 
			\delta^{p}  \sup\limits_{t\in [T-\delta,T]} \mathcal{W}^p_p(\mathbb{P}_{U_t},\mathbb{P}_{U_t'})    
			+\delta^{\frac{p}{2}}   \Bigg(
			\int_{T-\delta}^T \Big(   \mathbb{E}\left[   |V_s-V_s'|^p
			\right]\Big)^{\frac{2}{p}} ds  \Bigg)^{\frac{p}{2}}   \Bigg\}\\
			&\leq C \mathbb{E}\left[ \delta^{p}  \sup\limits_{t\in [T-\delta,T]}|U_t-U_t'|^p +\delta^{\frac{p}{2}}\Bigg(
			\int_{T-\delta}^T |V_s-V_s'|^2ds\Bigg)^{\frac{p}{2}}\right].
		\end{align*}
		Choose \(\delta>0\)  such that \(C(\delta^p+\delta^{\frac{p}{2}})< 1\).
		 Then \(\Phi\) is a strict contraction on \(	S_{\mathbb{F}}^{p}([T-\delta,T];\mathbb{R}^k)\times 	L_{\mathbb{F}}^{p}([T-\delta,T];\mathbb{R}^{k\times d}) \).  
		To extend the construction to the whole interval \([0,T]\), choose a partition such that
		\begin{align*}
			0=t_0<t_1<\dots <t_N=T\q\ \hbox{and} \q\
			\max_{l} (t_{l+1}-t_l)  \leq \delta.
		\end{align*}
	Starting from \([t_{N-1},T]\), we repeat this argument backward on each interval \([t_l,t_{l+1}]\), using the previously obtained value \(Y_{t_{l+1}}\) as the terminal condition. This gives a solution on \([0,T]\). 
	Uniqueness follows from uniqueness on each subinterval.
	If \(q=2\), the same argument with exponent \(2\), using \(\mathcal W_p\leq\mathcal W_2\), gives a square-integrable solution. 
	By uniqueness, it coincides with the solution constructed above.
	\end{proof}
	
	Next, we work on the product probability space
	\((\Omega^n,\mathscr{F}^n,\mathbb{F}^n,\mathbb{P}^n)\),
	where \(\Omega^n=\prod_{i=1}^n\Omega \), \(\mathscr{F}^n=\overline{\bigotimes_{i=1}^n  \mathscr{F}}^{\mathbb{P}^n}\), \(\mathbb{F}^n=(\mathscr{F}_t^n)_{t\in [0,T]}\), and \(\mathbb{P}^n= \mathbb{P}^{\otimes n}\).
We also write \(\mathbb{E}\) for expectation under \(\mathbb{P}^n\).
Let \((\bar{X}^i,\bar{Y}^i,\bar{Z}^i)_{i=1}^n\) be i.i.d. copies
of the solution \((X,Y,Z)\) to \eqref{x1n}--\eqref{y1n}.
These processes satisfy
	\begin{align}
		&	\bar{X}_t^i=\bar{X}_0^i+\int_0^t b(s,\bar{X}_s^i,\mathbb{P}_{\bar{X}_s^i})ds+
	\int_0^t 	\sigma(s,\bar{X}_s^i,\mathbb{P}_{\bar{X}_s^i}) dW^i_s,     \label{xi}\\
		&	\bar{Y}_t^i=h(\bar{X}^i_T,\mathbb{P}_{\bar{X}_T^i})+\int_t^T f(s,\bar{X}_s^i,\mathbb{P}_{\bar{X}_s^i},\bar{Y}^i_s,\mathbb{P}_{\bar{Y}^i_s}, \bar{Z}^i_s,\mathbb{P}_{\bar{Z}^i_s})ds -\int_t^T \bar{Z}^i_sdW^i_s,      \label{yzi}
	\end{align}
	where \(\bar{X}_0^i=X_0^i\). 
	Define	\begin{equation*} 
		\bar{Z}_t^{i,j}=\left\{
		\begin{aligned}
			\bar{Z}^i_t,  \q\ &\hbox{if} ~j=i,\\
			0,    \q\ &\hbox{if}~  j\neq i.
		\end{aligned}
		\right.
	\end{equation*}
In the \(i\)-th backward equation, the coefficients \(h\) and \(f\)
are evaluated at \(\omega^i\).
	More precisely, for \(\overrightarrow{\omega}=(\omega^1,\cdots,\omega^n)\in \Omega^n\), the \(i\)-th component of the finite-dimensional generator is given by
	\begin{align*}
	F^n_i(t,\overrightarrow{\omega},\overrightarrow{y},\overrightarrow{z})
	=f(t,\omega^i,X_t^i,\theta_t^n,y^i,\mathcal{L}^n(\overrightarrow{y}),z^{i,i},\mathcal{L}^n(\overrightarrow{z})),
	\end{align*}
	where
	\begin{align*}
	\mathcal{L}^n(\overrightarrow{y})=\frac{1}{n}\sum_{j=1}^{n}\delta_{y^j}\q\
	\hbox{and} \q\ 
	\mathcal{L}^n(\overrightarrow{z})=\frac{1}{n}\sum_{j=1}^{n}\delta_{z^{j,j}}.
	\end{align*}
		The data for individual particles are i.i.d., the coefficients are defined coordinatewise, and the interactions are invariant under permutations.
By uniqueness of the particle system, established below, the coupled family
	\begin{align*}
	 (X^i,\bar{X}^i,Y^i,\bar{Y}^i,(Z^{i,j})_{j=1}^n,\bar{Z}^i)_{i=1}^n
	\end{align*}
	is invariant in law under \(i\longmapsto\pi(i)\) and
\((i,j)\longmapsto(\pi(i),\pi(j))\), for every permutation \(\pi\) of \(\{1,\ldots,n\}\).
In particular, the diagonal processes and the sums of squared errors
in each row are exchangeable across \(i\).
We next prove convergence of the particle system by coupling.

	\begin{theorem}\label{xyzcon}\sl 
		Assume that \autoref{lipp} holds with \(p\in (1,2)\), \(q=2\), and that
		\(\sigma\) is independent of \(\theta\).  
		Then, for every \(n\geq 1,\)   \(i\in \{1,\cdots,n\}\), and
	 	\(r\in (p,2]\), we have
		\begin{align}\label{converyz}
			\mathbb{E}\left[ \sup\limits_{t\in [0,T]}
			\Big( |X_t^i-\bar{X}_t^i|^p+ |Y^i_t-\bar{Y}^i_t|^p\Big)
			+\Bigg( \int_0^T \Big(  |Z_t^{i,i}-\bar{Z}^i_t|^2
			+\sum_{j\neq i} |Z_t^{i,j}|^2\Big)dt \Bigg)^{\frac{p}{2}} 
			\right]\leq C\bar{\lambda}_{p,r,n},
		\end{align}
		where 	  \(r\neq \frac{sp}{s-p}\) for every
		\(s\in\{d, k, kd\} \) with \(p<\frac{s}{2}\),
		 \(\bar{\lambda}_{p,r,n}=\lambda_{p,r,n}(d)+\lambda_{p,r,n}(k)+\lambda_{p,r,n}(kd)\),
		 \(C\) depends only on \(p,k,d,r, \beta, \gamma,   K_2, T \),  
		 but is independent of \(n, i\),
		and \(\lambda_{p,r,n}(s)\) is defined by
		\begin{equation}\label{lamda}
			\lambda_{p,r,n}(s)=\left\{
			\begin{aligned}
				n^{-\frac{1}{2}}+n^{-\frac{r-p}{r}}, \qquad\qquad  \q\ &\hbox{if} ~p>\frac{s}{2}; \\
				n^{-\frac{1}{2}}
				\log(1+n)+n^{-\frac{r-p}{r}},\q
				& \hbox{if} ~p=\frac{s}{2};\\
				n^{-\frac{p}{s}}+n^{-\frac{r-p}{r}},\q\ \qquad \qquad &\hbox{if} ~
				p< \frac{s}{2}.
			\end{aligned}
			\right.
		\end{equation}
	\end{theorem}
	\begin{proof}
	Throughout the proof, \(C>0\) denotes a constant that may vary
	from line to line and is independent of \(n\) and \(i\).
		Let
		\begin{align*}
			&\qquad \q \theta_t\triangleq \mathbb{P}_{X_t},\q\ 
			\mu_t\triangleq \mathbb{P}_{Y_t},\q\
			\nu_t\triangleq \mathbb{P}_{Z_t},\\
			&
			\bar{\theta}_t^n\triangleq
			\frac{1}{n}\sum_{i=1}^{n} \delta_{\bar{X}_t^i},\q\ 
			\bar{\mu}^n_t\triangleq\frac{1}{n}\sum\limits^n_{i=1}\delta_{\bar{Y}_t^i},\q\
			\bar{\nu}^n_t\triangleq \frac{1}{n}\sum\limits^n_{i=1}\delta_{\bar{Z}_t^{i}}.
		\end{align*}
		By \autoref{exanduni}, the limiting FBSDE system \eqref{x1n}--\eqref{y1n} admits a unique solution \((X,Y,Z)\in 
		S_{\mathbb{F}}^{2}([0,T];\mathbb{R}^d)\times 	S_{\mathbb{F}}^{2}([0,T];\mathbb{R}^k)\times 	L_{\mathbb{F}}^{2}([0,T];\mathbb{R}^{k\times d})\). 
The forward particle system \eqref{x1} is a finite-dimensional SDE
with Lipschitz coefficients. 
It therefore admits a unique solution
\((X^1,\ldots,X^n)\), with
\(X^i\in S_{\mathbb{F}^n}^{2}([0,T];\mathbb{R}^d)\)
for every \(i\). 
We next consider the backward particle system.
Since \(p\in(1,2)\), 
		\begin{align*}
\mathcal{W}_p\Big(\frac{1}{n}\sum_{i=1}^n \delta_{y_i},\frac{1}{n}\sum_{i=1}^n\delta_{y_i'}\Big)
\leq \Big(\frac{1}{n}\sum_{i=1}^n|y_i-y_i'|^p  \Big)^{\frac{1}{p}}
\leq \Big(\frac{1}{n}\sum_{i=1}^n|y_i-y_i'|^2  \Big)^{\frac{1}{2}}
=\frac{1}{\sqrt{n}}\Big( \sum_{i=1}^n|y_i-y_i'|^2  \Big)^{\frac{1}{2}}.
		\end{align*}
For any \(\overrightarrow{y}=(y^1,\cdots,y^n)\in \mathbb{R}^{nk}\),
\(\overrightarrow{z}=(z^1,\cdots,z^n)\in \mathbb{R}^{nk\times nd}\),
let
		\begin{align*}
 F^n_i(t,\overrightarrow{\omega},\overrightarrow{y},\overrightarrow{z})
 =f(t,\omega^i,X_t^i,\theta_t^n,y^i,\mathcal{L}^n(\overrightarrow{y}),z^{i,i},\mathcal{L}^n(\overrightarrow{z}))\q\ 
 \hbox{and} \q\ 
 G_i^n(\overrightarrow{\omega})=h(\omega^i,X_T^i(\overrightarrow{\omega}),\theta_T^n(\overrightarrow{\omega})).
		\end{align*}
		In view of \autoref{lipp},
		\begin{align*}
	\sum_{i=1}^{n}	\Big\langle y^i-(y^i)', \big(  F_i^n(t,\overrightarrow{y},\overrightarrow{z})-F_i^n(t,\overrightarrow{y'},\overrightarrow{z}) \Big\rangle
&	\leq 
	 \beta 	\sum_{i=1}^{n} |y^i-(y^i)'|^2 +L_f \mathcal{W}_p(\mathcal{L}^n(\overrightarrow{y}),\mathcal{L}^n(\overrightarrow{y'})) \sum_{i=1}^{n}|y^i-(y^i)'|\\
	 &\leq ( \beta +L_f)	\sum_{i=1}^{n} |y^i-(y^i)'|^2;\\
	\sum_{i=1}^{n} |   	F_i^n(t,\overrightarrow{y},\overrightarrow{z})-F_i^n(t,\overrightarrow{y},\overrightarrow{z'}) |^2
&	\leq 2L_f^2 \Big( 	\sum_{i=1}^{n}|z^{i,i}-(z^{i,i})'|^2+n\mathcal{W}_p^2(\mathcal{L}^n(\overrightarrow{z}),\mathcal{L}^n(\overrightarrow{z'}))  \Big)	 \\
	&\leq  4L_f^2 \sum_{i=1}^{n} |z^{i,i}-(z^{i,i})'|^2
	\\ 
	&\leq 4L_f^2 |\overrightarrow{z}-\overrightarrow{z'}|^2.
		\end{align*}
		Moreover, \(\overrightarrow{y}\longmapsto F_i^n(t,\overrightarrow{y},\overrightarrow{z})\) is continuous.
		For 
		\(	G^n=(G_1^n,\cdots,G_n^n)\) and
		\(F^n=(F_1^n,\cdots,F_n^n)\), we have
		\begin{align*}
	G^n\in L^2_{\mathscr{F}_T^n}(\Omega^n;\mathbb{R}^{nk})\q\ 
	\hbox{and}\q\ 
	\mathbb{E}\left[
	\Bigg(\int_0^T |F^n(t,0,0)|dt\Bigg)^2\right]<\infty.
		\end{align*}
		For every \(\bar{M}>0,\)
		\begin{align*}
\psi^n_{\bar{M}}(t)
= \sup\limits_{|\overrightarrow{y}|\leq \bar{M}}
|F^n(t,\overrightarrow{y},0)-F^n(t,0,0)|\in L^1(dt\otimes d\mathbb{P}^n).
		\end{align*}
		The backward particle system \eqref{y1} is an \(\mathbb{R}^{nk}\)-valued BSDE driven by the
		\(nd\)-dimensional Brownian motion \((W^1,\cdots,W^n)\).
		The preceding estimates show that its 
		 generator is monotone in \(y\) and Lipschitz continuous in \(z\),
		with constants independent of \(n\).
	By Briand et al.  \cite[Theorem 4.2]{bran_lp_03}, applied with exponent \(2\),
	the backward particle system \eqref{y1} admits a unique solution
		\((\overrightarrow{Y},\overrightarrow{Z}) \in 	S_{\mathbb{F}^n}^{2}([0,T];\mathbb{R}^{nk})\times 	L_{\mathbb{F}^n}^{2}([0,T];\mathbb{R}^{ nk \times nd}) \), where \(\overrightarrow{Y}=(Y^1,\cdots,Y^n) \), 
		\(\overrightarrow{Z}=(Z^1,\cdots,Z^n) \).

		We next estimate the difference between
		\(X^i\) and \(\bar X^i.\)
		For any \(\varepsilon>0,\) let \(\psi_\varepsilon(x)=\big(|x|^2+\varepsilon \big)^{\frac{p}{2}}\). Applying It\^o's formula to \(\psi_\varepsilon(X_t^i-\bar{X}_t^i)\)
		gives
		\begin{align*}
   \psi_\varepsilon(X_t^i-\bar{X}_t^i)
 \leq&\  \varepsilon^\frac{p}{2}
 + \int_0^t  \psi_\varepsilon^{\frac{p-2}{p}}(X_s^i-\bar{X}_s^i)
 \Big( p |X_s^i-\bar{X}_s^i| |b(s,X_s^i,\theta_s^n)
 -b(s,\bar{X}_s^i, \theta_s)| \\
 & +\frac{p}{2}|\sigma(s,X_s^i )
 -\sigma(s,\bar{X}_s^i )|^2  \Big)ds
 +p\int_0^t  \psi_\varepsilon^{\frac{p-2}{p}}(X_s^i-\bar{X}_s^i)
  \langle  X_s^i-\bar{X}_s^i, \big(\sigma(s,X_s^i )
 -\sigma(s,\bar{X}_s^i)\big) dW_s^i\rangle,
		\end{align*}
		where we have used \(p<2\)
		to discard the nonpositive Hessian term.
	We first localize the stochastic integral and then pass to the limit
	using the estimates below.
	By the Lipschitz continuity of \(b\) and \( \sigma\),  the
		Burkholder--Davis--Gundy inequality,
		 and Young's inequality,  
		 we obtain
		 \begin{align*}
		  \mathbb{E}\left[\sup\limits_{u\in [0,t]}
		  \psi_\varepsilon(X_u^i-\bar{X}_u^i)\right]
		  \leq \varepsilon^\frac{p}{2}+
		   C \mathbb{E}  
		  \int_0^t  \Big( \psi_\varepsilon(X_s^i-\bar{X}_s^i)+  \mathcal{W}_p^p(\theta_s^n,\theta_s)\Big)
		  ds.
		 \end{align*}
		 Letting \(\varepsilon \longrightarrow 0,\) we obtain
		\begin{align}\label{xxi}
		 \mathbb{E}\left[\sup\limits_{s\in [0,t]}
		 |X^i_s-\bar{X}^i_s|^p
		 \right]\leq
		 C \mathbb{E}  
		 \int_0^t  \Big( \sup\limits_{r\in [0,s]} |X^i_r-\bar{X}^i_r|^p+  \mathcal{W}_p^p(\theta_s^n,\theta_s)\Big)
		 ds,\q\ t\in [0,T].
		\end{align}
			By the triangle inequality, 
			\begin{align*}
 \mathcal{W}^p_p(\theta_t^n,\theta_t)
 \leq 2^p \Big(	\mathcal{W}^p_p(\theta_t^n,\bar{\theta}^n_t)+
 \mathcal{W}^p_p(\bar{\theta}_t^n,\theta_t)
 \Big),
			\end{align*}
where
			\begin{align*}
 	\mathcal{W}^p_p(\theta_t^n,\bar{\theta}^n_t)
 	\leq 
 	\frac{1}{n}\sum_{i=1}^{n}|X_t^i-\bar{X}_t^i|^p.
			\end{align*}
		By   Fournier and Guillin \cite[Theorem 1]{four_on_2015}, 
		\begin{align*}
	\mathbb{E}\left[ 	\mathcal{W}^p_p(\bar{\theta}_t^n,\theta_t)  \right]
	\leq 	 C \Big(\mathbb{E}\left[|X_t|^r\right] \Big)^{\frac{p}{r}}
	\lambda_{p,r,n}(d).
		\end{align*}
		Since \(X\in S^2_{\mathbb{F}}\) and \(r\leq2\),
	\(  \sup\limits_{t\in [0,T]}\Big(\mathbb{E}\left[|X_t|^r\right] \Big)^{\frac{p}{r}}<\infty. \)
The same estimate applies to \(r=2\) under the stated restrictions.
	Averaging \eqref{xxi} over \(i\) and using the preceding bounds, we obtain
	\begin{align*}
			 \frac{1}{n}\sum_{i=1}^{n} \mathbb{E}\left[\sup\limits_{s\in [0,t]}
	|X^i_s-\bar{X}^i_s|^p
	\right]\leq
		 C\int_0^t 
		  \frac{1}{n}\sum_{i=1}^{n} \mathbb{E}\left[\sup\limits_{r\in [0,s]}
		 |X^i_r-\bar{X}^i_r|^p	 \right]ds
		 +C\lambda_{p,r,n}(d), \q\ t\in [0,T].
	\end{align*}
	 Gronwall's inequality gives
			\begin{align}\label{xestimat_1}
		 \frac{1}{n}\sum_{i=1}^{n}\mathbb{E}\left[
		 \sup\limits_{t\in [0,T]}|X_t^i-\bar{X}_t^i|^p\right]
		 \leq 
		 C \lambda_{p,r,n}(d).
			\end{align}
			By exchangeability, \eqref{xestimat_1} implies, for every \(i,\)
			\begin{align*}
		 \mathbb{E}\left[
		 \sup\limits_{t\in [0,T]}|X_t^i-\bar{X}_t^i|^p\right]
		 \leq 
		 C \lambda_{p,r,n}(d).
			\end{align*}
				Let \(\delta>0\) be   chosen later.
		By  Briand et al. \cite[Proposition 3.2]{bran_lp_03},  
		 \begin{equation}\label{dif}
		\begin{aligned}
			&\ \mathbb{E}\left[ \sup\limits_{t\in [T-\delta,T]}
		  |Y^i_t-\bar{Y}^i_t|^p
			+\Bigg( \int_{T-\delta}^T\Big(  |Z_t^{i,i}-\bar{Z}^i_t|^2
			+\sum_{j\neq i} |Z_t^{i,j}|^2\Big)
			dt \Bigg)^{\frac{p}{2}} 
			\right]\\
			&\leq C
			\mathbb{E}\left[ \sup\limits_{t\in [T-\delta,T]}|X_t^i-\bar{X}^i_t|^p+
			\mathcal{W}^p_p(\theta_T^n, \theta_T)+
			\Bigg( \int_{T-\delta}^T \Big( \mathcal{W}_p(\theta_t^n, \theta_t)+ \mathcal{W}_p(\mu_t^n,\mu_t)
			+\mathcal{W}_p(\nu^n_t, \nu_t) \Big) dt\Bigg)^p \right],
		\end{aligned}
		\end{equation}
		where
		\begin{align*}
		 \mathbb{E}\left[\mathcal{W}_p^p(\theta_T^n,\theta_T)\right]
		 \leq C\lambda_{p,r,n}(d) \q\ \hbox{and} \q\ 
		 \mathbb{E}\left[\Bigg(\int_{T-\delta}^T\mathcal{W}_p(\theta_t^n,\theta_t)dt
		 \Bigg)^p\right]
		 \leq C \delta^p
		 \lambda_{p,r,n}(d).
		\end{align*}
		By the triangle inequality,
		\begin{align*}
		&\  
		\mathcal{W}^p_p(\mu_t^n,\mu_t)
		+\mathcal{W}^p_p(\nu^n_t, \nu_t)
	 	\leq 2^p \Big( 
			\mathcal{W}^p_p(\mu_t^n,\bar{\mu}^n_t)
			+\mathcal{W}^p_p(\nu^n_t, \bar{\nu}^n_t)
		 +\mathcal{W}^p_p(\bar{\mu}^n_t,\mu_t)
			+\mathcal{W}^p_p(\bar{\nu}^n_t, \nu_t)\Big).
		\end{align*}
	  H\"older's inequality gives
		\begin{align*}
			\mathbb{E}\left[  
			\Bigg( \int_{T-\delta}^T \mathcal{W}_p(\mu^n_t, \bar{\mu}^n_t) dt\Bigg)^p\right]
			&\leq \delta^{p-1}\int_{T-\delta}^T\mathbb{E}\left[\mathcal{W}^p_p(\mu^n_t, \bar{\mu}^n_t)\right]dt
			 \leq \delta^p\frac{1}{n}\sum_{i=1}^{n}
			\mathbb{E}\left[\sup\limits_{t\in [T-\delta,T]}
			|Y^i_t-\bar{Y}^i_t|^p 
		 \right].
		\end{align*} 
	Similarly,
		\begin{align*}
		\mathbb{E}\left[  
		\Bigg( \int_{T-\delta}^T \mathcal{W}_p(\nu^n_t, \bar{\nu}^n_t) dt\Bigg)^p\right]
			& \leq \delta^{p-1}\frac{1}{n}\sum_{i=1}^{n}
			\mathbb{E} 
			\int_{T-\delta}^T |Z_t^{i,i}-\bar{Z}_t^i |^p dt 
			\leq \delta^{\frac{p}{2}}\frac{1}{n}\sum_{i=1}^{n}  \mathbb{E}\left[ \Bigg( \int_{T-\delta}^T |Z_t^{i,i}-\bar{Z}^i_t|^2dt \Bigg)^{\frac{p}{2}}
			\right].
		\end{align*} 
	By    Fournier and Guillin \cite[Theorem 1]{four_on_2015}, together with H\"older's inequality,
		\begin{align*}
				\mathbb{E}\left[
		\Big(\int_{T-\delta}^{T} \mathcal{W}_p(\bar{\mu}_t^n,\mu_t)dt
		\Big)^p\right]
		&\leq \delta^{p-1} 
		\int_{T-\delta}^T\mathbb{E}\left[   \mathcal{W}_p^p(\bar{\mu}_t^n,\mu_t) \right]dt\\
		&\leq C \delta^{p-1} \lambda_{p,r,n}(k)
			 \int_{T-\delta}^T \Big(\mathbb{E}\left[|Y_t|^r\right] \Big)^{\frac{p}{r}} dt\\
			 &\leq C\delta^{p}\lambda_{p,r,n}(k)\Bigg(\mathbb{E}\left[
			 \sup\limits_{t\in [T-\delta,T]}|Y_t|^2\right] \Bigg)^{\frac{p}{2}}.
		\end{align*}
		Similarly,
	\begin{align*}
		 	\mathbb{E}\left[
		 	\Big(\int_{T-\delta}^{T} \mathcal{W}_p(\bar{\nu}_t^n,\nu_t)dt
		 	\Big)^p\right]
		 	&\leq \delta^{p-1} 
		 	\int_{T-\delta}^T\mathbb{E}\left[   \mathcal{W}_p^p(\bar{\nu}_t^n,\nu_t) \right]dt\\
		 	&\leq   C\delta^{p-1}\lambda_{p,r,n}(kd)
		 	\int_{T-\delta}^T \Big(\mathbb{E}\left[|Z_t|^r\right] \Big)^{\frac{p}{r}} dt\\
		 	& \leq C\delta^{p-1}\lambda_{p,r,n}(kd)\int_{T-\delta}^T \Big(\mathbb{E}\left[|Z_t|^2\right] \Big)^{\frac{p}{2}} dt\\
		 	& \leq C
		 	\delta^{\frac{p}{2}} \lambda_{p,r,n}(kd)\Bigg( \mathbb{E}\int_{T-\delta}^T|Z_t|^2dt\Bigg)^{\frac{p}{2}}.
	\end{align*}
		Hence,
		\begin{align*}
				\mathbb{E}\left[
			\Big(\int_{T-\delta}^{T} \mathcal{W}_p(\bar{\mu}_t^n,\mu_t)dt
			\Big)^p\right]\leq  C\lambda_{p,r,n}(k) \q\ \hbox{and} \q\ 
\mathbb{E}\left[
\Big(\int_{T-\delta}^{T} \mathcal{W}_p(\bar{\nu}_t^n,\nu_t)dt
\Big)^p\right]\leq
 C\lambda_{p,r,n}(kd).
		\end{align*}
		Averaging \eqref{dif}   over \(i\) and combining the preceding estimates, we obtain
		\begin{align*}
 	&\  \frac{1}{n}\sum_{i=1}^{n}  \mathbb{E}\left[ \sup\limits_{t\in [T-\delta,T]}  |Y^i_t-\bar{Y}^i_t|^p 
 	+\Bigg( \int_{T-\delta}^T\Big(  |Z_t^{i,i}-\bar{Z}^i_t|^2
 	+\sum_{j\neq i} |Z_t^{i,j}|^2\Big)
 	dt \Bigg)^{\frac{p}{2}} 
 	\right]\\
 &\leq C
  \bar{\lambda}_{p,r,n}+ C(\delta^p+\delta^{\frac{p}{2}})
  \frac{1}{n}\sum_{i=1}^{n}  \mathbb{E}\left[ \sup\limits_{t\in [T-\delta,T]}  |Y^i_t-\bar{Y}^i_t|^p 
  +\Bigg( \int_{T-\delta}^T\Big(  |Z_t^{i,i}-\bar{Z}^i_t|^2
  +\sum_{j\neq i} |Z_t^{i,j}|^2\Big)
  dt \Bigg)^{\frac{p}{2}} 
  \right].
		\end{align*}
		Choose \(\delta>0\)   such that \(C(\delta^p+\delta^{\frac{p}{2}})< \frac{1}{2}\).
		Then
		\begin{align*}
		&\  \frac{1}{n}\sum_{i=1}^{n} \mathbb{E}\left[ \sup\limits_{t\in [T-\delta,T]}
	   |Y^i_t-\bar{Y}^i_t|^p 
		+\Bigg( \int_{T-\delta}^T\Big(  |Z_t^{i,i}-\bar{Z}^i_t|^2
		+\sum_{j\neq i} |Z_t^{i,j}|^2\Big)
		dt \Bigg)^{\frac{p}{2}} 
		\right]
		\leq C \bar{\lambda}_{p,r,n}.
		\end{align*}
	Next, choose a partition satisfying
		\begin{align*}
			0=t_0<t_1<\dots <t_N=T\q\ \hbox{and} \q\
			\max_{l} (t_{l+1}-t_l)  \leq \delta.
		\end{align*}
	Let
	\begin{align*}
		\Lambda_l=
		\frac{1}{n}\sum_{i=1}^{n}
	\mathbb{E}\left[ \sup\limits_{t\in [t_l,t_{l+1}]}
  |Y^i_t-\bar{Y}^i_t|^p 
	+\Bigg( \int_{t_l}^{t_{l+1}} \Big(  |Z_t^{i,i}-\bar{Z}^i_t|^2
	+\sum_{j\neq i} |Z_t^{i,j}|^2\Big)dt \Bigg)^{\frac{p}{2}} 
	\right].
	\end{align*}
	For  \(  l=0,\cdots, N-2,\) the 
	same argument gives
	\begin{align*}
	 \Lambda_l\leq C \frac{1}{n}\sum_{i=1}^{n}
	 \mathbb{E}\left[  
	 |Y^i_{t_{l+1}}-\bar{Y}^i_{t_{l+1}}|^p 
	 \right]+
	 C \bar{\lambda}_{p,r,n} +C(\delta^p+\delta^{\frac{p}{2}}) \Lambda_{l}.
	\end{align*}
	Since \(C(\delta^p+\delta^{\frac{p}{2}})< \frac{1}{2}\), 
	\begin{align*}
	  \Lambda_l\leq C(\Lambda_{l+1}+\bar{\lambda}_{p,r,n}).
	\end{align*}
Backward induction gives
\begin{align*}
	 \max_{0\leq l\leq N-1} \Lambda_l\leq C \bar{\lambda}_{p,r,n}.
\end{align*}
The constant \(C\) may depend on \(N\), which is chosen independently of \(n\).
By exchangeability, each   expectation in
  the definition of  \(\Lambda_l\) equals \(\Lambda_l\).
Moreover,
\begin{align*}
	\mathbb{E}\left[   \sup\limits_{t\in [0,T]}
	|Y^i_t-\bar{Y}^i_t|^p 
	\right]
 &	\leq \sum_{l=0}^{N-1}	\mathbb{E}\left[ \sup\limits_{t\in [t_l,t_{l+1}]}
	|Y^i_t-\bar{Y}^i_t|^p \right],\\
	\mathbb{E}\left[  \Bigg( \int_{0}^T\Big(  |Z_t^{i,i}-\bar{Z}^i_t|^2
	+\sum_{j\neq i} |Z_t^{i,j}|^2\Big)
	dt \Bigg)^{\frac{p}{2}} 
	\right]
&	\leq \sum_{l=0}^{N-1}	\mathbb{E}\left[  \Bigg( \int_{t_l}^{t_{l+1}}\Big(  |Z_t^{i,i}-\bar{Z}^i_t|^2
	+\sum_{j\neq i} |Z_t^{i,j}|^2\Big)
	dt \Bigg)^{\frac{p}{2}} 
	\right].
\end{align*} 
	This completes the proof.
	\end{proof}
	
 	We use the Wasserstein distances defined in \autoref{section1.2},
 	with the Euclidean norm on the product spaces.
 	The coupling estimate gives the following bounds for the
 	\(m\)-particle marginals.
	
	\begin{corollary}\label{coup}\sl 
		Assume that \autoref{lipp} holds with \(p\in (1,2)\), \(q=2\)
		and \(\sigma\) is independent of \(\theta\). For every \(1\leq m \leq n\),
		we have 
		\begin{equation}
		\begin{aligned}
			&\  	\mathcal{W}_{p,\|\cdot\|_{\infty}}^p(\mathbb{P}_{(X^1,\cdots,X^m)},(\mathbb{P}_{X})^{\otimes m})+
			\mathcal{W}_{p,\|\cdot\|_{\infty}}^p(\mathbb{P}_{(Y^1,\cdots,Y^m)},(\mathbb{P}_{Y})^{\otimes m}) 
			+
	 \mathcal{W}_{p,L^2}^p(\mathbb{P}_{(Z^{1,1},\cdots,Z^{m,m})},(\mathbb{P}_{Z})^{\otimes m})\\ 
		&	\leq
			Cm \bar{ \lambda}_{p,r,n},\label{t3.5result}
		\end{aligned}
		\end{equation}
	where \(r\) and \(\bar{\lambda}_{p,r,n}\) are as in \autoref{xyzcon},
	and \(C\) depends only on the parameters listed there.
	In particular, \(C\) is independent of \(m\) and \(n\).
	For each fixed \(m\), the right-hand side of \eqref{t3.5result}
	tends to zero as \(n\longrightarrow \infty\).
	\end{corollary}
	\begin{proof} 
	\begin{align*}
	&\	\mathcal{W}_{p,\|\cdot\|_{\infty}}^p(\mathbb{P}_{(X^1,\cdots,X^m)},(\mathbb{P}_{X})^{\otimes m})+
		\mathcal{W}_{p,\|\cdot\|_{\infty}}^p(\mathbb{P}_{(Y^1,\cdots,Y^m)},(\mathbb{P}_{Y})^{\otimes m})\\
		&\leq    \sum_{i=1}^{m} \mathbb{E}\left[   \|X^i-\bar{X}^i\|_{\infty}^p+
		\|Y^i-\bar{Y}^i\|_{\infty}^p \right]\\
&\leq   2 \sum_{i=1}^{m} \mathbb{E}\left[ \sup\limits_{t\in [0,T]} \Big(  |X_t^i-\bar{X}_t^i|^p+
|Y_t^i-\bar{Y}_t^i|^p \Big) \right],
	\end{align*}
	and
	\begin{align*}
 &\ \mathcal{W}_{p,L^2}^p(\mathbb{P}_{(Z^{1,1},\cdots,Z^{m,m})},(\mathbb{P}_{Z})^{\otimes m})
 	\leq \sum_{i=1}^{m}\mathbb{E}\left[  
 	\Bigg(\int_0^T |Z_t^{i,i}-\bar{Z}_t^i|^2dt\Bigg)^{\frac{p}{2}} \right].
	\end{align*}
	Combining the preceding coupling bounds with \autoref{xyzcon}  yields \eqref{t3.5result}.
	\end{proof}
	
	\begin{remark}\label{resultforsigma}\rm 
		In \autoref{xyzcon}, we assume that $\sigma$ is independent of \(\theta\) to obtain the forward $L^p$ stability estimate for
		$1<p<2$ using only second-moment bounds.
		Indeed, applying It\^o's formula to
		$(|\Delta X_t|^2+\varepsilon)^{p/2}$, where
		$\Delta X_t=X_t^i-\bar X_t^i$, produces a diffusion term involving
		\begin{align*}
		 	\Big(|\Delta X_t|^2+\varepsilon\Big)^{\frac{p-2}{2}}
		 |\sigma(t,X_t^i,\theta^n_t)-\sigma(t,\bar X_t^i, \theta_t)|^2.
		\end{align*}
		When $\sigma$ is independent of the law and Lipschitz in the state
		variable, this term is bounded by
		$C(|\Delta X_t|^2+\varepsilon)^{p/2}$.
	When the diffusion depends on \(\theta\), the same argument produces  an additional term involving
	\begin{align*}
	 	\Big(|\Delta X_t|^2+\varepsilon\Big)^{\frac{p-2}{2}}
	 \mathcal W_p^2(\theta_t^n,\theta_t),
	\end{align*}
	which cannot in general be absorbed into the preceding
	estimate uniformly in $\varepsilon$.
	The restriction on \(\sigma\) is therefore used in this \(L^p\) argument.
	Under the interaction structure in \autoref{nodimen}, an \(L^2\)
	argument also allows \(\sigma\) to depend on \(\theta\);
	see \autoref{nodi}.
	\end{remark}

		\begin{assumption}\label{nodimen}\sl 
			  There exist  positive integers
				\(r_b, r_\sigma,r_h,r_x,r_y,r_z\), measurable functions
				\begin{align*}
&	\phi_{b}: \mathbb{R}^d	\longmapsto \mathbb{R}^{r_{b}},\q\ 
		\phi_{h}: \mathbb{R}^d	\longmapsto \mathbb{R}^{r_{h}},\q\ 
		\phi_\sigma:\mathbb{R}^d \longmapsto \mathbb{R}^{r_\sigma},\\
	&	\phi_{x}: \mathbb{R}^d	\longmapsto \mathbb{R}^{r_{x}},\q\ 
			\phi_y:\mathbb{R}^k\longmapsto \mathbb{R}^{r_y}, \q\ 
		\phi_z:\mathbb{R}^{k\times d}\longmapsto \mathbb{R}^{r_z},
				\end{align*} 
					and   functions
					\(\hat{b}:[0,T]\times \Omega\times  \mathbb{R}^d\times \mathbb{R}^{r_b}
					\longmapsto \mathbb{R}^d\), 
					\(\hat{\sigma}:[0,T]\times \Omega\times  \mathbb{R}^d\times \mathbb{R}^{r_\sigma} \longmapsto \mathbb{R}^{d\times d}
					\), and 	\(\hat{f}:[0,T]\times
					\Omega \times \mathbb{R}^d \times \mathbb{R}^{k}
					\times\mathbb{R}^{k\times d} \times \mathbb{R}^{r_x}\times \mathbb{R}^{r_y}\times \mathbb{R}^{r_z}  
					\longmapsto \mathbb{R}^k\).
					The functions $\hat b$, $\hat\sigma$ and $\hat f$ are progressively measurable in $(t,\omega)$ for every fixed choice of their remaining
					arguments, and  
					\( \hat{h}: \Omega \times \mathbb{R}^d \times \mathbb{R}^{r_h}\longmapsto \mathbb{R}^k\) is 
					\(\mathscr{F}_T\otimes \mathcal{B}(\mathbb{R}^d)\otimes \mathcal{B}(\mathbb{R}^{r_h})\)-measurable.
				The coefficients admit the following representations:
			\begin{align*}
 b(t,\omega,x,\theta)&=\hat{b}(t,\omega,x,\langle \phi_b,\theta\rangle),\\
 \sigma(t,\omega,x,\theta)&=\hat{\sigma}(t,\omega,x,\langle \phi_\sigma,\theta\rangle),\\
 h(\omega,x,\theta)&=\hat{h}(\omega,x,\langle \phi_h,\theta\rangle),\\
 f(t,\omega,x,\theta,y,\mu,z,\nu)&=\hat{f}(t,\omega,x,y,z,\langle \phi_x,\theta\rangle,\langle \phi_y,\mu\rangle,
 \langle\phi_z,\nu\rangle),
			\end{align*}
			where
			\(\langle\phi_\alpha,\rho\rangle
			=\int_{E_\alpha}\phi_\alpha(u)\rho(du)\)
			for \(\rho\in\mathcal P_1(E_\alpha)\) and
			\(\alpha\in\{b,\sigma,h,x,y,z\}\), with
			\(E_b=E_\sigma=E_h=E_x=\mathbb R^d\),
			\(E_y=\mathbb R^k\), and \(E_z=\mathbb R^{k\times d}\).
		The functions \(\phi_\alpha\), \(\alpha\in\{b,\sigma,h,x,y,z\}\),
		are globally Lipschitz with a common constant \(C_\phi\).
		Moreover, there exist
			  constants \(\hat{L}_b,\hat{L}_\sigma,\hat{L}_h, \hat{L}_f\) such that
			\begin{align*}
	|\hat{b}(t,\omega,x,a)-\hat{b}(t,\omega,x,a')|& \leq \hat{L}_b|a- a'|,\\
		|\hat{\sigma}(t,\omega,x,a)-\hat{\sigma}(t,\omega,x,a')|& \leq \hat{L}_\sigma|a- a'|,\\
	|\hat{h}(\omega,x,a)-\hat{h}(\omega,x,a')|
	&\leq \hat{L}_h|a-a'|,\\
	|\hat{f}(t,\omega,x,y,z,a_x,a_y,a_z)
	-\hat{f}(t,\omega,x,y,z,a_x',a_y',a_z')|
	&\leq\hat{L}_f \Big( |a_x-a_x'|+|a_y-a_y'|+|a_z-a_z'|
	\Big).
			\end{align*}
		The constants in this assumption and the integers
		\(r_b,r_\sigma,r_h,r_x,r_y,r_z\) are independent of \(n,d,k\).
	\end{assumption} 
	
		\begin{theorem}\label{nodi}\sl 
	Assume that \autoref{lipp} holds with \(p\in(1,2]\) and \(q=2\),
	and that \autoref{nodimen} holds.
	Then, for every \(1\leq m\leq n\), we have
		\begin{equation}
			\begin{aligned}
				&\  	\mathcal{W}_{2,\|\cdot\|_{\infty}}^2(\mathbb{P}_{(X^1,\cdots,X^m)},(\mathbb{P}_{X})^{\otimes m})+
				\mathcal{W}_{2,\|\cdot\|_{\infty}}^2(\mathbb{P}_{(Y^1,\cdots,Y^m)},(\mathbb{P}_{Y})^{\otimes m}) 
				+
				\mathcal{W}_{2,L^2}^2(\mathbb{P}_{(Z^{1,1},\cdots,Z^{m,m})},(\mathbb{P}_{Z})^{\otimes m})\\ 
				&	\leq \frac{C
			m}{n}, \label{nodim_result}
			\end{aligned}
		\end{equation}
		where \(C>0\) depends on 
		\(\beta, L_b,L_\sigma,L_h,L_f,K_2,r_b,r_\sigma,r_h,r_x,r_y,r_z,\hat{L}_b,\hat{L}_\sigma,\hat{L}_h,\hat{L}_f, C_{\phi}\) and \(T\). 
		The constant \(C\) is independent of \(m\) and \(n\).
		If the listed quantities, including
		\(K_2\), are bounded uniformly as \(d, k\) vary, then \(C\) is also independent 	of \(d,k\).
		Furthermore, 
			\begin{equation}
			\begin{aligned}
				&\  	\mathcal{W}_{p,\|\cdot\|_{\infty}}^p(\mathbb{P}_{(X^1,\cdots,X^m)},(\mathbb{P}_{X})^{\otimes m})+
				\mathcal{W}_{p,\|\cdot\|_{\infty}}^p(\mathbb{P}_{(Y^1,\cdots,Y^m)},(\mathbb{P}_{Y})^{\otimes m}) 
				+
				\mathcal{W}_{p,L^2}^p(\mathbb{P}_{(Z^{1,1},\cdots,Z^{m,m})},(\mathbb{P}_{Z})^{\otimes m})\\ 
				&	\leq C_p
				\big( \frac{m}{n}\big)^{\frac{p}{2}}, \label{nodim_resultforp}
			\end{aligned}
		\end{equation}
	where \(C_p>0\) depends only on \(p\) and the parameters listed above.
	Under the same uniform bounds, \(C_p\) is also independent of \(d\) and \(k\).
				\end{theorem}
	\begin{proof}
	We use the notation introduced in the proof of \autoref{xyzcon}.
	Throughout the proof, \(C>0\) may vary from line to line.
By \autoref{exanduni} and the well-posedness argument in the proof
of \autoref{xyzcon}, both systems admit unique square-integrable solutions.
  From the triangle inequality,
 \begin{align*}
  \Big| b(t,X_t^i,\theta_t^n)-b(t,\bar{X}_t^i,\theta_t)\Big|
  &\leq
  \Big| b(t,X_t^i,\theta_t^n)-b(t,\bar{X}_t^i,\bar{\theta}^n_t) \Big|
  +\Big| b(t,\bar{X}_t^i,\bar{\theta}_t^n)-b(t,\bar{X}_t^i,\theta_t)\Big|.
 \end{align*}
 By the Lipschitz continuity of \(b\),
 \begin{align*}
\mathbb{E}\left[\Bigg(\int_0^T  \Big| b(t,X_t^i,\theta_t^n)-b(t,\bar{X}_t^i,\bar{\theta}^n_t) \Big| dt
\Bigg)^2\right]
\leq C \mathbb{E}\left[\sup\limits_{t\in [0,T]}
|X_t^i-\bar{X}_t^i|^2\right]
+\frac{C}{n}\sum_{j=1}^{n} \mathbb{E}\left[\sup\limits_{t\in [0,T]}
|X_t^j-\bar{X}_t^j|^2\right].
 \end{align*}
Since \(\bar X^1,\ldots,\bar X^n\) are i.i.d.,
\(X\in S^2_{\mathbb F}([0,T];\mathbb R^d)\), and
\(\phi_b\) is globally Lipschitz, we have
 \begin{align*}
 \mathbb{E}\left[ \Big| 
 \frac{1}{n}\sum_{j=1}^{n}\phi_b(\bar{X}_t^j)-\mathbb{E}\left[\phi_b(X_t)\right]\Big|^2\right] 
 &\leq 
  \frac{1}{n}\mathbb{E}\left[ \Big| 
 \phi_b(X_t)-\mathbb{E}\left[\phi_b(X_t)\right]\Big|^2\right].
 \end{align*}
 For the second term, 
 \autoref{nodimen} and  
  H\"older's inequality give
 \begin{align*}
  \mathbb{E}\left[\Bigg(\int_0^T  \Big| b(t,\bar{X}_t^i,\bar{\theta}_t^n)-b(t,\bar{X}_t^i,\theta_t) \Big| dt
  \Bigg)^2\right]
 & \leq 
  C \int_0^T \mathbb{E}\left[ \Big| b(t,\bar{X}_t^i,\bar{\theta}_t^n)-b(t,\bar{X}_t^i,\theta_t) \Big|^2\right] dt \\
  &\leq 
  C \int_0^T \mathbb{E}\left[ \Big| 
  \frac{1}{n}\sum_{j=1}^{n}\phi_b(\bar{X}_t^j)-\mathbb{E}\left[\phi_b(X_t)\right]\Big|^2\right] dt \\
  &\leq 
 \frac{C}{n}.
 \end{align*}
 Similarly,
 \begin{align*}
	\Big| \sigma(t,X_t^i,\theta_t^n)-\sigma(t,\bar{X}_t^i,\theta_t)\Big|
	&\leq
	\Big| \sigma(t,X_t^i,\theta_t^n)-\sigma(t,\bar{X}_t^i,\bar{\theta}^n_t) \Big|
	+\Big| \sigma(t,\bar{X}_t^i,\bar{\theta}_t^n)-\sigma(t,\bar{X}_t^i,\theta_t)\Big|.
\end{align*}
By \autoref{lipp}, \autoref{nodimen}, and  the Burkholder--Davis--Gundy inequality,
\begin{align*}
&\  \mathbb{E}\left[\sup\limits_{t\in [0,T]}
 \Bigg|\int_0^t \Big( \sigma(s,X_s^i,\theta_s^n)-\sigma(s,\bar{X}_s^i,\theta_s)\Big)dW_s^i \Bigg|^2
 \right]\\
 &\leq 
 C \mathbb{E} \int_0^T \Big(\big| \sigma(s,X_s^i,\theta_s^n)-\sigma(s,\bar{X}_s^i,\bar{\theta}^n_s) \big|^2
 +
 \big| \sigma(s,\bar{X}_s^i,\bar{\theta}_s^n)-\sigma(s,\bar{X}_s^i,\theta_s)\big|^2 \Big)ds\\
 &\leq C \mathbb{E}\left[\sup\limits_{t\in [0,T]}
 |X_t^i-\bar{X}_t^i|^2\right]
 +\frac{C}{n}\sum_{j=1}^{n} \mathbb{E}\left[\sup\limits_{t\in [0,T]}
 |X_t^j-\bar{X}_t^j|^2\right]+\frac{C}{n}.
\end{align*}
Applying the same argument on \([0,t]\) gives
\begin{align*}
	e_i(t)&\leq C\int_0^t\Big(e_i(s)+\frac1n\sum_{j=1}^n e_j(s)\Big)ds
	+\frac Cn,\\
	e_i(t)&\triangleq \mathbb E\sup_{0\leq u\leq t}|X_u^i-\bar X_u^i|^2.
\end{align*}
Averaging over \(i\) and applying
Gronwall's inequality gives
	\begin{align}\label{nodimforx}
	\frac{1}{n}\sum_{i=1}^{n}\mathbb{E}\left[
	\sup\limits_{t\in [0,T]}|X_t^i-\bar{X}_t^i|^2\right]
	\leq \frac{C}{n}.
\end{align}
Applying the same argument to \(h\),
we obtain
\begin{align*}
\mathbb{E}\left[ 
|h(X_T^i, \theta_T^n)-
 h(\bar{X}_T^i,\theta_T)|^2 \right] 
 \leq C\mathbb{E}\left[
 |X_T^i-\bar{X}_T^i|^2+\frac{1}{n}\sum_{j=1}^{n}|X_T^j-\bar{X}_T^j|^2\right]
+\frac{C}{n},
\end{align*}
which implies that
\begin{align*}
	 	\frac{1}{n}\sum_{i=1}^{n}
	 	\mathbb{E}\left[ 
	 	|h(X_T^i, \theta_T^n)-
	 	h(\bar{X}_T^i,\theta_T)|^2 \right] 
	 	\leq \frac{C}{n}.
\end{align*}
Moreover,
\begin{align*}
  &\  \Big| f(t,X_t^i,\theta_t^n,
\bar{Y}_t^i,\mu_t^n,\bar{Z}_t^{i,i}, \nu_t^n)-f(t,\bar{X}_t^i,\theta_t,
 \bar{Y}_t^i,\mu_t,\bar{Z}_t^{i,i},\nu_t) \Big| \\
 &\leq 
 \Big| f(t,X_t^i,\theta_t^n,
 \bar{Y}_t^i,\mu_t^n,\bar{Z}_t^{i,i}, \nu_t^n)-f(t,\bar{X}_t^i,\bar{\theta}^n_t,
\bar{Y}_t^i,\bar{\mu}^n_t,\bar{Z}_t^{i,i},\bar{\nu}^n_t)\Big| \\
&
 +\Big|f(t,\bar{X}_t^i,\bar{\theta}^n_t,
 \bar{Y}_t^i,\bar{\mu}^n_t,\bar{Z}_t^{i,i},\bar{\nu}^n_t)-f(t,\bar{X}_t^i,\theta_t,
 \bar{Y}_t^i,\mu_t,\bar{Z}_t^{i,i},\nu_t) \Big| \\
& \triangleq A_t^i+B_t^i,
\end{align*}
where by \autoref{lipp},
\begin{align*}
  |A_t^i|^2
\leq&\
C\Big(  |X_t^i-\bar{X}_t^i|^2+\frac{1}{n}\sum_{j=1}^{n}|X_t^j-\bar{X}_t^j|^2
+\frac{1}{n}\sum_{j=1}^{n}|Y_t^j-\bar{Y}_t^j|^2
 +\frac{1}{n}\sum_{j=1}^{n}|Z_t^{j,j}-\bar{Z}_t^{j,j}|^2\Big).
\end{align*}
By independence,
\begin{align*}
&\  \mathbb{E}\left[ \Big| 
\frac{1}{n}\sum_{i=1}^{n}\phi_x(\bar{X}_t^i)-\mathbb{E}\left[\phi_x(X_t)\right]\Big|^2+
 \Big| 
\frac{1}{n}\sum_{i=1}^{n}\phi_y(\bar{Y}_t^{i})-\mathbb{E}\left[\phi_y(Y_t)\right]\Big|^2
+ \Big| 
\frac{1}{n}\sum_{i=1}^{n}\phi_z(\bar{Z}_t^{i,i})-\mathbb{E}\left[\phi_z(Z_t)\right]\Big|^2\right] \\
&\leq 
\frac{1}{n}\mathbb{E}\left[ \Big| 
\phi_x(X_t)-\mathbb{E}\left[\phi_x(X_t)\right]\Big|^2+
\Big| 
\phi_y(Y_t)-\mathbb{E}\left[\phi_y(Y_t)\right]\Big|^2
+\Big| 
\phi_z(Z_t)-\mathbb{E}\left[\phi_z(Z_t)\right]\Big|^2\right].
\end{align*}
Together with \autoref{nodimen}, this gives
\begin{align*}
 \frac{1}{n}\sum_{i=1}^{n} \mathbb{E}
 \int_0^T |B_t^i|^2dt \leq \frac{C}{n}.
\end{align*}
Choose \(\delta>0\) such that \(C(\delta+\delta^2)\leq1/2\),
and choose a partition satisfying
\begin{align*}
	0=t_0<t_1<\dots <t_N=T\q\ \hbox{and} \q\
	\max_{l} (t_{l+1}-t_l)  \leq \delta.
\end{align*}
Let
\begin{align*}
	\Lambda_l=
	\frac{1}{n}\sum_{i=1}^{n}
	\mathbb{E}\left[ \sup\limits_{t\in [t_l,t_{l+1}]}
	|Y^i_t-\bar{Y}^i_t|^2
	+ \int_{t_l}^{t_{l+1}} \Big(  |Z_t^{i,i}-\bar{Z}^i_t|^2
	+\sum_{j\neq i} |Z_t^{i,j}|^2\Big)dt 
	\right].
\end{align*}
By the BSDE stability estimate used in the proof of \autoref{xyzcon},
\begin{equation} 
	\begin{aligned}\label{nodimforyandz}
\Lambda_{N-1}
&\leq \frac{C}{n}+C(\delta+\delta^2)\Lambda_{N-1}.
\end{aligned}
\end{equation}
For every \(0\leq l\leq N-2\), 
\begin{align*}
 \Lambda_{l}
 &\leq C\frac{1}{n}\sum_{i=1}^{n}
 \mathbb{E}\left[ |Y_{t_{l+1}}^i-\bar{Y}_{t_{l+1}}^i|^2\right]
 +\frac{C}{n}+C(\delta+\delta^2)\Lambda_l
 \\
 &\leq C\Lambda_{l+1}+\frac{C}{n}+C(\delta+\delta^2)\Lambda_{l}.
\end{align*}
The same backward induction argument as in \autoref{xyzcon} yields that
\begin{align*}
	\max_{0\leq l\leq N-1}
\Lambda_l
\leq \frac{C}{n}.
\end{align*}
Furthermore,
\begin{align}\label{nodimforyz}
	\frac{1}{n}\sum_{i=1}^{n}
\mathbb{E}\left[ \sup\limits_{t\in [0,T]}
|Y^i_t-\bar{Y}^i_t|^2
+ \int_0^T\Big(  |Z_t^{i,i}-\bar{Z}^i_t|^2
+\sum_{j\neq i} |Z_t^{i,j}|^2\Big)dt 
\right]
\leq \sum_{l=0}^{N-1}\Lambda_l\leq \frac{C}{n}.
\end{align}
Combining \eqref{nodimforx} with \eqref{nodimforyz}
and using  exchangeability,  we obtain
\begin{align*}
 \mathbb{E}\left[
 \sup\limits_{t\in [0,T]}\Big(|X_t^i-\bar{X}_t^i|^2 +
|Y^i_t-\bar{Y}^i_t|^2\Big)
+ \int_{0}^{T}\Big(  |Z_t^{i,i}-\bar{Z}^i_t|^2
+\sum_{j\neq i} |Z_t^{i,j}|^2\Big)
dt  
\right]\leq  \frac{C}{n}.
\end{align*}
By the definition of the Wasserstein distances,
	\begin{align*}
	&\	\mathcal{W}_{2,\|\cdot\|_{\infty}}^2(\mathbb{P}_{(X^1,\cdots,X^m)},(\mathbb{P}_{X})^{\otimes m})+
	\mathcal{W}_{2,\|\cdot\|_{\infty}}^2(\mathbb{P}_{(Y^1,\cdots,Y^m)},(\mathbb{P}_{Y})^{\otimes m})+\mathcal{W}_{2,L^2}^2(\mathbb{P}_{(Z^{1,1},\cdots,Z^{m,m})},(\mathbb{P}_{Z})^{\otimes m})\\
	&\leq   2 \sum_{i=1}^{m} \mathbb{E}\left[ \sup\limits_{t\in [0,T]} \Big(  |X_t^i-\bar{X}_t^i|^2+
	|Y_t^i-\bar{Y}_t^i|^2 \Big) \right]
	+\sum_{i=1}^{m}\mathbb{E} 
	\int_0^T |Z_t^{i,i}-\bar{Z}_t^i|^2dt,
\end{align*}
which implies \eqref{nodim_result}.
Finally, \eqref{nodim_resultforp} follows from
\(\mathcal W_p\leq\mathcal W_2\) and
\(a^{\frac{p}{2}}+b^{\frac{p}{2}}+c^{\frac{p}{2}}\leq3^{1-\frac{p}{2}}(a+b+c)^{\frac{p}{2}}\).
This completes the proof.
	\end{proof}

		\section{Sharp propagation of chaos in 
		Wasserstein distance} \label{sec4}
		
	 Throughout this section, we work in the deterministic Markovian setting.
	 The coefficients \(b,f,h\) are deterministic, the diffusion coefficient
	 is a constant matrix \(\Sigma\in\mathbb R^{d\times d}\), and
	 \autoref{lipp} holds with \(p=q=2\).
	The well-posedness argument in the proof of \autoref{xyzcon}
	also applies with \(p=2\). 
	Hence the particle BSDE admits
	a unique square-integrable solution.
	Set 
		\begin{align*}
		\theta_t=\mathbb{P}_{X_t}\q\ \hbox{and} \q\ 
		\theta_t^n=\frac{1}{n}\sum_{i=1}^{n}\delta_{X_t^i}.
		\end{align*}
	We use the following notation for mixed derivatives.
		Let \(w,l\in\mathbb N_0\) and
		\(\overrightarrow c=(c_1,\ldots,c_w)\in\mathbb N_0^w\).
		Set \(\alpha=(l,\overrightarrow c)\) and
		\(|\alpha|=l+\sum_{j=1}^w c_j\).
		For a map \(u :\mathbb{R}^d\times  \mathcal{P}_2(\mathbb{R}^d)
		\longmapsto \mathbb{R}^d\), define  
	\begin{align*}
	 D^{w,\alpha}u(x,\mu,v_1,\cdots,v_w)
	 \triangleq \triangledown_{v_1}^{c_1} \cdots \triangledown_{v_w}^{c_w}
	 D_\mu^w \triangledown_x^l u(x,\mu,v_1,\cdots,v_w).
	\end{align*}
The derivative \(D^{w,\alpha}u\) takes values in the corresponding
finite-dimensional tensor space, equipped with its Euclidean norm.
We denote by \(C_{bd}^k(\mathbb R^d\times\mathcal P_2(\mathbb R^d))\)
the class of functions \(u\) such that, for every \((w,\alpha)\)
with \(0<w+|\alpha|\leq k\), the derivative \(D^{w,\alpha}u\)
exists, is bounded, and satisfies
		\begin{align*}
		| D^{w,\alpha}u(x,\mu,v)-D^{w,\alpha}u(x',\mu',v')|
		\leq C\Big( |x-x'|+\mathcal{W}_2(\mu,\mu')+\sum_{j=1}^{w}|v_j-v_j'|\Big),
		\end{align*}
		 for all admissible arguments, with a constant \(C>0\).
		 The above definition  extends componentwise
		  to functions taking values
		 in any finite-dimensional Euclidean space.

The following example shows that the order \(m/n\) in
\autoref{nodi} cannot be improved in general.
		 
		\begin{example}\label{ex}\rm 
		 Let \(d=k=1\), 
		 \begin{align*}
		 	X_0^i=0,\q\ 
X_t^i=W_t^i,\q\ 
b=0,\q\  \sigma=1,
		 \end{align*}
	where \(X^1,\cdots,X^n\) are i.i.d., and for all \(1\leq m\leq n,\)
		 we have 
		 \begin{align*}
		  \mathcal{W}_{2,\|\cdot\|_{\infty}}^2(
		  \mathbb{P}_{(X^1,\cdots,X^m)},(\mathbb{P}_X)^{\otimes m})=0.
		 \end{align*}
Set
\begin{align*}
f(t,x,\theta,y,\mu,z,\nu)=0\q\ \hbox{and} \q\ 
h(x,\theta)=\int_{\mathbb{R}}u \theta(du ).
\end{align*}		 
These coefficients satisfy \autoref{lipp} with \(p=q=2\)
and \autoref{nodimen} with \(\phi_h(u)=u\) and \(\hat h(x,a)=a\).
Since \(h(X_T,\mathbb{P}_{X_T})=\mathbb{E}\left[X_T\right]=0,\) the limiting solution
satisfies
\(Y=Z=0\).
For each  \(i=1,\cdots, n,\) \(h(X_T^i,\theta_T^n)=\frac{1}{n}\sum_{j=1}^{n}W_T^j\), and 
\begin{align*}
 Y_t^i=\mathbb{E}\left[\frac{1}{n}\sum_{j=1}^{n}W_T^j ~\Big|~ 
 \mathscr{F}_t^n
   \right]=\frac{1}{n}\sum_{j=1}^{n}W_t^j\q\ \hbox{and}\q\ 
   Z_t^{i,j}=\frac{1}{n},\q 1\leq i,j\leq n.
\end{align*}
Let \(\overline{W}_t^n=\frac{1}{n}\sum_{j=1}^{n}W_t^j\).
Since
  \(\mathbb{P}_Y=\delta_0\), we have
\begin{align*}
\mathcal{W}_2^2(\mathbb{P}_{(Y_t^1,\cdots,Y_t^m)},\delta_0^{\otimes m})
=m\mathbb{E}\left[|\overline{W}_t^n|^2\right]=\frac{mt}{n},\q\
t\in [0,T].
\end{align*}
Since  \(\sqrt{n}\overline{W}^n \overset{\mathrm{d}}{=} W\), 
for every \(1\leq m\leq n,\)  
\begin{align*}
	\mathcal{W}_{2,\|\cdot\|_{\infty}}^2(
	\mathbb{P}_{(Y^1,\cdots,Y^m)},\delta_0^{\otimes m})
	= \frac{m}{n}\mathbb{E}\left[\sup\limits_{t\in [0,T]}|W_t|^2\right],
\end{align*}
Doob's inequality gives
\begin{align*}
 \frac{mT}{n}\leq 
 	\mathcal{W}_{2,\|\cdot\|_{\infty}}^2(
 \mathbb{P}_{(Y^1,\cdots,Y^m)},\delta_0^{\otimes m})
 \leq \frac{4mT}{n}.
\end{align*}

For comparison, change the terminal condition to
\(h(x,\mu)=x+\int_{\mathbb{R}} v\mu(dv)\), keeping \(b=f=0\), \(X_0=0\), and \(\Sigma=1\).
Then \(U(t,x,\mu)=x+\int_{\mathbb{R}} v\mu(dv)\), and
\begin{align*}
 Y_t^i=W_t^i+\overline W_t^n,\q\
  Z_t^{i,j}=\mathbf1_{\{i=j\}}+n^{-1},
 \q\  Y_t=W_t,\q\  Z_t=1.
\end{align*}
Set
\(	\mathbf{1}_m=(1,\cdots,1)^\top,
I_m=\operatorname{diag}(1, 1, \dots, 1)\).
The covariance matrix of \((Y_t^1,\cdots,Y_t^m)\) is
\(t(I_m+3n^{-1}\mathbf1_m\mathbf1_m^\top)\).
Diagonalizing this matrix gives
\begin{align*}
 \mathcal W_2^2(\mathbb P_{(Y_t^1,\ldots,Y_t^m)},\mathcal N(0,tI_m))
 =t\left(\sqrt{1+3m/n}-1\right)^2.
\end{align*}
To bound the path distance, couple a standard \(m\)-dimensional Brownian motion
\(B\) with \(\big(I_m+(\sqrt{1+3m/n}-1)uu^\top\big)B\), where
\(u=m^{-1/2}\mathbf1_m\). 
Evaluation at \(T\) gives the lower bound, and Doob's inequality
gives the upper bound:
\begin{align*}
 T\left(\sqrt{1+3m/n}-1\right)^2
 &\leq\mathcal W_{2,\|\cdot\|_\infty}^2
 (\mathbb P_{(Y^1,\ldots,Y^m)},  (\mathbb{P}_Y)^{\otimes m})
 \leq4T\left(\sqrt{1+3m/n}-1\right)^2.
\end{align*}
Thus the squared path Wasserstein distance for the \(Y\)-components
is of order \(m^2/n^2\), while the corresponding distance for the
diagonal \(Z\)-components is exactly \(mT/n^2\).
This model does not satisfy \autoref{asp4.2}, 
since its measure
dependence is linear and the mean statistic is unbounded.
The squared path Wasserstein distance for the joint law of the
forward and backward processes is of order \(m/n\).
The lower bound follows by applying
 \(  (\mathbf{x}, \mathbf{y})
\longmapsto (\mathbf y-\mathbf x)/ \sqrt{2}\) at time \(T\), and the upper bound follows from synchronous coupling.
Thus the cancellation condition is not necessary for the
\(m^2/n^2\) rate of the \(Y\)-marginals in this example.
		\end{example}
		\ms

We consider the master equation  
		\begin{equation}\label{master}
			\begin{aligned}
			&\  \partial_tU(t,x,\mu) +  D_x U(t,x,\mu) b(t,x,\mu) +
			\frac{1}{2}
			\operatorname{Tr} \Big(\Sigma \Sigma^{\top} D^2_x U(t,x,\mu) \Big)  \\
		&	+\int_{\mathbb{R}^d} D_{\mu} U(t,x,\mu)(\nu)
		b(t,\nu,\mu)\mu (d\nu)
			+\frac{1}{2}\int_{\mathbb{R}^d } \operatorname{Tr}\Big(\Sigma \Sigma^{\top}   D_{\nu} D_{\mu}
			U(t,x,\mu)(\nu) \Big) \mu (d\nu) \\
			& 	+f\Bigg(t,x,\mu,U(t,x,\mu),m^{U}(t,\mu),
		D_x U(t,x,\mu)\Sigma,m^{\mathcal{Z}}(t,\mu)\Bigg)=0,\\
		&   U(T,x,\mu)=h(x,\mu),
		\end{aligned}
	\end{equation}
	where 
	\begin{align*}
	 m^{U}(t,\mu)=\Big(U(t,\cdot,\mu) \Big)_{\# \mu}
	 \q\ \hbox{and} \q\ 
	m^{\mathcal{Z}}(t,\mu)=\Big(D_x U(t,\cdot,\mu) \Sigma\Big)_{\# \mu}.
	\end{align*}

	We adopt the Lions notion of classical solution:
all derivatives appearing in \eqref{master} are assumed to exist and to be jointly continuous, and the equation is satisfied pointwise. 
The regularity imposed below is stronger than the minimal regularity
needed to formulate the equation. 
In particular, \(D_xD_\mu U\)
and \(D_\mu^2U\) appear in the Itô formula for the empirical measure
lift, and their bounds control the resulting correction terms
at order \(n^{-1}\).
We refer to \cite{buckdahn_li_peng_rainer_17,chassagneux_crisan_delarue_22,carmona18_2}
for the It\^o formula on the space of measures, classical master equations,
and regularity of decoupling fields.
We assume below that \eqref{master} admits a classical solution
with this regularity.

			\begin{assumption}\label{asp4.1}\sl 
			The master equation \eqref{master} admits a classical solution \(U: [0,T]\times \mathbb{R}^d \times \mathcal{P}_2(\mathbb{R}^d)
		\longmapsto \mathbb{R}^k\) satisfying the following conditions:
		\begin{enumerate}[(i)]
				\item The function \(U\) and the derivatives
		\(  
		 \partial_tU, D_xU,  D_x^2U, D_\mu U,
		 D_\nu D_\mu U,  D_xD_\mu U,  D_\mu D_x U, D_\mu^2U\) exist and are jointly continuous in their respective variables, and 
		 \(D_xD_\mu U= D_\mu D_x U \);
	\item 
	There exists a constant \(L_U>0\) such that,
	for all \((t,x,\mu,\nu,\nu')\in 
	[0,T]\times \mathbb{R}^d \times \mathcal{P}_2(\mathbb{R}^d)
	\times \mathbb{R}^d\times \mathbb{R}^d,
      \)
      \begin{align*}
      &\ |D_x U(t,x,\mu)|
      +|D_{x}^2 U(t,x,\mu)|
      +|D_\mu U(t,x,\mu)(\nu)|
       +|D_\nu D_\mu U(t,x,\mu)(\nu)|\\
     & +|D_x D_\mu U(t,x,\mu)(\nu)|+|D_\mu D_x U(t,x,\mu)(\nu)|
      +|D_{\mu}^2 U(t,x,\mu)(\nu,\nu')|\leq L_U;
      \end{align*}
	\item  The functions
	\(U\) and \( \partial_t U \)  have at most  linear growth:
			\begin{align*}
|U(t,x,\mu)|+|\partial_t U(t,x,\mu)|
\leq 
L_U \Bigg( 1+|x|+\Big( \int_{\mathbb{R}^d} |\nu|^2 \mu (d\nu) \Big)^{\frac{1}{2}} \Bigg);
			\end{align*}
				\item  For every \((x,\mu)\in \mathbb{R}^d\times \mathcal{P}_2(\mathbb{R}^d)\),
				\(U(T,x,\mu)=h(x,\mu)\).
					\end{enumerate}
					All derivatives of the vector-valued function \(U\), as well as
				all trace operations involving these derivatives, are understood componentwise.
		\end{assumption}
		
		Under \autoref{asp4.1}, applying the Lions--It\^o formula to 
		\(U(t,X_t,\theta_t)\) and using the master equation \eqref{master}, we obtain 
		\begin{align*}
dU(t,X_t,\theta_t)
=-f(t,X_t,\theta_t,U(t,X_t,\theta_t),m^U(t,\theta_t), D_x U(t,X_t,\theta_t)\Sigma,
m^{\mathcal{Z}}(t,\theta_t)) dt +D_x U(t,X_t,\theta_t)\Sigma dW_t, 
		\end{align*}
		where
		\begin{align*}
m^U(t,\theta_t)=\mathbb{P}_{U(t,X_t,\theta_t)}
\q\ \hbox{and} \q\ m^{\mathcal{Z}}(t,\theta_t)
=\mathbb{P}_{D_x U(t,X_t,\theta_t)\Sigma}.
		\end{align*}
		By the uniqueness  of the
		 limiting BSDE \eqref{y1n}, 
		 as established
		 in  \autoref{exanduni}, we have
		\begin{align*}
			Y_t=U(t,X_t,\theta_t) \q\ 
\hbox{and} \q\ 
Z_t=D_x U(t,X_t,\theta_t)\Sigma.
		\end{align*}
Thus \(U\) is a  decoupling field for the limiting FBSDE system \eqref{x1n}--\eqref{y1n}.

		For each \(i=1,\cdots,n,\) define
		\begin{align*} 
			\widetilde{Y}_t^i&=U(t,X_t^i,\theta_t^n) \q\ \hbox{and} \q\ 
			\widetilde{Z}_t^{i,j}=\mathbf{1}_{\{i=j\}}
			D_xU(t,X^i_t,\theta_t^n)\Sigma
			+\frac{1}{n}  D_{\mu}U(t,X_t^i,\theta_t^n)(X^j_t)\Sigma;\\
			\widehat{Y}_t^i&=U(t,X_t^i,\theta_t)\q\ \hbox{and} \q\
			\widehat{Z}_t^{i,j}=\mathbf{1}_{\{i=j\}}
			D_xU(t,X^i_t,\theta_t)\Sigma.
		\end{align*}
		
\begin{lemma}\label{empirical}\sl 
  Assume that  \autoref{lipp} holds with \(p=q=2\), and that \autoref{asp4.1} is satisfied. Then, for every \(1\leq m\leq n,\)  
  \begin{align}\label{firstone}
 \mathcal{W}_{2,\|\cdot\|_{\infty}}^2(\mathbb{P}_{(Y^1,\cdots,Y^m)},\mathbb{P}_{(\widetilde{Y}^1,\cdots,\widetilde{Y}^m)})
 + \mathcal{W}_{2,L^2}^2(\mathbb{P}_{(Z^{1,1},\cdots,Z^{m,m})},\mathbb{P}_{(\widetilde{Z}^{1,1},\cdots,\widetilde{Z}^{m,m})})
   \leq C\frac{m}{n^2},
  \end{align}
  where the constant \(C>0\)  
   depends only on \(|\Sigma|, T\),
    the parameters in \autoref{lipp},   and  the constant \(L_U\)   in \autoref{asp4.1}.
   More precisely, for every \(i=1,\cdots,n\),
    \begin{align}\label{empirical-coupling}
    	\mathbb E\left[\sup_{t\in [0,T]}|Y_t^i-\widetilde Y_t^i|^2
    	+\int_0^T\sum_{j=1}^n|Z_t^{i,j}-\widetilde Z_t^{i,j}|^2dt\right]
    	\leq \frac C{n^2}.
    \end{align}
  Moreover, for every \(i=1,\cdots,n\),
    \begin{align*}
    \mathbb E\int_0^T\sum_{j\ne i}|Z_t^{i,j}|^2dt\leq \frac Cn.
    \end{align*}
\end{lemma}

\begin{proof}
Throughout the proof, \(C>0\)   may vary from line to line.
For  \(\overrightarrow{x}=(x^1,\cdots,x^n)\in (\mathbb{R}^{d})^n\),
set
\(\mu_{\overrightarrow{x}}^n= \frac{1}{n}\sum_{j=1}^{n} \delta_{x^j}\), and define
\begin{align*}
	u^{i,n}(t,\overrightarrow{x})\triangleq U(t,x^i,\mu_{\overrightarrow{x}}^n),\q\ i=1,\cdots,n.
\end{align*}
Under \autoref{asp4.1}, \(u^{i,n}\in C^{1,2}([0,T]\times (\mathbb{R}^d)^n;\mathbb{R}^k)\).
By the differentiation rule for empirical-measure lifts,
\begin{align*}
	 D_{x^j}u^{i,n}(t,\overrightarrow{x})
	 =\mathbf{1}_{\{i=j\}}D_x U(t,x^i,\mu_{\overrightarrow{x}}^n)
	 +\frac{1}{n}D_\mu U(t,x^i,\mu_{\overrightarrow{x}}^n)(x^j).
\end{align*}
Differentiating once more and using
 \(D_x D_\mu U=D_\mu D_x U\),
 we obtain
\begin{equation}\label{cal1}
\begin{aligned}
\frac{1}{2}\sum_{j=1}^{n}\operatorname{Tr}\Big(\Sigma\Sigma^\top D_{x^j}^2 u^{i,n}(t,\overrightarrow{x})\Big)
=&\ \frac{1}{2}\operatorname{Tr}\Big(\Sigma\Sigma^\top D_x^2U(t,x^i,\mu_{\overrightarrow{x}}^n) \Big)
+\frac{1}{2n}\sum_{j=1}^{n} \operatorname{Tr}\Big( \Sigma\Sigma^\top
D_\nu D_\mu U(t,x^i,\mu_{\overrightarrow{x}}^n)(x^j) \Big)\\
&+\frac{1}{n}\operatorname{Tr}\Big(\Sigma\Sigma^\top D_x D_\mu U(t,x^i,\mu_{\overrightarrow{x}}^n)(x^i) \Big)\\
&+\frac{1}{2n^2}\sum_{j=1}^{n}\operatorname{Tr}\Big(\Sigma\Sigma^\top
D_\mu^2 U(t,x^i,\mu_{\overrightarrow{x}}^n) (x^j,x^j)\Big).
\end{aligned}
\end{equation}
Moreover,
\begin{equation}\label{cal2}
	\begin{aligned}
	 \sum_{j=1}^{n} D_{x^j}u^{i,n}(t,\overrightarrow{x})b(t,x^j,\mu_{\overrightarrow{x}}^n)
	 =D_xU(t,x^i,\mu_{\overrightarrow{x}}^n) b(t,x^i,\mu_{\overrightarrow{x}}^n)
	 +\frac{1}{n}\sum_{j=1}^{n} D_{\mu} U(t,x^i,\mu_{\overrightarrow{x}}^n)(x^j)b(t,x^j,\mu_{\overrightarrow{x}}^n).
	\end{aligned}
\end{equation}
For each \(i=1,\cdots,n\),
set 
\begin{align*}
	\xi_t^{i,n}=D_x U(t,X^i_t,\theta_t^n)\Sigma\q\  \hbox{and} \q\ 
	\eta_t^n=\frac{1}{n}\sum_{j=1}^{n}\delta_{\xi_t^{j,n}}.
\end{align*}
Applying It\^o's formula to \(u^{i,n}(t,\overrightarrow{X_t})\), 
and using \eqref{cal1} and \eqref{cal2}, together with \eqref{master}, yields
\begin{align}\label{widey}
d\widetilde{Y}_t^i
=\Big(-f(t,X_t^i,\theta_t^n,\widetilde{Y}_t^i,m^U(t,\theta_t^n),\xi_t^{i,n},m^{\mathcal{Z}}(t,\theta_t^n))+R_t^{i,n}\Big) dt
+\sum_{j=1}^{n}\widetilde{Z}_t^{i,j}dW_t^j,
\end{align}
where
\begin{align*}
&	m^U(t,\theta_t^n)=\frac{1}{n}\sum_{j=1}^{n}\delta_{U(t,X_t^j,\theta_t^n)}
	 \q\ \hbox{and} \q\ 
	m^{\mathcal{Z}}(t,\theta_t^n)=
	\frac{1}{n}\sum_{j=1}^{n}\delta_{D_x U(t,X_t^j,\theta_t^n)\Sigma}=\eta_t^n,\\
	&R_t^{i,n}=
	\frac{1}{n}\operatorname{Tr}\Big(\Sigma \Sigma^\top D_x D_\mu U(t,X_t^i,\theta_t^n) (X_t^i)\Big) +\frac{1}{2n^2} \sum_{j=1}^{n}
	\operatorname{Tr} \Big( \Sigma \Sigma^\top D_\mu^2 U(t,X_t^i,\theta_t^n)(X_t^j,X_t^j) \Big).
\end{align*}
In particular, the uniform bounds in \autoref{asp4.1} imply 
  \(|R_t^{i,n}|\leq \frac{C}{n}\).
  Set \(\widetilde{\mu}_t^n
  =\frac{1}{n}\sum_{j=1}^{n}\delta_{\widetilde{Y}_t^j} \)
  and \(\widetilde{\nu}_t^n=\frac{1}{n}\sum_{j=1}^{n}\delta_{\widetilde{Z}_t^{j,j}}
  \).
  By the definition of
  \(\widetilde{Z}^{i,j}\),
\begin{align*}
|\widetilde{Z}^{i,i}_t-\xi_t^{i,n}|
&=\frac{1}{n}|D_\mu U(t,X_t^i,\theta_t^n)(X_t^i)\Sigma|\leq 
\frac{C}{n}\\
\mathcal{W}_2^2(\widetilde{\nu}_t^n,\eta_t^n)
&\leq \frac{1}{n}\sum_{j=1}^{n}|\widetilde{Z}_t^{j,j}-\xi_t^{j,n}|^2\leq \frac{C}{n^2}.
\end{align*}
Let 
\begin{align*}
	\mathcal{E}_t^{i,n}=
	f(t,X_t^i,\theta_t^n,\widetilde{Y}_t^i,\widetilde{\mu}_t^n,\xi_t^{i,n},\eta_t^n)
	-f(t,X_t^i,\theta_t^n,\widetilde{Y}_t^i,\widetilde{\mu}_t^n,\widetilde{Z}_t^{i,i},\widetilde{\nu}_t^n)-R_t^{i,n}.
\end{align*}
By the Lipschitz continuity of \(f\)
in \(z\) and \(\nu,\)
 \(|\mathcal{E}_t^{i,n}| \leq \frac{C}{n}\).
Since \(\widetilde{Y}_T^i= U(T,X_T^i,\theta_T^n)=h(X_T^i,\theta_T^n)\), 
we can rewrite
\eqref{widey}  as
\begin{align}\label{pury}
\widetilde{Y}_t^i
=h(X_T^i,\theta_T^n)+\int_t^T
\Big(f(s,X_s^i,\theta_s^n,\widetilde{Y}_s^i,\widetilde{\mu}_s^n,\widetilde{Z}_s^{i,i},\widetilde{\nu}_s^n)+\mathcal{E}_s^{i,n}\Big) ds
-\sum_{j=1}^{n}\int_t^T\widetilde{Z}_s^{i,j}dW_s^j,\q\ 0\leq t\leq 
T,
\end{align}
where
\begin{align*}
	&	\widetilde{\mu}_t^n=\frac{1}{n}\sum_{j=1}^{n}\delta_{\widetilde{Y}_t^j}
=m^{U}(t,\theta_t^n)\q\ \hbox{and} \q\ 	\widetilde{\nu}_t^n=\frac{1}{n}\sum_{j=1}^{n}\delta_{\widetilde{Z}_t^{j,j}}.
\end{align*}
By \autoref{asp4.1} and the moment bound for the forward particle system,
 	\((\widetilde{Y}^i,(\widetilde{Z}^{i,j})_{j=1}^n) \in 	S_{\mathbb{F}^n}^{2}([0,T];\mathbb{R}^{k})\times 	L_{\mathbb{F}^n}^{2}([0,T];\mathbb{R}^{ k \times nd}) \).
 	Set
 \begin{align*}
 \triangle Y^i=Y^i-\widetilde{Y}^i\q\ \hbox{and}\q\ 
 \triangle Z^{i,j}=Z^{i,j}-\widetilde{Z}^{i,j}.
 \end{align*}
 Then 
 \begin{align*}
 \triangle Y_t^i=\int_t^T \Big( F_s^{i,n}-\mathcal{E}_s^{i,n}\Big)ds -\sum_{j=1}^{n}\int_t^T \triangle Z_s^{i,j} dW_s^j,\q\ 0\leq t\leq T,
  \end{align*}
  where
  \begin{align*}
F_t^{i,n}=
f(t,X_t^i,\theta_t^n,Y_t^i,\mu_t^n,Z_t^{i,i}, \nu_t^n)
-f(t,X_t^i,\theta_t^n,\widetilde{Y}_t^i,\widetilde{\mu}_t^n,\widetilde{Z}_t^{i,i},\widetilde{\nu}_t^n).
  \end{align*}
By the monotonicity of \(f\) in \(y\), its Lipschitz continuity
in \(z,\mu,\nu\), and Young's inequality, for every \(\delta>0\),
  \begin{align*}
  	2\langle \triangle Y_t^i, F_t^{i,n} \rangle 
  	\leq 
  	C_\delta |\triangle Y_t^i|^2
  	+C_\delta \mathcal{W}_2^2 (\mu_t^n,\widetilde{\mu}_t^n)
  	+\delta |\triangle Z_t^{i,i}|^2
  	+\delta \mathcal{W}_2^2 (\nu_t^n,\widetilde{\nu}_t^n),
  \end{align*}
  where
\begin{align*}
 \mathcal{W}_2^2(\mu_t^n,\widetilde{\mu}^n_t)
& \leq \frac{1}{n}\sum_{i=1}^{n}|\triangle Y_t^i|^2\q\ \hbox{and} \q\ 
  \mathcal{W}_2^2(\nu_t^n,\widetilde{\nu}^n_t)
  \leq \frac{1}{n}\sum_{i=1}^{n}|\triangle Z_t^{i,i}|^2.
\end{align*}
Moreover,
\begin{align*}
 -2\langle \triangle Y_t^i,\mathcal{E}_t^{i,n}\rangle 
 \leq 
 |\triangle Y_t^i|^2+ |\mathcal{E}_t^{i,n}|^2.
\end{align*}
Applying Itô's formula to \(|\triangle Y_t^i|^2\),
averaging over \(i\), and taking expectations, we obtain,
for sufficiently small \(\delta>0\),
	\begin{align*}
	\frac{1}{n}\sum_{i=1}^{n}	\mathbb{E}\left[ 
	|\triangle Y^i_t |^2
	+ \frac{1}{2}\int_t^T  \sum_{j=1}^{n} |\triangle Z_s^{i,j} |^2
	ds
	\right]\leq
	 \frac{C}{n}\sum_{i=1}^{n} \mathbb{E}\int_t^T \Big( 
	|\triangle Y_s^i|^2 +|\mathcal{E}_s^{i,n}|^2\Big) ds.
\end{align*}
Then,
Gronwall's inequality gives
\begin{align*}
 \sup\limits_{t\in [0,T]}	\frac{1}{n}\sum_{i=1}^{n}	\mathbb{E}\left[ 
 |\triangle Y^i_t |^2\right] +	\frac{1}{n}\sum_{i=1}^{n}\mathbb{E}
  \int_0^T  \sum_{j=1}^{n} |\triangle Z_s^{i,j} |^2
 ds
 \leq 
 \frac{C}{n}\sum_{i=1}^{n}\mathbb{E} \int_0^T |\mathcal{E}_s^{i,n}|^2ds\leq
 \frac{C}{n^2}.
\end{align*}
Consequently, the Burkholder--Davis--Gundy inequality gives
	\begin{align}\label{newstabili}
	\frac{1}{n}\sum_{i=1}^{n}	\mathbb{E}\left[ \sup\limits_{t\in [0,T]}
	|\triangle Y^i_t |^2
	+ \int_0^T  \sum_{j=1}^{n} |\triangle Z_t^{i,j} |^2
	dt 
	\right]\leq  
 \frac{C}{n^2}.
\end{align}
By exchangeability, \eqref{newstabili} gives
\eqref{empirical-coupling}. The same coupling then gives
 \begin{align*}
\mathcal{W}_{2,\|\cdot\|_{\infty}}^2(\mathbb{P}_{(Y^1,\cdots,Y^m)},\mathbb{P}_{(\widetilde{Y}^1,\cdots,\widetilde{Y}^m)})
&\leq \sum_{i=1}^{m} \mathbb{E}\left[ \sup_{t\in [0,T]} |Y_t^i-\widetilde{Y}_t^i|^2\right]
\leq \frac{Cm}{n^2},\\ \mathcal{W}_{2,L^2}^2(\mathbb{P}_{(Z^{1,1},\cdots,Z^{m,m})},\mathbb{P}_{(\widetilde{Z}^{1,1},\cdots,\widetilde{Z}^{m,m})})
&\leq \sum_{i=1}^{m}\mathbb{E}\int_0^T |Z_s^{i,i}-\widetilde{Z}_s^{i,i}|^2 ds
\leq 
\frac{Cm}{n^2}.
 \end{align*}
 This proves \eqref{firstone}. For \(j\ne i\),
\(\widetilde Z_t^{i,j}=
n^{-1}D_\mu U(t,X_t^i,\theta_t^n)(X_t^j)\Sigma\).
Hence
\(
\mathbb E\int_0^T\sum\limits_{j\neq i}|\widetilde Z_t^{i,j}|^2 dt
\leq   Cn^{-1}.
\)
Together with \eqref{empirical-coupling}, this proves the
off-diagonal estimate.
\end{proof}

Under \autoref{nodimen}, the coefficients depend on the laws through
finitely many statistics. To compare
\((\widetilde Y,\widetilde Z)\) with \((\widehat Y,\widehat Z)\),
we impose a similar structure on \(U\), together with a first-order
cancellation condition along \((\theta_t)_{t\in[0,T]}\).

\begin{assumption}\label{asp4.2}\sl 
	There exist  an integer \(r_U\geq 1,\) two constants \(L_G, L_B>0,\) a function \(\phi_U \in C_b^{1,2}([0,T]\times \mathbb{R}^d; \mathbb{R}^{r_U})\), a function \(G:
	[0,T]\times \mathbb{R}^d \times \mathbb{R}^{r_U}\longmapsto \mathbb{R}^k\), and a measurable function \(B: [0,T]\times \mathbb{R}^{r_U}\longmapsto \mathbb{R}^{r_U}\).
	Let 
	\begin{align*}
	 \Phi_t(\mu) \triangleq \int_{\mathbb{R}^d} \phi_U(t,x)\mu (dx),\q\ 
	 A_t \triangleq \Phi_t(\theta_t),\q\ A_t^n \triangleq \Phi_t(\theta_t^n).
	\end{align*} 
	The following conditions hold:
	\begin{enumerate}[(i)]
	\item 
	For every \((t,x,\mu)\in [0,T]\times \mathbb{R}^d \times \mathcal{P}_2(\mathbb{R}^d)\),
	\begin{align*}
	 U(t,x,\mu)=G(t,x,   \Phi_t(\mu) ).
	\end{align*} 
	\item
	For every \((t,x)\in [0,T]\times \mathbb{R}^d,\)
	\begin{align*}
	 D_a G(t,x,A_t)=0.
	\end{align*} 
	\item 
	\(G\) and \(D_x G\) are jointly continuous in \(t,x,a\).
	For every \((t,x)\in [0,T]\times \mathbb{R}^d\), the maps \(a\longmapsto G(t,x,a)\) and 
	\(a\longmapsto D_x G(t,x,a)\) are twice continuously differentiable.
	Moreover,
	\begin{align*}
	 D_a D_x G=D_x D_a G\q\ \hbox{and}\q\  D_a^2 D_x G= D_x D_a^2 G.
	\end{align*}
	 The derivatives
	\(
	 D_a G, D_a^2 G, D_a D_x G,  D_a^2 D_x G
\)
	are jointly continuous.  For every \((t,x,a)\in [0,T]\times \mathbb{R}^d \times \mathbb{R}^{r_U}\),
	\begin{align*}
		|D_a G(t,x,a)|
		+|D_x D_a G(t,x,a)|
		+|D_a^2 G(t,x,a)|
		+|D_x D_a^2G(t,x,a)|
		\leq L_G.
	\end{align*}
	\item  
	For every \((t,\mu)\in [0,T]\times \mathcal{P}_2(\mathbb{R}^d)\),
	\begin{align*}
	 \int_{\mathbb{R}^d}
	 \Big( \partial_t \phi_U(t,x) + D_x \phi_U (t,x)b(t,x,\mu)
	 +\frac{1}{2}\operatorname{Tr}\big(\Sigma\Sigma^\top
	 D_x^2 \phi_U (t,x)\big)
	 \Big) \mu (dx)
	 =B(t, \Phi_t(\mu)).
	\end{align*}
	All differential operators and trace terms involving the vector-valued function \(\phi_U\) are understood componentwise.
	\item 
	For every 
	\((t,a,a')\in [0,T]\times \mathbb{R}^{r_U}\times \mathbb{R}^{r_U}\),
	\begin{align*}
		 |B(t,a)-B(t,a')| \leq L_B |a-a'|\q\ 
		 \hbox{and} \q\ 
		 \sup\limits_{t\in [0,T]}|B(t,0)|<\infty.
	\end{align*}
	\end{enumerate}
\end{assumption}

\begin{remark}\label{reamasp4.2}\rm 
	\begin{enumerate}[(1)]
	\item
	Condition (ii) eliminates the linear term
	\(
	D_aG(t,x,A_t)(A_t^n-A_t)
\)
	in the Taylor expansion of 
	\(
		G(t,x,A_t^n)-G(t,x,A_t).
\)
It  also implies that
\begin{align*}
D_aG(t,x,A_t)\big(\phi_U(t,v)-A_t\big)=0
\end{align*} 
	for every \(v\in\mathbb R^d\).
	Integrating this variation against \(\theta_t^n\) gives the linear term above.
 \item
Conditions (iv) and (v) may be replaced by the bound
 \begin{align*}
 \mathbb{E}\left[ 
 \sup\limits_{t\in [0,T]}
| A_t^n-A_t |^4 \right]
\leq \frac{C}{n^2}.
 \end{align*} 
with \(C\) independent of \(n\). In this case, \(B\) and \(L_B\)
are not needed.
 	\end{enumerate}
\end{remark}

\begin{example}\label{exasp4.2}\rm
Let \(d=k=r_U=1\), \(b=0\), \(\Sigma=\sigma_0>0\).
Let
\((P_t)_{t\geq 0}\) denote the heat semigroup associated with 
\( 
dX_t=\sigma_0 dW_t.
\)

We first consider the case \(f=0\).
Let \(g, \phi \in C_b^{\infty}(\mathbb{R})\), with \(\phi\)  nonconstant, and set 
\(a^* = \langle \phi,\theta_T\rangle \).
Choose \(\rho\in C_b^{\infty}(\mathbb{R})\) such that
\(\rho(q)=q^2\)  for \(|q|\leq 
2\|\phi\|_{\infty}\), and define 
\begin{align*}
h(x,\mu)=g(x)+\rho(\langle \phi, \mu \rangle- a^*).
\end{align*}
Then the corresponding master equation admits a classical solution 
given by
\begin{align*}
 U(t,x,\mu)
 =P_{T-t}g(x)+\rho (\langle P_{T-t}\phi,\mu\rangle -a^*).
\end{align*}
\rm Indeed,  set
\( 
 u_t(x)=P_{T-t}g(x)\) and \(\psi_t(x)=P_{T-t}\phi(x)
\).
Then
\begin{align*}
 &(\partial_t+\frac{\sigma_0^2}{2}\partial_{xx}) u_t=0\q\ \hbox{and}\q\ 
 (\partial_t+\frac{\sigma_0^2}{2}\partial_{xx}) \psi_t=0,\\
 &D_\mu U(t,x,\mu)(\nu)=\rho' (\langle \psi_t,\mu\rangle -a^*)
 D_\nu \psi_t (\nu),\\
 &D_\nu D_\mu U(t,x,\mu)(\nu)
 =\rho' (\langle \psi_t,\mu\rangle -a^*)
 D^2_\nu \psi_t (\nu).
\end{align*}
 Hence 
\begin{align*}
 &  \partial_t U(t,x,\mu)+\frac{\sigma_0^2}{2} \partial_{xx} U(t,x,\mu)
 +\frac{\sigma_0^2}{2}\int_{\mathbb{R}}D_\nu D_\mu U(t,x,\mu)(\nu) \mu (d\nu)\\
&  =(\partial_t+\frac{\sigma_0^2}{2}\partial_{xx})  u_t(x)
 +\rho' (\langle \psi_t,\mu\rangle -a^*)\langle (\partial_t+\frac{\sigma_0^2}{2}\partial_{xx}) \psi_t,\mu \rangle =0,\\
  & U(T,x,\mu)=g(x)+\rho(\langle \phi, \mu \rangle- a^*)=h(x,\mu).
\end{align*}
Set 
\begin{align*}
\phi_U(t,x)=P_{T-t}\phi(x) \q\ \hbox{and} \q\ 
G(t,x,a)=P_{T-t}g(x)+\rho(a-a^*).
\end{align*}
Then
\begin{align*}
 U(t,x,\mu)&=G(t,x,\langle \phi_U(t,\cdot),\mu \rangle).
\end{align*}
Since \(g,\phi, \rho \in C_b^{\infty}\), conditions (i) and (iii) are satisfied. Moreover, 
\begin{align*}
 \int_{\mathbb{R}}\Big(\partial_t \phi_U +\frac{\sigma_0^2}{2}\partial_{xx}
 \phi_U \Big)d\mu=0.
\end{align*}
Thus
 conditions (iv) and (v) in \autoref{asp4.2} hold with \(B=0\).   By the semigroup property,
\begin{align*}
\langle P_{T-t}\phi,\theta_t\rangle
=\langle \phi,\theta_T\rangle =a^*.
\end{align*}
It follows that
\begin{align*}
D_a G(t,x,a^*)=\rho'(0)=0 \q\ \hbox{and} \q\
D_x D_a G(t,x,a^*)=0.
\end{align*}
Thus \autoref{asp4.2} is satisfied. 
Since \(\phi\) is nonconstant,  there exists
\(x_0\in \mathbb{R}\) such that 
\(\phi(x_0)\neq a^*\).
Moreover, 
\begin{align*}
	U(T,x,\theta_T)=g(x)\q\ \hbox{and} \q\ 
	U(T,x,\delta_{x_0})
	=g(x)+(\phi(x_0)-a^*)^2.
\end{align*}
Hence
 \(U\) depends nonlinearly and nontrivially on its measure argument. 
 The particle correction is explicit. Set
\(q_t^n=\langle\psi_t,\theta_t^n\rangle-a^*\) and
\(\kappa_t=P_{T-t}(\phi^2)-\psi_t^2\).
Conditional independence of the forward particles gives
\begin{align*}
 Y_t^i&=u_t(X_t^i)+(q_t^n)^2+\frac1{n^2}\sum_{j=1}^n\kappa_t(X_t^j),\\
 Z_t^{i,j}&=\mathbf1_{\{i=j\}}\sigma_0u_t'(X_t^i)
 +\sigma_0\left(\frac{2q_t^n}{n}\psi_t'(X_t^j)+\frac1{n^2}\kappa_t'(X_t^j)\right).
\end{align*}
In particular, \(Y_t^i-\widetilde Y_t^i=n^{-1}\langle\kappa_t,\theta_t^n\rangle\).
To obtain a matching lower bound,
 set \(g=0\).
 Then \(Y=0\), and all particle value processes coincide.
Let \(v=\operatorname{Var}(\phi(X_T))>0\) and
\(\mu_4=\mathbb E\left[ |\phi(X_T)-a^*|^4\right]\).
For \(T>0\), positivity of \(v\) follows from the nonconstancy of \(\phi\)
and the strictly positive density of \(X_T\).
The terminal squared distance is exactly
\begin{align*}
 \mathcal W_2^2(\mathbb P_{(Y_T^1,\ldots,Y_T^m)},\delta_0^{\otimes m})
 =\sum_{i=1}^{m}\mathbb{E}\left[|Y_T^i|^2\right]
 =m\left(\frac{3v^2}{n^2}+\frac{\mu_4-3v^2}{n^3}\right).
\end{align*}
Since \(Y_t^i=\mathbb E\left[Y_T^i\mid \mathscr F_t^n\right]\), Doob's inequality gives 
\begin{align*}
 \mathcal{W}_{2,\|\cdot\|_{\infty}}^2( \mathbb{P}_{(Y^1,\cdots,Y^m)},
 \delta_0^{\otimes m}
 )=m\mathbb{E}\left[\sup\limits_{t\in [0,T]} |\mathbb{E}\left[Y_T^i | \mathscr{F}_t^n\right]|^2 \right]
 \leq 
 4m\left(\frac{3v^2}{n^2}+\frac{\mu_4-3v^2}{n^3}\right).
\end{align*}
The squared path distance is bounded below by the terminal squared distance.
Thus the \(m/n^2\) term in \autoref{sharpuperr} is optimal even
when the forward particles are independent and \(U\) depends
nontrivially on the measure.

 We now return to arbitrary \(g\) and introduce a nonzero generator \(f\), depending on both backward laws, while keeping the same decoupling field.
Fix   \(\eta_Y, \eta_Z \in \mathbb{R}\setminus \{0\}\).
 Keep the forward equation and terminal condition above, take a
Gaussian initial law, 
and define
\begin{align*}
 f(t,x,\theta,y,\mu,z,\nu)
 &=\eta_Y\left(\int_{\mathbb R} y'\mu(dy')-\mathfrak m_Y(t,\theta)\right)
 +\eta_Z\left(\int_{\mathbb R} z'\nu(dz')-\mathfrak m_Z(t,\theta)\right),\\
 \mathfrak m_Y(t,\theta)
 &=\langle u_t,\theta\rangle+
 \rho(\langle\psi_t,\theta\rangle-a^*),
 \q\
 \mathfrak m_Z(t,\theta)=\sigma_0\langle u_t',\theta\rangle,
\end{align*}
  These are prescribed functions of the
forward law: \(\mathfrak m_Y=\int_\mathbb{R} U(t,v,\theta)\theta(dv)\) and
\(\mathfrak m_Z=\int_\mathbb{R} \sigma_0D_xU(t,v,\theta)\theta(dv)\).
Their uniform \(\mathcal W_2\)-Lipschitz bounds follow from the bounded
 derivatives of \(u_t,\psi_t,\rho\). Thus the new generator
satisfies \autoref{lipp} with \(p=q=2\).
 It does not depend on \(y\), 
 so its monotonicity constant can be 
 taken to be zero.
The generator vanishes when \(U\) is substituted into the master equation, so the same \(U\) remains a classical solution.
The preceding verification of
\autoref{asp4.2} still applies.
The field \(U\)
 satisfies 
 \autoref{asp4.1},  and
\autoref{asp4.3} holds since \(b=0\), \(\Sigma=\sigma_0>0\), and the initial law is Gaussian.
Hence
\autoref{sharpuperr} applies.

 For a
configuration \(\mathbf x\),  set \(\theta^n=n^{-1}\sum_{j=1}^n \delta_{x_j}\),
\(a_n=\langle\psi_t,\theta^n\rangle\), and evaluate the generator at the
empirical lift and its diagonal martingale integrands. Its value is
\begin{align*}
 \frac{\eta_Z\sigma_0}{n}\rho'(a_n-a^*)
 \langle\psi_t',\theta^n\rangle.
\end{align*}
For example, with \(\phi=\sin\) and a centered Gaussian initial law,
\(a^*=0\), and this expression is nonzero when all coordinates lie in
\((0,\pi/2)\). 
Thus the feedback is nonzero for some particle configurations,
although it vanishes when \(U\) is substituted into the master equation.

We next consider a linear generator without the terms \(\mathfrak m_Y\) and \(\mathfrak m_Z\).
Keep \(b=0\), \(\Sigma=\sigma_0>0\), take a centered Gaussian initial law, and set
\begin{align*}
 h(x,\theta)=\langle\sin,\theta\rangle^2\q\ \hbox{and} \q\ 
 f(t,x,\theta,y,\mu,z,\nu)
 =\sigma_0^2\int y'\mu(dy')+\eta_Z\int z'\nu(dz'),\quad \eta_Z\ne0.
\end{align*}
A classical solution is given by \(U(t,x,\theta)=\langle\sin,\theta\rangle^2\).
Indeed, the term involving the measure derivative is \(-\sigma_0^2\langle\sin,\theta\rangle^2\),
the feedback through the mean of \(Y\) cancels this term,
and \(D_xU=0\).
All mixed derivatives required by \autoref{asp4.1} are bounded.
For \autoref{asp4.2}, take \(\phi_U=\sin\),
\(G(t,x,a)=\rho_0(a)\), with \(\rho_0\in C_b^\infty\) equal to \(a^2\) on \([-1,1]\),
and \(B(t,a)=-\sigma_0^2a/2\).
Symmetry gives \(A_t=0\), so first-order cancellation holds.
The generator is globally Lipschitz in the two laws and  does not depend on \(x, y, z\).
Thus \autoref{lipp} holds, and the forward assumptions in \autoref{asp4.3} follow as above.

The feedback through the law of \(Z\) is also nontrivial in the particle system.
To see this, put \(\tau=T-t\), \(a_{t,n}=\sigma_0\eta_Z\tau/n\), and
\begin{align*}
 q_{t,n}(x)=P_\tau\sin(x+a_{t,n}),\q\
 v_{t,n}(x)=P_\tau(\sin^2)(x+a_{t,n})-q_{t,n}(x)^2.
\end{align*}
Set \(\mathbf X_t=(X_t^1,\ldots,X_t^n)\). 
For every \(i\), \(Y_t^i=Q_t^n\), where
\begin{align*}
 Q_t^n=\mathcal Q^n(t,\mathbf X_t),\q\
 \mathcal Q^n(t,\mathbf x)
 =e^{\sigma_0^2\tau}\left[
 \left(\frac1n\sum_{j=1}^n q_{t,n}(x_j)\right)^2
 +\frac1{n^2}\sum_{j=1}^n v_{t,n}(x_j)\right],
\end{align*}
and \(Z_t^{i,j}=\sigma_0\partial_{x_j}\mathcal Q^n(t,\mathbf X_t)\) for every \(i,j\).
This follows from the linear BSDE representation, using a change of measure that shifts each Brownian drift by \(\eta_Z/n\).
At \(x_j=\pi/4-a_{t,n}\) for every \(j\), direct differentiation gives
\begin{align*}
 \sigma_0\partial_{x_j}\mathcal Q^n(t,\mathbf x)
 =\frac{\sigma_0}{n^2}\left(n-1+e^{-\sigma_0^2\tau}\right)>0.
\end{align*}
For \(t<T\), \(v_{t,n}(x)>0\) for every \(x\), so \(Q_t^n>0\). By continuity, 
the average of the diagonal integrands
 is positive on a neighborhood of this configuration.
  For \(0<t<T\), 
the forward particle vector
  has a strictly positive Gaussian density, 
  so this event has positive probability. 
Thus both feedback terms are nonzero with positive probability, 
   although the limiting solution is \((Y,Z)=(0,0)\).
The terminal value is 
\begin{align*}
  \Bigg(\frac 1n \sum_{j=1}^n\sin(X_T^j)\Bigg)^2.
\end{align*}
The  terminal lower bound above and \autoref{sharpuperr}, with \(F_{m,n}=0\),   show that the squared path Wasserstein distance for
 \(Y\) is of order \(m/n^2\).
\end{example}

\begin{lemma}\label{middle}\sl 
	Assume that \autoref{asp4.1} and  \autoref{asp4.2} hold. Then,  for every \(1\leq m\leq n,\)
	\begin{align}\label{second}
		\mathcal{W}_{2,\|\cdot\|_{\infty}}^2(\mathbb{P}_{(\widetilde{Y}^1,\cdots,\widetilde{Y}^m)},\mathbb{P}_{(\widehat{Y}^1,\cdots,\widehat{Y}^m)})
+
	\mathcal{W}_{2,L^2}^2(\mathbb{P}_{(\widetilde{Z}^{1,1},\cdots,\widetilde{Z}^{m,m})},\mathbb{P}_{(\widehat{Z}^{1,1},\cdots,\widehat{Z}^{m,m})})		\leq C\frac{m}{n^2},
	\end{align}
	where \(C>0\) depends on 
	\(T, L_B, L_G, r_U, \|\phi_U\|_{\infty}, \|D_x \phi_U \Sigma\|_{\infty}, |\Sigma|\).
\end{lemma}

 \begin{proof}
 	The constant \(C\) in the following may vary from line to line.
 Applying It\^o's formula to \(A_t^n\) and using condition (iv) of \autoref{asp4.2},  we have
 \begin{align*}
 A_t^n
 =A_0^n +\int_0^t   B(s,A_s^n) ds
 +Q_t^n, \q\ 0\leq t\leq T,
 \end{align*}
 where
 \begin{align*}
Q_t^n=\frac{1}{n}\sum_{j=1}^{n} \int_0^t D_x \phi_U (s,X_s^j)\Sigma dW_s^j.
 \end{align*}
Applying It\^o's formula to 
 \(\phi_U (t,X_t)\) and taking expectations gives
 \begin{align*}
 	A_t=A_0+\int_0^t B(s,A_s)ds,\q\ 0\leq t\leq T.
 \end{align*}
 For every \(i=1,\cdots,n\), let
 \begin{align*}
\xi^i=\phi_U(0,X_0^i)-\mathbb{E}\left[\phi_U (0,X_0)\right].
 \end{align*}
 Since \(\xi^1,\cdots,\xi^n\)  are  i.i.d.,
 centered,
 and  bounded,
 Rosenthal's inequality gives
 \begin{align*}
 \mathbb{E}\left[\big |\sum_{i=1}^{n}\xi^i\big|^4\right]
 \leq 
 C_4 \Bigg( \sum_{i=1}^{n}\mathbb{E}\left[|\xi^i|^4\right]
 +\Big( \sum_{i=1}^{n}\mathbb{E}\left[|\xi^i|^2\right]\Big)^2\Bigg).
 \end{align*}
Therefore, 
\begin{align*}
\mathbb{E}\left[ |A_0^n-A_0|^4\right]
=\frac{1}{n^4} \mathbb{E}\left[ \big|\sum_{i=1}^{n}\xi^i\big|^4\right]
\leq \frac{C}{n^4}
\Big(n\mathbb{E}\left[|\xi^1|^4\right]
+n^2 \Big(\mathbb{E}\left[|\xi^1|^2\right]\Big)^2\Big)
\leq C\Big( \frac{1}{n^3}+\frac{1}{n^2}\Big) \leq 
\frac{C}{n^2}.
\end{align*}
By the Burkholder--Davis--Gundy inequality,
\begin{align*}
 \mathbb{E} \left[ \sup\limits_{t\in [0,T]}
 |Q_t^n|^4 \right] \leq 
 C\mathbb{E}
 \left[\Bigg(\frac{1}{n^2}\sum_{j=1}^{n}\int_0^T \big|D_x \phi_U (s,X_s^j)\Sigma \big|^2ds\Bigg)^2 \right]
 \leq C\Bigg( \frac{T \|D_x \phi_U \Sigma\|_{\infty}^2}{n} \Bigg)^2
\leq  \frac{C}{n^2}.
\end{align*}
 By the Lipschitz continuity 
 of \(B\),
\begin{align*}
	\mathbb{E}\left[\sup\limits_{s\in [0,t]}
	|A_s^n-A_s|^4\right]\leq 
C \mathbb{E}\left[
|A_0^n-A_0|^4 +\sup\limits_{s\in [0,t]} |Q_s^n|^4
+\int_0^t  \sup\limits_{r\in [0,s]}|A_r^n-A_r|^4 ds \right].
\end{align*}
Gronwall's inequality then gives
\begin{align*}
	\mathbb{E}\left[\sup\limits_{t\in [0,T]}
	|A_t^n-A_t|^4\right]\leq 
	\frac{C}{n^2}.
\end{align*}
A second-order Taylor expansion of \(G\)
in \(a\) gives
\begin{align*}
 \widetilde{Y}_t^i-\widehat{Y}_t^i
 =D_a G(t,X_t^i,A_t)(A_t^n-A_t)
 +\int_0^1 (1-r) D_a^2 G(t,X_t^i,A_t+r(A_t^n-A_t)) [A_t^n-A_t,
 A_t^n-A_t] dr.
\end{align*}
Condition (ii) of \autoref{asp4.2} implies that
\begin{align*}
	|\widetilde{Y}_t^i-\widehat{Y}_t^i|
	\leq 
	C|A_t^n-A_t|^2.
\end{align*}
Hence
\begin{align*}
\mathbb{E}\left[
\sup\limits_{t\in [0,T]}
	|\widetilde{Y}_t^i-\widehat{Y}_t^i|^2 \right]
	\leq 	C \mathbb{E}\left[\sup\limits_{t\in [0,T]}
	|A_t^n-A_t|^4\right]\leq 
	\frac{C}{n^2}.
\end{align*}
Furthermore,  the chain rule for Lions derivatives yields
\begin{align*}
D_\mu U(t,x,\mu)(\nu)
=D_a G(t,x,\Phi_t(\mu)) D_\nu \phi_U(t,\nu).
\end{align*}
Conditions (ii) and (iii) give
\begin{align*}
 D_x D_a G(t,x,A_t)=
 D_a D_x G(t,x,A_t)=0.
\end{align*}
A second-order Taylor expansion of \(D_xG\) in \(a\) gives
\begin{align*}
|\widetilde{Z}_t^{i,i}
-\widehat{Z}_t^{i,i}|
&\leq |\big( D_x G(t,X_t^i,A_t^n)-D_xG(t,X_t^i,A_t)\big) \Sigma|
+\frac{1}{n}|D_a G(t,X_t^i,A_t^n) D_x \phi_U (t,X_t^i)\Sigma|\\
&\leq 
C|A_t^n-A_t|^2
+\frac{C}{n}|A_t^n-A_t|.
\end{align*}
By  H\"older's inequality,
\begin{align*}
 \mathbb{E}\int_0^T |\widetilde{Z}_t^{i,i}
 -\widehat{Z}_t^{i,i}|^2 dt
& \leq C\mathbb{E}\int_0^T \Big( 
 |A_t^n-A_t|^4+\frac{1}{n^2}|A_t^n-A_t|^2\Big)dt\\
 &\leq CT\Bigg\{ \mathbb{E}\left[\sup\limits_{t\in [0,T]}
 |A_t^n-A_t|^4  \right]+\frac{1}{n^2} \Bigg( \mathbb{E}\left[\sup\limits_{t\in [0,T]}
 |A_t^n-A_t|^4\right]\Bigg)^{\frac{1}{2}}\Bigg\}\\
 &\leq C\Big( \frac{1}{n^3}+\frac{1}{n^2}\Big)\\ & \leq 
 \frac{C}{n^2}.
\end{align*} 
 Finally, the synchronous coupling gives
 \begin{align*}
 	&\	\mathcal{W}_{2,\|\cdot\|_{\infty}}^2(\mathbb{P}_{(\widetilde{Y}^1,\cdots,\widetilde{Y}^m)},\mathbb{P}_{(\widehat{Y}^1,\cdots,\widehat{Y}^m)})
 	+
 	\mathcal{W}_{2,L^2}^2(\mathbb{P}_{(\widetilde{Z}^{1,1},\cdots,\widetilde{Z}^{m,m})},\mathbb{P}_{(\widehat{Z}^{1,1},\cdots,\widehat{Z}^{m,m})})\\
 	&\leq 
 	\sum_{i=1}^{m}\mathbb{E}\left[
 	\sup\limits_{t\in [0,T]}
 	|\widetilde{Y}_t^i-\widehat{Y}_t^i|^2+\int_0^T |\widetilde{Z}_t^{i,i}
 	-\widehat{Z}_t^{i,i}|^2 dt
 	\right]\\
 &\leq C\frac{m}{n^2}.
 \end{align*}
This completes the proof.
 \end{proof}

\begin{assumption}\label{asp4.3}\sl 
Set \(V_t(x,\mu)=b(t,x,\mu)\).
	\begin{enumerate}[(i)]
	\item 
 \( V_t \in C^6_{bd}(\mathbb{R}^d\times \mathcal{P}_2(\mathbb{R}^d))\),
 uniformly in \(t\in [0,T]\).
\item
\(\mathbb{P}_{X_0}\) satisfies a \(T_1\) transport inequality:
there exists a constant \(C_{T_1}>0\) such that for every \(\theta \in \mathcal{P}_1(\mathbb{R}^d)\),
\begin{align*}
 \mathcal{W}_1^2(\mathbb{P}_{X_0},\theta)
 \leq 2C_{T_1}	H(\theta|\mathbb{P}_{X_0}).
\end{align*}
\item
 \(\Sigma \in \mathbb{R}^{d\times d}\)
 is constant and invertible.
 	\end{enumerate}
\end{assumption}
		 
			\begin{remark}\label{sharpforsde}\rm 
			The result of
			Arnese and Lacker \cite[Corollary 2.4, Remarks 2.7 and 3.6]{man_sha_26} extends to a constant invertible diffusion matrix \(\Sigma\)
			by the change of variables
				\(x\longmapsto \Sigma^{-1}x\).
				Set \(\check X^i=\Sigma^{-1}X^i\), \(\check X=\Sigma^{-1}X\),
				and \(\check\theta_t^n=n^{-1}\sum_{i=1}^n\delta_{\check X_t^i}\).
				The transformed system is
				 \begin{align*}
d\check{X}_t^{i}=\check{b}(t,\check{X}_t^i,\check{\theta}_t^n)dt+dW_t^i,\q\ 
0\leq t\leq T,
				 \end{align*}
				 where for any \((t,x,\mu)\in [0,T] \times \mathbb{R}^d \times \mathcal{P}_2(\mathbb{R}^d)\),
				 \( 
		 \check{b}(t,x,\mu)=\Sigma^{-1} b(t,\Sigma x,\Sigma_{\#} \mu).
				 \)
				 The derivative bounds for \(\check b\) depend additionally on
				 \(\|\Sigma\|_{op}\) and \(\|\Sigma^{-1}\|_{op}\).
				 The law of \(\check X_0\) satisfies a \(T_1\) inequality with constant
				 \(\|\Sigma^{-1}\|_{op}^2C_{T_1}\).
				 Since \(x\longmapsto\Sigma x\) is Lipschitz,
		\begin{align*}
	 \mathcal{W}_{2,\|\cdot\|_{\infty}}^2
	 (\mathbb{P}_{(X^1,\cdots,X^m)}, (\mathbb{P}_{X})^{\otimes m})
	 \leq \|\Sigma\|_{op}^2 \mathcal{W}_{2,\|\cdot\|_{\infty}}^2
	 (\mathbb{P}_{(\check{X}^1,\cdots,\check{X}^m)}, (\mathbb{P}_{\check{X}})^{\otimes m}).
		\end{align*}
		Hence,
		there exists a constant \(C>0\), depending on \(\mathbb{P}_{X_0}, C_{T_1}, 	\|\Sigma\|_{op}, \|\Sigma^{-1}\|_{op}, T\)
		and the bounds on the derivatives of   \(b\), such that
		for all \(1\leq m\leq n,\)  
			\begin{align}\label{sdesharp}
				\mathcal{W}_{2,\|\cdot\|_{\infty}}^2(\mathbb{P}_{(X^1,\cdots,X^m)},(\mathbb{P}_X)^{\otimes m})
			\leq C\frac{m^2}{n^2}.
			\end{align}
		\end{remark}
		
		\begin{lemma}\label{final}\sl 
		Assume that \autoref{asp4.1} and \autoref{asp4.3} hold. 
		Then,
		   for every \(1\leq m\leq n,\) 
			\begin{align}\label{third}
				\mathcal{W}_{2,\|\cdot\|_{\infty}}^2(\mathbb{P}_{(\widehat{Y}^1,\cdots,\widehat{Y}^m)},(\mathbb{P}_{Y})^{\otimes m})
				+	\mathcal{W}_{2,L^2}^2(\mathbb{P}_{(\widehat{Z}^{1,1},\cdots,\widehat{Z}^{m,m})},(\mathbb{P}_{Z})^{\otimes m})
				\leq C\frac{m^2}{n^2},
			\end{align}
			where \(C>0\) depends on \(\mathbb{P}_{X_0}, C_{T_1}, T, L_U, 	\|\Sigma\|_{op}, \|\Sigma^{-1}\|_{op}\)
			and the bounds on the derivatives of  \(b\).
		\end{lemma}
		\begin{proof}
		Let
			\begin{align*}
	&\mathcal{T}_Y: C([0,T];\mathbb{R}^d)\longmapsto C([0,T];\mathbb{R}^k),\q\ 
(	 \mathcal{T}_Y(x))_t=U(t,x_t,\theta_t);\\
	&\mathcal{T}_Z:C([0,T];\mathbb{R}^d)\longmapsto L^2([0,T];\mathbb{R}^{k\times d}),\q\ (\mathcal{T}_Z(x))_t= D_xU(t,x_t,\theta_t)\Sigma.
			\end{align*}
			By the boundedness of \(D_x U\) and \(D_x^2 U\), we have
			\begin{align*}
	\|\mathcal{T}_Y(x)-\mathcal{T}_Y(x')\|_{\infty}
	\leq L_U \|x-x'\|_{\infty}\q\ \hbox{and} \q\ 
		\|\mathcal{T}_Z(x)-\mathcal{T}_Z(x')\|_{L^2}\leq 
		L_{U}\|\Sigma\|_{op}\|x-x'\|_{L^2}.
			\end{align*}
			The representation of the limiting solution gives
			\begin{align*}
		 Y=\mathcal{T}_Y(X) \q\ \hbox{and} \q\ 
		 Z=\mathcal{T}_Z(X).
		 			\end{align*}
			Moreover,
			\begin{align*}
(\mathbb{P}_Y)^{\otimes m}&=(\mathcal{T}_Y^{\otimes m})_{\#}(\mathbb{P}_X)^{\otimes m},\q\ 
\mathbb{P}_{(\widehat{Y}^1,\cdots,\widehat{Y}^m)}
=(\mathcal{T}_Y^{\otimes m})_{\#}\mathbb{P}_{(X^1,\cdots,X^m)},\\
(\mathbb{P}_Z)^{\otimes m}&=(\mathcal{T}_Z^{\otimes m})_{\#}(\mathbb{P}_X)^{\otimes m},\q\ 
\mathbb{P}_{(\widehat{Z}^{1,1},\cdots,\widehat{Z}^{m,m})}=(\mathcal{T}_Z^{\otimes m})_{\#}\mathbb{P}_{(X^1,\cdots,X^m)}.
			\end{align*}
		The Lipschitz bounds above imply
			\begin{align*}
		\mathcal{W}_{2,\|\cdot\|_{\infty}}^2(\mathbb{P}_{(\widehat{Y}^1,\cdots,\widehat{Y}^m)},(\mathbb{P}_{Y})^{\otimes m})
	& \leq L_U^2 
	 \mathcal{W}_{2,\|\cdot\|_{\infty}}^2(\mathbb{P}_{(X^1,\cdots,X^m)},(\mathbb{P}_{X})^{\otimes m});\\
		\mathcal{W}_{2,L^2}^2(\mathbb{P}_{(\widehat{Z}^{1,1},\cdots,\widehat{Z}^{m,m})},(\mathbb{P}_{Z})^{\otimes m})
	& \leq L_{U}^2 \|\Sigma\|_{op}^2 T
	 \mathcal{W}_{2,\|\cdot\|_{\infty}}^2(\mathbb{P}_{(X^1,\cdots,X^m)},(\mathbb{P}_{X})^{\otimes m}).
			\end{align*}
		Combining these estimates with \eqref{sdesharp} completes the proof.
		\end{proof}

	 The three comparisons also control the joint law of the forward and backward paths.
For \(\mathbf x\in C([0,T];(\mathbb R^d)^m)\),
\(\mathbf y\in C([0,T];(\mathbb R^k)^m)\), and
\(\mathbf z\in L^2(0,T;(\mathbb R^{k\times d})^m)\), equip the product space with the metric
\begin{align*}
 d_m^2((\mathbf x,\mathbf y,\mathbf z),(\mathbf x',\mathbf y',\mathbf z'))
 =\|\mathbf x-\mathbf x'\|_\infty^2+
 \|\mathbf y-\mathbf y'\|_\infty^2+\|\mathbf z-\mathbf z'\|_{L^2}^2.
\end{align*}
Write \(\mathcal W_{2,d_m}\) for the resulting Wasserstein distance and set
\begin{align*}
 F_{m,n}=\mathcal W_{2,\|\cdot\|_\infty}^2
 (\mathbb P_{(X^1,\ldots,X^m)},(\mathbb P_X)^{\otimes m}),
 \qquad
 \mathbf Q^{m,n}=((X^i,Y^i,Z^{i,i}))_{i=1}^m.
\end{align*}

\begin{theorem}\label{sharpuperr}\sl
Assume that \autoref{lipp} holds with \(p=q=2\), and that
\autoref{asp4.1} and \autoref{asp4.2} hold. Then, for every \(1\leq m\leq n\) and \(1\leq i\leq n\),
\begin{align*}
 \mathcal W_{2,d_m}^2(\mathbb P_{\mathbf Q^{m,n}},
( \mathbb P_{(X,Y,Z)})^{\otimes m})
 \leq C\left(F_{m,n}+\frac{m}{n^2}\right),
 \qquad
 \mathbb E\int_0^T\sum_{j\ne i}|Z_t^{i,j}|^2dt\leq\frac C{n^2}.
\end{align*}
Here \(C\) depends only on \(d,k,T, |\Sigma|\) and the bounds in these assumptions,
and is independent of \(m,n,i\).
If \autoref{asp4.3} also holds, the squared Wasserstein distance
for the joint law is bounded by
 \(Cm^2/n^2\).
  In particular,
\begin{align}\label{wconver2}
 \mathcal{W}_{2,\|\cdot\|_{\infty}}^2(\mathbb{P}_{(Y^1,\cdots,Y^m)},(\mathbb{P}_{Y})^{\otimes m})
 +\mathcal{W}_{2,L^2}^2(\mathbb{P}_{(Z^{1,1},\cdots,Z^{m,m})},(\mathbb{P}_Z)^{\otimes m})
 \leq C\frac{m^2}{n^2}.
\end{align}
The constant in this last assertion may additionally depend on
\(\mathbb P_{X_0}\),
\(  C_{T_1},   \|\Sigma^{-1}\|_{op}\),
 and the derivative bounds in \autoref{asp4.3}.
\end{theorem}

\begin{proof}
In the couplings used in \autoref{empirical} and \autoref{middle},
the forward processes coincide. Thus each comparison gives a bound
of order \(m/n^2\) for the joint squared Wasserstein distance.
For the last comparison, take any coupling of the forward path laws
and apply the following map to both components:
\begin{align*}
 (\mathbf x_t)_{t\in [0,T]}\longmapsto
 \big((x_t^i,U(t,x_t^i,\theta_t),D_xU(t,x_t^i,\theta_t)\Sigma)_{i=1}^m\big)_{t\in [0,T]}.
\end{align*}
The squared Lipschitz constant of this map is at most
\(1+L_U^2+T L_U^2\|\Sigma\|_{op}^2\).
Taking the infimum over the forward couplings and using the triangle inequality
gives the first estimate.

To prove the off-diagonal bound, the representation in \autoref{asp4.2} gives
\begin{align*}
 D_\mu U(t,x,\theta_t^n)(v)
 =D_aG(t,x,A_t^n)D_v\phi_U(t,v),\qquad
 |D_aG(t,x,A_t^n)|\leq L_G|A_t^n-A_t|.
\end{align*}
The fourth-moment estimate in the proof of \autoref{middle} implies
\(\mathbb E\left[\sup\limits_{t\in [0,T]}|A_t^n-A_t|^2\right]\leq  Cn^{-1}\).
Hence
\(
\mathbb E\int_0^T\sum\limits_{j\neq i}|\widetilde Z_t^{i,j}|^2 dt
\leq   C n^{-2}.
\)
Together with \eqref{empirical-coupling}, 
this gives the same bound
for \(Z\).
Finally, \eqref{sdesharp} yields the sharp bound for the joint law under
\autoref{asp4.3}. The sum of the two marginal transport costs is bounded by
the joint transport cost, which proves \eqref{wconver2}.
\end{proof}

The cancellation condition can also be relaxed.
 If condition (ii) in
\autoref{asp4.2} is replaced by
\begin{align*}
 \sup_{t\in [0,T],x\in \mathbb{R}^d}\Big(|D_aG(t,x,A_t)|+|D_xD_aG(t,x,A_t)|\Big)\leq\varepsilon,
\end{align*}
while all other assumptions and bounds are retained, then
\begin{align*}
 \mathcal W_{2,d_m}^2(\mathbb P_{\mathbf Q^{m,n}},(\mathbb P_{(X,Y,Z)})^{\otimes m})
 &\leq C\left(F_{m,n}+\frac m{n^2}+\frac{\varepsilon^2m}{n}\right),\\
 \mathbb E\int_0^T\sum_{j\ne i}|Z_t^{i,j}|^2dt
 &\leq C\left(\frac1{n^2}+\frac{\varepsilon^2}{n}\right).
\end{align*}
Indeed, the increments of \(G\) and \(D_xG\) in the proof of \autoref{middle} are bounded by
\(C(\varepsilon|A_t^n-A_t|+|A_t^n-A_t|^2)\).
The diagonal correction is bounded by \(Cn^{-1}(\varepsilon+|A_t^n-A_t|)\), since \(D_\mu U=D_aG\,D_v\phi_U\).
The same bound holds for \(|\widetilde Z_t^{i,j}|\) when \(j\neq i\).
The second- and fourth-moment estimates above give the result.
The constant \(C\) is independent of \(\varepsilon,m,n\)
when the other bounds are fixed.

Under \autoref{asp4.3}, \autoref{sharpuperr} gives an upper bound
of order \(m^2/n^2\). The power of \(n\) is independent of the
dimensions, although the constant may depend on them.
	The following two examples attain this rate for the \(Y\)- and
	\(Z\)-components, respectively. In both examples, \(U\) is independent
	of the measure, so the cancellation condition holds automatically.
	Example \ref{exasp4.2} gives models with nontrivial measure dependence.
	
	\begin{example}\label{lowerfory}\rm
	 Let \(d=k=1\), \(\sigma>0,  \tau> 0\).
	 Let \((X_0^{i,n})_{i=1}^n\) be i.i.d. with distribution \(  \mathcal{N}(0,\tau^2)\), independent of the Brownian motions \((W^i)_{i=1}^n\). Set
	 \(b(t,x,\mu)=\int_{\mathbb{R}}v \mu (dv)\) and  
	 consider the \(n\)-particle system:
	 \begin{align*}
	  dX_t^{i,n}=\frac{1}{n}\sum_{j=1}^{n}X_t^{j,n}dt
	  +\sigma dW_t^i,\q\ 0\leq t\leq T.
	 \end{align*}
	The corresponding McKean--Vlasov equation  reduces to 
  \begin{align*}
  dX_t=\sigma dW_t,\q\ 0\leq 
  t\leq T.
  \end{align*}
  Since the system is linear and the initial data are Gaussian,
  for every \(1\leq m\leq n\), 
  \((X_t^{1,n},\cdots,X_t^{m,n})\) is a centered Gaussian vector. 
  A direct covariance calculation gives
  \begin{align*}
  \operatorname{Cov}	(X_t^{1,n},\cdots,X_t^{m,n})
  	=q_t I_m+\frac{s_t}{n}\mathbf{1}_m \mathbf{1}_m^\top,
  \end{align*}
    where 
  \begin{align*}
  	q_t&=\tau^2+\sigma^2 t,\q\ 
  	s_t=\tau^2 (e^{2t}-1)+\sigma^2 \big( \frac{e^{2t}-1}{2}-t\big),\q\
  	\mathbf{1}_m=(1,\cdots,1)^\top,\q\ 
  	I_m=\operatorname{diag}(1, 1, \dots, 1).
  \end{align*}
  Moreover, \((\mathbb{P}_{X_t})^{\otimes m}\)
  is the centered Gaussian distribution with covariance matrix \(q_t I_m\).
 The formula for the Wasserstein distance between Gaussian measures gives
  \begin{align*}
&\  \mathcal{W}_{2}^2(\mathbb{P}_{(X_t^{1,n},\cdots,X_t^{m,n})},(\mathbb{P}_{X_t})^{\otimes m})
  = \Big( \sqrt{q_t+s_t \frac{m}{n}} -\sqrt{q_t}\Big)^2, \q\ 
  0<t\leq T.
  \end{align*}
 	Set
  \begin{align*}
 h(x,\theta)=(1+T)x\q\ 
 \hbox{and} \q\ 
 f(t,x,\theta,y,\mu,z,\nu)
 =-x-(1+t)\int_{\mathbb{R}}v \theta (dv).
  \end{align*}
  The corresponding decoupling field is   \(U(t,x,\mu)=(1+t)x\).
  Since \(D_x^2 U=0\) and all the measure derivatives of \(U\) vanish, the diffusion term and the terms involving measure derivatives in the master equation are zero.
  The remaining terms satisfy
\begin{align*}
&\ \partial_t U(t,x,\theta)+D_x U(t,x,\theta)b(t,x,\theta)
+f(t,x,\theta,U(t,x,\theta),m^U(t,\theta),D_x U(t,x,\theta)\sigma,m^{\mathcal{Z}}(t,\theta))\\
&=x+ (1+t)\int_{\mathbb{R}} v\theta (dv)-x-(1+t)\int_{\mathbb{R}}v \theta (dv)=0,\q\ 
U(T,x,\theta)=(1+T)x=h(x,\theta).
\end{align*}
  By It\^o's formula and  uniqueness of the particle and limiting BSDEs,
  \begin{align*}
Y_t^{i,n}=(1+t)X_t^{i,n},\q\ 
Z_t^{i,j,n}=\mathbf{1}_{\{i=j\}} \sigma (1+t),\q\ 
Y_t=(1+t)X_t,\q\ Z_t=\sigma (1+t).
  \end{align*}
  These coefficients  
   satisfy \autoref{lipp} with \(p=q=2\).
  \(U(t,x,\mu)\) satisfies \autoref{asp4.1}, and  \autoref{asp4.2} holds by taking 
  \(r_U=1, \phi_U=0, G(t,x,a)=(1+t)x,\) and \(B=0\).
  Since \(D_\mu b=1\) and all  higher-order derivatives vanish, \(b(t,\cdot,\cdot)\in C_{bd}^6\). Moreover,  
   \(\mathcal{N}(0,\tau^2)\) satisfies
   the  \(T_1(\tau^2)\) inequality and \(\Sigma=\sigma I_1\).
   Thus 
    \autoref{asp4.3} holds.
For all \(1\leq m\leq n,\)
\begin{align*}
 \mathcal{W}_2^2 (\mathbb{P}_{(Y_t^{1,n},\cdots,Y_t^{m,n})},(\mathbb{P}_{Y_t})^{\otimes m} )
=(1+t)^2 \Big( \sqrt{q_t+s_t \frac{m}{n}} -\sqrt{q_t}\Big)^2
=(1+t)^2 \frac{s_t^2 \frac{m^2}{n^2}}{\Big( \sqrt{q_t+s_t \frac{m}{n}} +\sqrt{q_t}\Big)^2}.
\end{align*}
Hence  
  \begin{align*}
  	  \mathcal{W}_{2,\|\cdot\|_{\infty}}^2 (\mathbb{P}_{(Y^{1,n},\cdots,Y^{m,n})},(\mathbb{P}_{Y})^{\otimes m} )
  	  \geq 
  	   \mathcal{W}_2^2 (\mathbb{P}_{(Y_t^{1,n},\cdots,Y_t^{m,n})},(\mathbb{P}_{Y_t})^{\otimes m} )
  \geq 
  \frac{(1+t)^2 s_t^2 }{\Big( \sqrt{q_t+s_t} +\sqrt{q_t}\Big)^2}
  	  \frac{m^2}{n^2}
  	  \triangleq c_t \frac{m^2}{n^2}.
  \end{align*}
 Furthermore,  since \(Z_t^{i,i,n}=\sigma(1+t)=Z_t\), 
  \begin{align*}
  \mathcal{W}_{2,L^2}^2 (\mathbb{P}_{(Z^{1,1,n},\cdots,Z^{m,m,n})},(\mathbb{P}_{Z})^{\otimes m})=0.
  \end{align*}
Combining the lower bound with the upper bound in  \autoref{sharpuperr},  we  obtain
	\begin{align*}
c_T  \frac{m^2}{n^2}\leq 	 \mathcal{W}_{2,\|\cdot\|_{\infty}}^2 (\mathbb{P}_{(Y^{1,n},\cdots,Y^{m,n})},(\mathbb{P}_{Y})^{\otimes m} ) 
\leq C_T \frac{m^2}{n^2},\q\ 
	 1\leq m\leq n,
	\end{align*}
	where \(C_T\) is a constant from \eqref{wconver2} and \(0<c_T\leq C_T<\infty\).
Thus the order \(m^2/n^2\) in \eqref{wconver2} is optimal
and is attained by the \(Y\)-component.
 	\end{example}

 	\begin{example}\label{lowerforz}\rm
 			 Let \(d=k=1\), \(\lambda, \sigma, \tau> 0\).
 		Let \((X_0^{i,n})_{i=1}^n\) be i.i.d. with distribution \(  \mathcal{N}(0,\tau^2)\), independent of the Brownian motions \((W^i)_{i=1}^n\). Set
 		\(b(t,x,\mu)=\lambda \tanh \Big( \int_{\mathbb{R}}v \mu (dv)\Big) \)
 		and  
 		consider the following \(n\)-particle system:
 		\begin{align*}
 			dX_t^{i,n}=\lambda 
 		\tanh \Big( \frac{1}{n}\sum_{j=1}^{n}X_t^{j,n} \Big) dt
 			+\sigma dW_t^i,\q\ 0\leq t\leq T.
 		\end{align*}
 		The corresponding McKean--Vlasov equation  is 
 		\begin{align*}
 			dX_t=\sigma dW_t,\q\ 0\leq 
 			t\leq T.
 		\end{align*}
 		  For every \((t,x,\theta,y,\mu,z,\nu)\in 
 		[0,T]\times \mathbb{R}\times \mathcal{P}_2(\mathbb{R})
 		\times \mathbb{R}\times \mathcal{P}_2(\mathbb{R})
 		\times \mathbb{R}\times \mathcal{P}_2(\mathbb{R}),\) let
 		\begin{align*}
 			h(x,\theta)=-\cos x \q\ 
 			\hbox{and} \q\ 
 			f(t,x,\theta,y,\mu,z,\nu)
 			=-\lambda \sin x \tanh \Big(\int_{\mathbb{R}}v \theta (dv)\Big)
 			-\frac{\sigma^2}{2}\cos x.
 		\end{align*}
 		The corresponding decoupling field is \(U(t,x,\mu)=-\cos x\).
 		Since \(D_x U=\sin x\)
 		and  \(D_x^2 U=\cos x\),
 		\begin{align*}
 	 &\ D_x U(t,x,\theta)b(t,x,\theta)+\frac{\sigma^2}{2}D_x^2 U (t,x,\theta) +f(t,x,\theta,U(t,x,\theta), m^{U}(t,\theta),\sigma D_x U(t,x,\theta), m^{\mathcal{Z}}(t,\theta))\\
 	 &=\lambda \sin x \tanh \Big(\int_{\mathbb{R}}v \theta (dv)\Big)
 	 +\frac{\sigma^2}{2}\cos x
 	 -\lambda \sin x \tanh \Big(\int_{\mathbb{R}}v \theta (dv)\Big)
 	 -\frac{\sigma^2}{2}\cos x=0.
 		\end{align*}
A direct  verification shows that	\(U\) satisfies \autoref{asp4.1}, and \autoref{asp4.2} holds with \(r_U=1\), \(\phi_U=0,\) 
\(G(t,x,a)=-\cos x\), \(B=0\).
Moreover, the coefficients \(b, f, h\) satisfy \autoref{lipp} with \(p=q=2\), and  \autoref{asp4.3} holds with 
\(b\in C_{bd}^6\), \(C_{T_1}=\tau^2\),  \(\Sigma=\sigma I_1\).
It\^o's formula and
  uniqueness of the particle and limiting BSDEs yield
\begin{align*}
 Y_t^{i,n}=-\cos (X_t^{i,n}),\q\ 
 Z_t^{i,j,n}=\mathbf{1}_{\{i=j\}}\sigma \sin (X_t^{i,n}),\q\ 
 Y_t=-\cos (X_t),\q\ Z_t=\sigma \sin(X_t).
\end{align*}
Set
\begin{align*}
	&q_t=\tau^2+\sigma^2t,\q\ 
	r_t=\tau^2 e^{2\lambda t}+\frac{\sigma^2}{2\lambda}
	(e^{2\lambda t}-1), \q\ 
	M^n_t=\frac{1}{n}\sum_{j=1}^{n}X_t^{j,n},
	\q\ N_t^n=\sqrt{n}M_t^n,\\
&	Q_{s,t}= \tau^2+\sigma^2(s\wedge t),\q\ 
	R_{s,t} =\tau^2 e^{\lambda(s+t)}
	+\sigma^2\int_0^{s\wedge t} e^{\lambda (s-u)} e^{\lambda (t-u)} du,\q\ 
	B_t^n=\frac{1}{\sqrt{n}}\sum_{j=1}^{n}W^j_t.
\end{align*}
Since \((X^{1,n},\cdots,X^{n,n})
\overset{d}{=} (-X^{1,n},\cdots, -X^{n,n})\), 
\(\mathbb{E}\left[\sin (X_t^{i,n})\right]=0\).
 Let \(D_t^{i,n}=X_t^{i,n}-M_t^n\).
 The common drift cancels in \(D^{i,n}\), so
 \begin{align*}
 D_t^{i,n}
 =X_0^{i,n}-M_0^n
 +\sigma\Big(W_t^i-\frac1n\sum_{j=1}^nW_t^j\Big).
 \end{align*}
 Here
 the vector \((X_0^{i,n}-M_0^n)_{i=1}^n\) and the process
 \((W^i-\frac 1n \sum\limits_{j=1}^nW^j)_{i=1}^n\) are jointly Gaussian
 and independent of \((M_0^n,B^n)\).
 Since \(M^n\) is determined by \((M_0^n,B^n)\),
 the processes \(D^{1,n},\cdots,D^{n,n}\) are centered Gaussian
 and independent of \(M^n\). Hence
\begin{align*}
	\mathbb{E}\left[
	D_s^{i,n} D_t^{j,n}\right]
	=\Big(\mathbf{1}_{\{i=j\}} -\frac{1}{n}\Big)Q_{s,t}.
\end{align*}
Moreover,
\begin{align*}
 N_t^n=N_0^n+\int_0^t \lambda \sqrt{n}
 \tanh (\frac{N_s^n}{\sqrt{n}}) ds +\sigma B_t^n,
\end{align*}
where \(N_0^n \sim \mathcal{N}(0,\tau^2)\) is independent of the
standard Brownian motion
 \(B^n\).
For each \(n\), define on the same space
\begin{align*}
	d\bar N_t^n=\lambda\bar N_t^n dt+\sigma dB_t^n
	\q\ \hbox{and} \q\  \bar N_0^n=N_0^n.
\end{align*}
The law of \(\bar N^n\) is independent of \(n\), with covariance
\(R_{s,t}\). Since \(\sqrt n|\tanh(\frac{x}{\sqrt{n}})|\leq|x|\), 
 for each
finite \(r\geq2\),
standard moment
estimates give 
\begin{align*}
 \sup\limits_{n\geq 1}\mathbb{E}\left[ \sup\limits_{t\in [0,T]} |N_t^n|^r\right]<\infty.
\end{align*}
The bound \(|\tanh u-u|\leq C|u|^3\) and Gronwall's
inequality then yield
\begin{align*}
 \mathbb{E}\left[
 \sup\limits_{t\in [0,T]}|N_t^n-\bar{N}^n_t|^2\right]
 \leq 
 \frac{C}{n^2}.
\end{align*}
In particular,
\(\mathbb E\left[N_s^nN_t^n\right]\longrightarrow R_{s,t}\) uniformly on \([0,T]^2\),
and the fourth moments are uniformly bounded.
The limiting process satisfies \(X_t\sim\mathcal N(0,q_t)\)
and \(\operatorname{Cov}(X_s,X_t)=Q_{s,t}\).
   Taylor's formula and the uniform fourth-moment bound give
\begin{align*}
	 \mathbb{E}\left[ \cos (X_s^{i,n} \pm X_t^{j,n})\right]
	 =& \exp\left\{- \frac{1}{2} \operatorname{Var} (D_s^{i,n}\pm D_t^{j,n}) \right\} 
	 \mathbb{E}\left[ \cos \big(\frac{N_s^n\pm N_t^n}{\sqrt{n}} \big) \right]\\
	 =&\exp\left\{- \frac{1}{2} \operatorname{Var} (D_s^{i,n}\pm D_t^{j,n}) \right\} 
	 \Big( 1-\frac{1}{2n}\mathbb{E} \left[|N_s^n\pm N_t^n|^2\right] +O(\frac{1}{n^2})\Big).
\end{align*}
Using \(2\sin x\sin y=\cos(x-y)-\cos(x+y)\), we obtain
 \begin{align*}
 	&n \operatorname{Cov}(\sin X_s^{i,n},\sin X_t^{j,n}) \longrightarrow 
 \kappa_{s,t},  
 \q\ i\neq j,\\
  &n	\Big(\operatorname{Cov}(\sin X_s^{i,n},\sin X^{i,n}_t)
  -	\operatorname{Cov}(\sin X_s,\sin X_t)\Big)
  \longrightarrow   \widetilde{\kappa}_{s,t},
 \end{align*}
 where
 \begin{align*}
 &\kappa_{s,t}=e^{-\frac{q_s+q_t}{2}} (R_{s,t}-Q_{s,t}),\\
& \widetilde{\kappa}_{s,t}
 =\frac{1}{4}\Big( e^{-\frac{A_+(s,t)}{2}}(B_+(s,t)-A_+(s,t))
 -e^{-\frac{A_-(s,t)}{2}}(B_-(s,t)-A_-(s,t))
 \Big),\\
 &A_{\pm}(s,t)=q_s+q_t\pm2Q_{s,t}\q\  \hbox{and} \q\ 
 B_{\pm}(s,t)=r_s+r_t\pm2R_{s,t}.
 \end{align*}
 The convergence is uniform on compact subsets of \((0,T)^2\).
In particular, \(\kappa_{t,t}=e^{-q_t}(r_t-q_t)>0\),
\(\widetilde{\kappa}_{t,t}=e^{-2q_t}(r_t-q_t)>0\).
Fix \(t_0\in(0,T)\). By continuity, there exists a closed interval
\(I\subset(0,T)\) of positive length containing \(t_0\)
such that both kernels are positive on \(I^2\).
Set
\begin{align*}
 \mathcal{V}^{i,n}=\frac{\sigma}{\sqrt{|I|}}\int_{I}\sin (X_t^{i,n}) dt\q\ \hbox{and}\q\ 
 \mathcal{V}=\frac{\sigma}{\sqrt{|I|}}\int_{I}\sin (X_t)dt.
\end{align*} 
By Fubini's theorem,
\begin{align*}
\operatorname{Var} (\mathcal{V}^{i,n})&=\frac{\sigma^2}{|I|}\int_{I} \int_{I}
 \operatorname{Cov} (\sin X_s^{i,n},\sin X_t^{i,n}) dsdt\triangleq v_n,\\
 \operatorname{Cov}(\mathcal{V}^{i,n},\mathcal{V}^{j,n})&=
 \frac{\sigma^2}{|I|}\int_{I} \int_{I}  \operatorname{Cov} (\sin X_s^{i,n},\sin X_t^{j,n}) dsdt \triangleq \hat{v}_n,\q\ i\neq j,\\
\operatorname{Var}(\mathcal{V})&=\frac{\sigma^2}{|I|}\int_{I} \int_{I}
\operatorname{Cov} (\sin X_s,\sin X_t)dsdt\triangleq v>0.
\end{align*}
For all \(1\leq m\leq n\),
\begin{align*}
 \operatorname{Cov}(\mathcal{V}^{1,n},\cdots,\mathcal{V}^{m,n})
 =(v_n-\hat{v}_n)I_m+\hat{v}_n \mathbf{1}_m \mathbf{1}_m^\top\q\ \hbox{and}\q\ 
 \operatorname{Cov}(\bar{\mathcal{V}}^1,\cdots,\bar{\mathcal{V}}^m)=v I_m,
\end{align*}
where \(\mathbb{P}_{(\bar{\mathcal{V}}^1,\cdots,\bar{\mathcal{V}}^m)}=(\mathbb{P}_\mathcal{V})^{\otimes m}\).
By
Gelbrich's inequality (see Gelbrich \cite[Theorem 2.1]{gelbrich_90}),
\begin{align*}
 \mathcal{W}_2^2( \mathbb{P}_{(\mathcal{V}^{1,n},\cdots,\mathcal{V}^{m,n})},
 (\mathbb{P}_{\mathcal{V}})^{\otimes m}
 )
 &\geq
 \Big(\sqrt{v_n +(m-1)\hat{v}_n} -\sqrt{ v}\Big)^2+(m-1)\Big( \sqrt{v_n-\hat{v}_n} -\sqrt{v}\Big)^2\\
 &=
 \frac{(v_n-v+(m-1)\hat{v}_n)^2}{\Big(\sqrt{v_n +(m-1)\hat{v}_n} +\sqrt{ v}\Big)^2}
 +(m-1)\frac{(v_n-v-\hat{v}_n)^2}{\Big( \sqrt{v_n-\hat{v}_n} +\sqrt{v}\Big)^2}.
\end{align*}
By the choice of \(I\), the following constants are positive:
\begin{align*}
	 c_I=\frac{\sigma^2}{|I|}\int_{I} \int_{I}
	 \widetilde{\kappa}_{s,t}dsdt>0\q\ \hbox{and}\q\ 
	 \hat{c}_I=\frac{\sigma^2}{|I|}\int_{I} \int_{I}
	 \kappa_{s,t}dsdt>0.
\end{align*}
By the uniform convergence on \(I^2,\)
\begin{align*}
	v_n-v=
 \frac{c_I}{n}+ o(\frac{1}{n})
 \q\ \hbox{and} \q\ 
 \hat{v}_n=
 \frac{\hat{c}_I}{n}+o(\frac{1}{n}).
\end{align*}
 Hence there exists \(n_0\in \mathbb{N}\) such 
 that,
for  every  \(n\geq n_0,\)
\(v_n-v\geq \frac{c_I}{2n}\) and \(\hat{v}_n\geq \frac{\hat{c}_I}{2n}\).
Moreover, there exists
  a constant \(K>0\),   independent of  \(m\) and \(n\), such that
\begin{align*}
 \Big(\sqrt{v_n +(m-1)\hat{v}_n} +\sqrt{ v}\Big)^2\leq K \q\ \hbox{and}\q\ 
 \Big( \sqrt{v_n-\hat{v}_n} +\sqrt{v}\Big)^2\leq K.
\end{align*}
Therefore,
\begin{align*}
  \mathcal{W}_2^2( \mathbb{P}_{(\mathcal{V}^{1,n},\cdots,\mathcal{V}^{m,n})},
 (\mathbb{P}_{\mathcal{V}})^{\otimes m}
 )&\geq
 \frac{m(v_n-v)^2+m(m-1)\hat{v}_n^2}{K} \\
 &\geq \frac{mc_I^2+m(m-1)\hat{c}_I^2 }{4n^2K}\\
 &\geq 
 \frac{\min \{c_I^2,\hat{c}_I^2\}}{4K}
 \frac{m^2}{n^2}\triangleq c_z\frac{m^2}{n^2}.
\end{align*}
Define 
\begin{align*}
\Lambda_I :L^2([0,T])\longmapsto \mathbb{R},\q\
\Lambda_I(z)=\frac{1}{\sqrt{|I|}}\int_{I} z_t dt.
\end{align*}
For any \(z, z'\in L^2([0,T])\), the Cauchy--Schwarz inequality gives 
\begin{align*}
	|\Lambda_I(z)-\Lambda_I(z')|
	\leq \| z-z'\|_{L^2([0,T])}.
\end{align*}
Since 
\(\mathcal{V}^{i,n}=\Lambda_I(Z^{i,i,n})\),  \(\mathcal{V}=\Lambda_I(Z)\),
\begin{align*}
	  \mathcal{W}_2^2( \mathbb{P}_{(\mathcal{V}^{1,n},\cdots,\mathcal{V}^{m,n})},
	 (\mathbb{P}_{\mathcal{V}})^{\otimes m})
	&= \mathcal{W}_2^2( \mathbb{P}_{(\Lambda_I(Z^{1,1,n}),\cdots,\Lambda_I(Z^{m,m,n}))},
	(\mathbb{P}_{\Lambda_I(Z)})^{\otimes m}
	)\\
&	\leq 
	 \mathcal{W}_{2,L^2}^2 (\mathbb{P}_{(Z^{1,1,n},\cdots,Z^{m,m,n})},(\mathbb{P}_{Z})^{\otimes m}).
\end{align*}
Combining this 
 lower bound with the upper bound in \autoref{sharpuperr}, we obtain
\begin{align*}
c_z \frac{m^2}{n^2}
\leq  \mathcal{W}_{2,L^2}^2 (\mathbb{P}_{(Z^{1,1,n},\cdots,Z^{m,m,n})},(\mathbb{P}_{Z})^{\otimes m})\leq C_z \frac{m^2}{n^2},\q\ 
n\geq n_0, \ 1\leq m\leq n,
\end{align*}
where \(C_z>0\) is a constant from \eqref{wconver2} and
\(0<c_z\leq C_z <\infty\).
Thus the order \(m^2/n^2\) in \eqref{wconver2} is also optimal
for the \(Z\)-component.
 	\end{example}
 	
\section{Weak propagation of chaos}\label{sec5}

Throughout this section, the coefficients \(b,\sigma,f,h\)
are deterministic and Markovian, and \autoref{lipp} holds
with \(p=q=2\).
In particular, \(\mathbb{E}\left[|X_0|^2\right]<\infty.\)
The diffusion coefficient \(\sigma(t,x,\mu)\) may depend on
\(t,x,\mu\), and \(\sigma\sigma^\top\) may be degenerate.
The dimensions \(d,k\) are fixed. Throughout this section,
\(C>0\) denotes a constant that may vary from line to line
and may depend on \(d,k\).
Set
\(\theta_t=\mathbb{P}_{X_t}\), \(\theta_t^n=\frac{1}{n}\sum_{j=1}^{n}\delta_{X_t^j}\).
The forward weak error estimates used below require only
a finite second moment of the initial law.

We first prove a weak error estimate of order \(n^{-1}\)
for \((\widetilde Y,\widetilde Z)\).
We then compare these processes with the solutions of the particle BSDEs
and extend the estimates to fixed \(m\)-particle marginals.

  We measure weak errors using the following test function classes.
For \(r\in\mathbb N\), let \(C_b^{3,1}(\mathbb R^r)\)
denote the space of functions \(f:\mathbb R^r
\longmapsto 
\mathbb R\)
such that \(f,Df,D^2f,D^3f\) are continuous and bounded,
and \(D^3f\) is globally Lipschitz.
Equip this space with the norm
\begin{align*}
	\|f\|_{3,1}
	\triangleq 
	\sum_{l=0}^{3}\|D^l f\|_{\infty}
	+[D^3 f]_{Lip},
\end{align*}
where \(\|D^l f\|_{\infty}= \sup\limits_{x\in \mathbb{R}^r}\|D^l f(x)\|\) (with the  Euclidean tensor norm) and \([D^3 f]_{Lip}=\sup\limits_{x\neq y}
\frac{\| D^3 f(x)-D^3 f(y)\|}{\|x-y\|}\).
Define the test classes
\begin{align*}
	\mathcal{F}_{3,1}
	=\{ \varphi \in C_b^{3,1}(\mathbb{R}^k): \|\varphi\|_{3,1}\leq 1\}\q\ 
	\hbox{and} \q\ 
	\mathcal{G}_{3,1}
	=\{ \psi \in C_b^{3,1}(\mathbb{R}^{k\times d}): \|\psi\|_{ 3,1}\leq 1\},
\end{align*}
We use \(\mathcal F_{3,1}\) for the value processes and
\(\mathcal G_{3,1}\) for the diagonal martingale integrands.
We identify \(\mathbb{R}^{k\times d}\) with \(\mathbb{R}^{kd}\) and equip it with the Frobenius norm.
For any \(\mu,\mu' \in \mathcal{P}_2(\mathbb{R}^k)\), 
\(\nu, \nu' \in \mathcal{P}_2(\mathbb{R}^{k\times d})\),
define the integral probability metrics
\begin{align*}
	&d_{3,1}^Y (\mu,\mu')=\sup\limits_{\varphi \in \mathcal{F}_{3,1}} \Big|\int_{\mathbb{R}^k} \varphi (y)\mu(dy)
	-\int_{\mathbb{R}^k} \varphi(y) \mu'(dy) \Big|,\\
	&d_{3,1}^Z(\nu,\nu')
	=\sup\limits_{\psi \in \mathcal{G}_{3,1}}
	\Big| \int_{\mathbb{R}^{k\times d}} \psi(z) \nu(dz)
	-\int_{\mathbb{R}^{k\times d}} \psi(z) \nu'(dz)    \Big|.
\end{align*}
These metrics are defined on all probability measures and metrize
weak convergence. Indeed, convergence in either metric implies
convergence against \(C_c^\infty\), hence vague convergence.
Since the limiting measure is a probability measure, this gives
weak convergence. Conversely, weak convergence implies convergence
in these metrics by comparison with the bounded-Lipschitz distance.
With \(d_{BL}\) defined by the constraint
\(\|g\|_\infty+[g]_{Lip} \leq1\), we have
\(d_{3,1}\leq d_{BL}\leq\mathcal W_1\leq\mathcal W_2\) whenever the
corresponding moments are finite. 
Each test class contains a countable subset that is dense for
uniform convergence on compact sets.
The supremum defining each metric can be taken over this subset.
Consequently, the distances appearing in the time integrals
below are measurable functions of time.

To obtain a weak rate of order \(n^{-1}\), we estimate
\(\mathbb E\left[\Phi(\theta_t^n)\right]-\Phi(\theta_t)\)
uniformly over the functionals associated with these test classes.
We first introduce the regularity class used in the semigroup argument.
We follow \cite[Definitions 2.16--2.17]{chassagen_22_aap}
for the spatial and measure regularity, with bounds uniform
in time. We also require continuous differentiability in time.
Let \(w,l\in\mathbb N_0\),
\(\overline c=(c_1,\ldots,c_w)\in\mathbb N_0^w\),
and \(\alpha=(l,\overline c)\), with
\(|\alpha|=l+\sum_{j=1}^w c_j\).
For a scalar-valued function
\(u:\mathbb R^d\times\mathcal P_2(\mathbb R^d)
\longmapsto
\mathbb R\),
define
\begin{align*}
	\mathcal{D}^{w,\alpha}u(x,\mu,v_1,\cdots,v_w)
	\triangleq \triangledown_{v_1}^{c_1} \cdots \triangledown_{v_w}^{c_w}
	 \triangledown_x^l D_\mu^w u(x,\mu,v_1,\cdots,v_w).
\end{align*}
Here the measure derivatives are taken before the spatial
derivatives in \(x\), as in \cite{chassagen_22_aap}.
This order differs from that used for \(D^{w,\alpha}\)
in \autoref{sec4}.
When \(w=0\), we use the convention
\begin{align*}
	\mathcal{D}^{0,l} u \triangleq \triangledown_x^l u \q\ 
	\hbox{and}\q\ 
	\mathcal{D}^{0,0}u=u.
\end{align*} 
As in \cite[Definition 2.4]{chassagen_22_aap},
boundedness is required only for derivatives of positive total order.

\begin{definition}\label{mclass}\rm
For	\(r\in \mathbb{N}\), 
 a scalar-valued function
 \(u: [0,T] \times \mathbb{R}^d \times \mathcal{P}_2(\mathbb{R}^d)
 \longmapsto \mathbb{R}\) belongs to \(\mathcal{M}_r([0,T]
 	\times \mathbb{R}^d\times \mathcal{P}_2 (\mathbb{R}^d) )\) if the following  hold:
 	\begin{enumerate}[(1)]
\item For every \((x,\mu)\in \mathbb{R}^d\times \mathcal{P}_2(\mathbb{R}^d)\), the map
 \(t \longmapsto u(t,x,\mu)\) is continuously differentiable.
 \item  For every multi-index satisfying \(0< w+|\alpha |\leq r\), the derivative
  \(\mathcal{D}^{w,\alpha}u\) exists,  is uniformly bounded,  
  and is globally Lipschitz in \((x,\mu,v_1,\cdots,v_w)\), with constants independent of \(t\).
\item  For every multi-index satisfying \( 0\leq w+|\alpha |\leq r,\)
\((t,x,\mu,v_1,\cdots,v_w)\longmapsto
\mathcal{D}^{w,\alpha}u(t,x,\mu,v_1,\cdots,v_w)\) is 
jointly continuous.
 	\end{enumerate}
\end{definition}

Define
\begin{align*} 
	&	\|u\|_{\mathcal{M}_r}
	\triangleq \sup\limits_{t\in [0,T]}|u(t,0,\delta_0)| 
	+	\max\limits_{  0<w+|\alpha|\leq r} \sup\limits_{t\in [0,T]}
	\Big\{ \| \mathcal{D}^{w,\alpha} u(t)\|_{\infty} +[\mathcal{D}^{w,\alpha} u(t)]_{Lip} \Big\},
\end{align*}
where
\begin{align*}
	[\mathcal{D}^{w,\alpha}u(t)]_{Lip}
	=\sup\limits_{(x,\mu,v_1,\cdots,v_w)
	\neq (x',\mu',v_1',\cdots,v_w')}\frac{\big| \mathcal{D}^{w,\alpha}u(t,x,\mu,v_1,\cdots,v_w)
		-\mathcal{D}^{w,\alpha}u(t,x',\mu',v_1',\cdots,v_w')\big| }{|x-x'|
		+\mathcal{W}_2(\mu,\mu')+\sum_{j=1}^{w}|v_j-v_j'|}.
\end{align*}
The definition and norm extend to functions of several Euclidean
and measure variables by including all mixed derivatives
of the prescribed total order and using the sum of the
Euclidean and \(\mathcal W_2\) distances.
For vector- or matrix-valued maps, the \(\mathcal{M}_r\)-norm is defined as the maximum of the
corresponding componentwise norms.
The time-independent class \(\mathcal{M}_r(\mathbb{R}^d \times \mathcal{P}_2(\mathbb{R}^d))\) is defined analogously by omitting the time variable and condition (1) of \autoref{mclass}.
A functional   on \(  \mathcal{P}_2(\mathbb{R}^d)\) is identified with its  
extension   independent of \(x\):
\begin{align*}
	u\in \mathcal{M}_r(  \mathcal{P}_2(\mathbb{R}^d)) \Longleftrightarrow
	\widetilde{u}(x,\mu) \triangleq u(\mu)\in \mathcal{M}_r(  \mathbb{R}^d \times \mathcal{P}_2(\mathbb{R}^d)).
\end{align*}
Similarly,
a functional defined only on \([0,t]\times \mathcal{P}_2(\mathbb{R}^d)\) is identified with  \((s,x,\mu)\longmapsto u(s,\mu)\):
\begin{align*}
	 u\in \mathcal{M}_r([0,t]\times \mathcal{P}_2(\mathbb{R}^d)) \Longleftrightarrow
	 \widetilde{u}(s,x,\mu) \triangleq u(s,\mu)\in \mathcal{M}_r([0,t]\times \mathbb{R}^d \times \mathcal{P}_2(\mathbb{R}^d)).
\end{align*}

We impose the following regularity assumptions on the forward coefficients.

\begin{assumption}\label{5forsde}\sl 
	\begin{enumerate}[(i)]
		\item 
		\(b, \sigma \in \mathcal{M}_3([0,T]\times \mathbb{R}^d \times \mathcal{P}_2(\mathbb{R}^d))\). 
		\item \(  \|\sigma\|_{\infty} \triangleq  \sup\limits_{(t,x,\mu)\in [0,T]\times \mathbb{R}^d \times \mathcal{P}_2(\mathbb{R}^d)} | \sigma(t,x,\mu)| <\infty. \)
	\end{enumerate}
\end{assumption}

To transfer forward weak error estimates to the backward
processes, we use a classical solution \(U\) of the following
master equation.
\begin{equation}\label{master2}
	\begin{aligned}
		&\  \partial_tU(t,x,\mu) +  D_x U(t,x,\mu) b(t,x,\mu) +
		\frac{1}{2}
		\operatorname{Tr} \Big(\sigma(t,x,\mu) \sigma^{\top}(t,x,\mu) D^2_x U(t,x,\mu) \Big)  \\
		&	+\int_{\mathbb{R}^d} D_{\mu} U(t,x,\mu)(\nu)
		b(t,\nu,\mu)\mu (d\nu)
		+\frac{1}{2}\int_{\mathbb{R}^d } \operatorname{Tr}\Big(\sigma(t,\nu,\mu) \sigma^{\top}  (t,\nu,\mu) D_{\nu} D_{\mu}
		U(t,x,\mu)(\nu) \Big) \mu (d\nu) \\
		& 	+f\Bigg(t,x,\mu,U(t,x,\mu),m^{U}(t,\mu),
		D_x U(t,x,\mu)\sigma(t,x,\mu),m^{\mathcal{Z}}(t,\mu)\Bigg)=0,\\
		&   U(T,x,\mu)=h(x,\mu),
	\end{aligned}
\end{equation}
where 
\begin{align*}
	m^{U}(t,\mu)=\Big(U(t,\cdot,\mu) \Big)_{\# \mu}
	\q\ \hbox{and} \q\ 
	m^{\mathcal{Z}}(t,\mu)=\Big(D_x U(t,\cdot,\mu) \sigma(t,\cdot,\mu)\Big)_{\# \mu}.
\end{align*}

We impose the following conditions on \(U\).

\begin{assumption}\label{asp5.2}\sl 
	The master equation \eqref{master2} admits a classical solution \(U: [0,T]\times \mathbb{R}^d \times \mathcal{P}_2(\mathbb{R}^d)
	\longmapsto \mathbb{R}^k\) satisfying the following conditions:
	\begin{enumerate}[(i)]
		\item The function \(U\) and the derivatives
		\(    
		\partial_tU, D_xU,  D_x^2U, D_\mu U,
		D_\nu D_\mu U,  D_xD_\mu U,  D_\mu D_x U, D_\mu^2U\) exist and are jointly continuous in their respective variables, and 
		\(D_xD_\mu U= D_\mu D_x U \);
		\item 
		There exists a constant \(L_U>0\) such that,
		for all \((t,x,\mu,\nu,\nu')\in 
		[0,T]\times \mathbb{R}^d \times \mathcal{P}_2(\mathbb{R}^d)
		\times \mathbb{R}^d\times \mathbb{R}^d,
		\)
		\begin{align*}
			&\ |D_x U(t,x,\mu)|
			+|D_{x}^2 U(t,x,\mu)|
			+|D_\mu U(t,x,\mu)(\nu)|
			+|D_\nu D_\mu U(t,x,\mu)(\nu)|\\
			& +|D_x D_\mu U(t,x,\mu)(\nu)|+|D_\mu D_x U(t,x,\mu)(\nu)|
			+|D_{\mu}^2 U(t,x,\mu)(\nu,\nu')|\leq L_U;
		\end{align*}
		\item  The functions
		\(U\) and \( \partial_t U \)  have at most  linear growth:
		\begin{align*}
			|U(t,x,\mu)|+|\partial_t U(t,x,\mu)|
			\leq 
			L_U \Bigg( 1+|x|+\Big( \int_{\mathbb{R}^d} |\nu|^2 \mu (d\nu) \Big)^{\frac{1}{2}} \Bigg);
		\end{align*}
		\item  For every \((x,\mu)\in \mathbb{R}^d\times \mathcal{P}_2(\mathbb{R}^d)\),
		\(U(T,x,\mu)=h(x,\mu)\).
	\end{enumerate}
	All derivatives of the vector-valued function \(U\), as well as
	all trace operations involving these derivatives, are understood componentwise.
\end{assumption}

\autoref{exasp4.2} provides a class of models satisfying
\autoref{asp4.1}.
The following proposition gives sufficient conditions for
\autoref{asp5.2} and, when \(\sigma\equiv\Sigma\), for
\autoref{asp4.1}.
Its proof is given in \autoref{proofoflemma}.

\begin{proposition}\label{regular-proof-U}\sl
	Assume that \autoref{5forsde} holds and that the following
	conditions are satisfied.
	\begin{enumerate}[(i)]
		\item
		\(h\in C_{bd}^3(\mathbb R^d\times\mathcal P_2(\mathbb R^d))\).
\item
The function \(f\) and its mixed spatial and Lions derivatives
of positive total order at most three are jointly continuous
in their respective variables, including \(t\).
Moreover, for every \(t\in[0,T]\),
\begin{align*}
	f(t,\cdot)\in\mathcal M_3\Big(
	\mathbb R^d\times\mathcal P_2(\mathbb R^d)
	\times\mathbb R^k\times\mathcal P_2(\mathbb R^k)
	\times\mathbb R^{k\times d}
	\times\mathcal P_2(\mathbb R^{k\times d})
	\Big),
\end{align*}
and
\(
\sup\limits_{t\in[0,T]}
\|f(t,\cdot)\|_{\mathcal M_3}<\infty.
\)
	\end{enumerate}
	Then the master equation \eqref{master2} admits a classical
	solution \(U\) satisfying the regularity and growth conditions
	in \autoref{asp5.2}.
\end{proposition}

 Let \(U\) be a classical solution of \eqref{master2} satisfying \autoref{asp5.2}.
 As in \autoref{sec4}, define
  \begin{align*}
  	\widetilde{Y}_t^i=U(t,X_t^i,\theta_t^n) \q\ 
  	\hbox{and} \q\ 
  	\widetilde{Z}_t^{i,j}=\mathbf{1}_{\{i=j\}}
  	D_xU(t,X^i_t,\theta_t^n)\sigma(t,X_t^i,\theta_t^n)
  	+\frac{1}{n}  D_{\mu}U(t,X_t^i,\theta_t^n)(X^j_t)\sigma(t,X_t^j,\theta_t^n).
  \end{align*}
Set \(  H_Z(t,x,\mu)\triangleq D_x U(t,x,\mu)\sigma(t,x,\mu) \in \mathbb{R}^{k\times d}
\).
For every \(t\in [0,T]\), \(\varphi\in \mathcal{F}_{3,1}\) and 
\(\psi \in \mathcal{G}_{3,1}\), 
the maps \(\Phi_{t,\varphi}^Y, \Phi_{t,\psi}^Z :
\mathcal{P}_2(\mathbb{R}^d) \longmapsto \mathbb{R}\) are defined by
\begin{align*}
	&\Phi_{t,\varphi}^Y (\mu) 
	=\int_{\mathbb{R}^d} \varphi\big( U(t,x,\mu)\big)\mu (dx)\q\ \hbox{and}
	\q\ \Phi_{t,\psi}^Z (\mu) 
	=\int_{\mathbb{R}^{ d}} 
	\psi\big( H_Z(t,x,\mu) \big)\mu (dx).
\end{align*}
The \(n^{-1}\) correction in \(\widetilde Z_t^{i,i}\)
will be estimated separately.

\begin{remark}\label{z=hz}\rm
	Under  the standing assumptions of this section and  \autoref{asp5.2},
	the Lions-It\^o formula applied to \(U(t,X_t,\theta_t)\),
	together with the master equation \eqref{master2}, shows that
	\( \big( U(t,X_t,\theta_t), H_Z(t,X_t,\theta_t)\big)
	\)
	satisfies the  limiting  mean-field BSDE \eqref{y1n}.
	Moreover,
	\begin{align*}
		m^{U}(t,\theta_t)=\mathbb{P}_{U(t,X_t,\theta_t)} \q\ 
		\hbox{and} \q\ 
		m^{\mathcal{Z}}(t,\theta_t)
		=\mathbb{P}_{H_Z(t,X_t,\theta_t)}.
	\end{align*}
	By uniqueness of the limiting BSDE,
	\(Y_t=U(t,X_t,\theta_t)\) almost surely for all \(t\in[0,T]\).
	Set \(Z_t^*=H_Z(t,X_t,\theta_t)\). 
	Then
	\(Z^*=Z\), \(dt\otimes d\mathbb P\)-almost everywhere.
	Fixed-time laws of the limiting integrand always refer to the
	Markov version \(Z_t^*=H_Z(t,X_t,\theta_t)\).
	The intermediate integrands \(\widetilde Z^{i,j}\) are also
	defined pointwise by their Markov representations.
	For the particle integrands \(Z^{i,j}\), choose predictable
	representatives. Their laws are used for almost every \(t\),
	and the time-integrated quantities do not depend on this choice
	by Fubini's theorem.
	The general estimates below for the particle integrands
	are integrated over time.
\end{remark}

We now introduce the forward law flow used to estimate
these functionals.
On an auxiliary probability space with \((\Omega,\mathscr F_s,\mathbb P)\) atomless, let
\(\xi\in L^2(\mathscr F_s;\mathbb R^d)\) have a prescribed law
\(\mu\in\mathcal P_2(\mathbb R^d)\) and be independent of the Brownian increments
\((W_r-W_s)_{s\le r\le T}\). 
Let \(X^{s,\xi}\) denote the unique solution of
\begin{align}\label{flowsde}
X_r^{s,\xi}
=\xi+\int_s^r b(q,X_q^{s,\xi},\mathbb{P}_{X_q^{s,\xi}}) dq
+\int_s^r \sigma(q,X_q^{s,\xi},\mathbb{P}_{X_q^{s,\xi}}) dW_q,
\q\ 0\leq s\leq r\leq t\leq T.
\end{align} 
By uniqueness in law, \(\mathbb P_{X_t^{s,\xi}}\) depends
on \(\xi\) only through its law \(\mu\).
Define \(\mathbf P_{s,t}\mu=\mathbb P_{X_t^{s,\xi}}\).
The flow property gives
\begin{align*}
\mathbf P_{r,t}\circ\mathbf P_{s,r}=\mathbf P_{s,t},
\q\ 0\leq s\leq r\leq t\leq T.
\end{align*}
Set
\begin{align*}
	V_{t,\varphi}^Y(s,\mu)\triangleq \Phi_{t,\varphi}^Y(\mathbf{P}_{s,t}\mu)\q\ 
	\hbox{and} \q\ 
	V_{t,\psi}^Z(s,\mu)\triangleq \Phi_{t,\psi}^Z(\mathbf{P}_{s,t}\mu),
	\q\ 0\leq 
	s\leq t.
\end{align*} 
In particular,
 \(\mathbf{P}_{s,t}\theta_s=\theta_t\), 
 so
\begin{align*}
	V^Y_{t,\varphi}(s,\theta_s)
	= \Phi_{t,\varphi}^Y (\theta_t)\q\ 
	\hbox{and} \q\ 
	V^Z_{t,\psi}(s,\theta_s)
	=\Phi_{t,\psi}^Z (\theta_t),\q\ 
	0\leq s\leq t.
\end{align*}

We next study the regularity of
\(V_t^\Phi(s,\mu)=\Phi(\mathbf P_{s,t}\mu)\).
The following lemma extends
\cite[Theorem 2.18]{chassagen_22_aap}
to time-dependent coefficients.
Its proof is given in \autoref{proofoflemma}.
 
\begin{lemma}\label{timeinhomogeneous}\sl 
 Assume that  \(b, \sigma \in \mathcal{M}_3([0,T]\times 
 \mathbb{R}^d \times \mathcal{P}_2(\mathbb{R}^d))\).
Let
 \( \Phi: \mathcal{P}_2(\mathbb{R}^d)\longmapsto \mathbb{R}\) belong to \(   \mathcal{M}_3 (\mathcal{P}_2(\mathbb{R}^d))\).
The map
  \( V_t^\Phi :[0,t]\times \mathcal{P}_2(\mathbb{R}^d)\longmapsto \mathbb{R} \) is
 defined by 
 \begin{align*}
  V_t^\Phi(s,\mu)=\Phi (\mathbf{P}_{s,t}\mu).
 \end{align*}
 Then  there exists a constant \(C>0\) such that
  \begin{align*}
 	\sup\limits_{t\in [0,T]}
 	\| V_t^\Phi \|_{\mathcal{M}_3([0,t]\times \mathcal{P}_2(\mathbb{R}^d))}
 	\leq C\|\Phi\|_{\mathcal{M}_3(\mathcal{P}_2(\mathbb{R}^d))},
 \end{align*}
 where \(C\) depends only on \(T, d, \|b\|_{\mathcal{M}_3}, \|\sigma\|_{\mathcal{M}_3}\), and \(V_t^\Phi\) satisfies
   the following PDE:
 \begin{align*}
 	&\	\partial_s V_{t}^\Phi (s,\mu)
 	+ \int_{\mathbb{R}^d} \langle D_\mu V_{t }^\Phi (s,\mu)(v), b(s,v,\mu)\rangle \mu (dv)\\
 	&	+\frac{1}{2} \int_{\mathbb{R}^d} \operatorname{Tr} \big(
 	\sigma(s,v,\mu) \sigma(s,v,\mu)^\top D_v D_\mu  V_{t }^\Phi (s,\mu) (v) \big)\mu (dv)=0, \q\ 0\leq s\leq t, \\
 	& V_{t}^\Phi (t,\mu)
 	=\Phi (\mu).
 \end{align*}
 Time derivatives at the endpoints are understood one-sided; for \(t=0\)
 only the identity \(V_0^\Phi(0,\mu)=\Phi(\mu)\) is asserted.
\end{lemma}

\begin{assumption}\label{phi}\sl 
		For every \(t\in [0,T]\), \(\varphi\in \mathcal{F}_{3,1}\),
	\(\psi \in \mathcal{G}_{3,1} \),
the functionals	\( \Phi_{t,\varphi}^Y,  \Phi_{t,\psi}^Z \) belong to \( \mathcal{M}_3 (\mathcal{P}_2(\mathbb{R}^d))\), and   there
	   exists a constant \(L_{\mathcal{M}_3}>0\) such that 
\begin{align*}
	\sup\limits_{t\in [0,T]}\Big\{ 
	\sup\limits_{\varphi \in \mathcal{F}_{3,1}}\|\Phi_{t,\varphi}^Y\|_{\mathcal{M}_3(\mathcal{P}_2)}
	+\sup\limits_{\psi \in \mathcal{G}_{3,1}} 
	\| \Phi_{t,\psi}^Z \|_{\mathcal{M}_3 (\mathcal{P}_2)}
	\Big\}\leq L_{\mathcal{M}_3}.
\end{align*}
\end{assumption}

\begin{remark}\label{explainasp5.2} \rm 
A sufficient condition for \autoref{phi} is that all mixed spatial
and Lions derivatives of \(U\) and \(H_Z\) of positive total order
at most three are bounded and globally Lipschitz, uniformly in \(t\).
For instance,
the first Lions derivative of \(\Phi_{t,\varphi}^Y\) is
\begin{align*}
 D_\mu\Phi_{t,\varphi}^Y(\mu)(v)
 =D\varphi(U(t,v,\mu))D_xU(t,v,\mu)
 +\int_{\mathbb R^d}D\varphi(U(t,x,\mu))D_\mu U(t,x,\mu)(v)\mu(dx).
\end{align*}
The argument for \(\Phi_{t,\psi}^Z\) follows by replacing
\((\varphi,U)\) with \((\psi,H_Z)\).
These bounds are additional to \autoref{asp5.2}.
In particular, derivatives of \(H_Z=D_xU\,\sigma\)
may require one further spatial derivative of \(U\).

 For every fixed \(m\),
the same sufficient condition gives \autoref{phim} and
\autoref{fandsong2m} below.
In particular, these conditions hold for all the models in
\autoref{exasp4.2}.
The explicit fields \(U\) and \(H_Z\) have bounded and globally
Lipschitz mixed spatial and Lions derivatives of every positive
order, uniformly in \(t\).
Repeated differentiation of the product-measure integrals
  gives the required uniform bounds over the
corresponding test classes, with constants depending on \(m\).
\end{remark}

\begin{lemma}\label{sharpsecodwithweak}\sl 
	Assume that \autoref{5forsde}, \autoref{asp5.2} and \autoref{phi} hold. 
	Then there exists a constant \(C>0\), independent of \(n\) and \(i\),
	such that, for every \(i=1,\cdots,n\),
	\begin{align}\label{sharpfor2}
	\sup\limits_{t\in [0,T]} \Big\{  d_{3,1}^Y(\mathbb{P}_{\widetilde{Y}_t^i},
	\mathbb{P}_{Y_t}) 
	+  d_{3,1}^Z(\mathbb{P}_{\widetilde{Z}_t^{i,i}},
	\mathbb{P}_{Z_t})\Big\}
	\leq \frac{C}{n},
	\end{align} 
	where \(C\) depends only on \(d, \|b\|_{\mathcal{M}_3}, \|\sigma\|_{\mathcal{M}_3}, \|\sigma\|_{\infty}, L_U, L_{\mathcal{M}_3}, \mathbb{E}\left[|X_0|^2\right]
	\) and \(T\).
\end{lemma}

\begin{proof}
By exchangeability and \autoref{z=hz},
\begin{align*}
& \mathbb{E}\left[\Phi_{t,\varphi}^Y(\theta_t^n)\right]
 =\mathbb{E}\left[\varphi(\widetilde{Y}_t^i)\right],
 \q\ 
  \mathbb{E}\left[\Phi_{t,\psi}^Z(\theta_t^n)\right]
 =\mathbb{E}\left[\psi(H_Z(t,X_t^i,\theta_t^n))\right],\\
 & \mathbb{E}\left[\Phi_{t,\varphi}^Y(\theta_t)\right]
 =\mathbb{E}\left[\varphi( Y_t)\right],
 \q\ 
 \mathbb{E}\left[\Phi_{t,\psi}^Z(\theta_t)\right]
 =\mathbb{E}\left[\psi(H_Z(t,X_t,\theta_t))\right].
\end{align*}
Set \(v^n(s,x_1,\cdots,x_n)=V_{t,\varphi}^Y(s,\mu^n)\), where
\(\mu^n=\frac{1}{n}\sum_{i=1}^{n}\delta_{x_i}\). By the differentiation rule for empirical measure lifts,
\begin{align*}
 D_{x_i}v^n(s,x_1,\cdots,x_n)&=\frac{1}{n}D_\mu V_{t,\varphi}^Y(s,\mu^n)(x_i),\\
D_{x_i}^2 v^n(s,x_1,\cdots,x_n)&=\frac{1}{n}D_v D_\mu V_{t,\varphi}^Y(s,\mu^n)(x_i)+\frac{1}{n^2}D_\mu^2 V_{t,\varphi}^Y(s,\mu^n)(x_i,x_i).
\end{align*}
Applying It\^o's formula to \(V_{t,\varphi}^Y(s,\theta_s^n)=v^n(s,X_s^1,\cdots,X_s^n)\),
\begin{equation}\label{itofor_v}
\begin{aligned}
 V_{t,\varphi}^Y(s,\theta_s^n)
 =&\ V_{t,\varphi}^Y(0,\theta_0^n)
 +\int_0^s \Big( \partial_r V_{t,\varphi}^Y(r,\theta_r^n)
 +\frac{1}{n}\sum_{i=1}^{n} D_\mu V_{t,\varphi}^Y(r,\theta_r^n)(X_r^i)b(r,X_r^i,\theta_r^n)\\
&+\frac{1}{2n}\sum_{i=1}^{n}
\operatorname{Tr}\big( \sigma(r,X_r^i,\theta_r^n) 
 \sigma(r,X_r^i,\theta_r^n)^\top D_v D_\mu V_{t,\varphi}^Y (r,\theta_r^n) (X_r^i)\big)\\
&+\frac{1}{2n^2} \sum_{i=1}^{n} \operatorname{Tr}
\big( \sigma(r,X_r^i,\theta_r^n) \sigma(r,X_r^i,\theta_r^n)^\top D_\mu^2V_{t,\varphi}^Y (r,\theta_r^n) (X_r^i,X_r^i)  \big) \Big) dr\\
& +\frac{1}{n}\sum_{i=1}^{n}\int_0^s 
D_\mu V_{t,\varphi}^Y(r,\theta_r^n) (X_r^i) \sigma(r,X_r^i,\theta_r^n) dW_r^i,\q\ 0\leq s\leq t.
\end{aligned}
\end{equation}
By \autoref{timeinhomogeneous}, \(V_{t,\varphi}^Y \in \mathcal{M}_3\) and satisfies
\begin{align*}
&\	\partial_s V_{t,\varphi}^Y(s,\mu)
	+ \int_{\mathbb{R}^d} D_\mu V_{t,\varphi}^Y (s,\mu)(v) b(s,v,\mu)\mu (dv)\\
&	+\frac{1}{2} \int_{\mathbb{R}^d} \operatorname{Tr} \big(
	\sigma(s,v,\mu) \sigma(s,v,\mu)^\top D_v D_\mu  V_{t,\varphi}^Y (s,\mu) (v) \big)\mu (dv)=0.
\end{align*}
Thus \eqref{itofor_v} reduces to 
\begin{equation}\label{masterreduce}
\begin{aligned}
 V_{t,\varphi}^Y(t,\theta_t^n)
 =&\ V_{t,\varphi}^Y(0,\theta_0^n)
 +\frac{1}{2n}\int_0^t\int_{\mathbb{R}^d} \operatorname{Tr}
 \big( \sigma(s,x,\theta_s^n) \sigma(s,x,\theta_s^n)^\top
 D_\mu^2 V_{t,\varphi}^Y(s,\theta_s^n)(x,x) \big)
 \theta_s^n (dx) ds\\
& +\frac{1}{n}\sum_{i=1}^{n}\int_0^t 
 D_\mu V_{t,\varphi}^Y(s,\theta_s^n) (X_s^i) \sigma(s,X_s^i,\theta_s^n) dW_s^i,\q\ 0\leq t\leq T.
\end{aligned}
\end{equation}
 Since \(D_\mu V^Y_{t,\varphi}\) and \(\sigma\) are bounded,
 \begin{align*}
 	 \mathbb{E}\left[
 	 \frac{1}{n^2}\sum_{i=1}^{n}\int_0^t  \big|
 	 D_\mu V_{t,\varphi}^Y(s,\theta_s^n) (X_s^i) \sigma(s,X_s^i,\theta_s^n) \big|^2 ds
 	 \right]\leq \frac{C}{n}.
 \end{align*}
 Hence  the stochastic integral in \eqref{masterreduce}  is a square-integrable martingale.  Taking expectations yields
\begin{align*}
 \mathbb{E}\left[
 \varphi(\widetilde{Y}_t^i)-\varphi(Y_t)\right]
 =&\ \mathbb{E}\left[ V^Y_{t,\varphi}(0,\theta_0^n)-V^Y_{t,\varphi}(0,\theta_0)\right]\\
 &+\frac{1}{2n}\int_0^t\mathbb{E}\int_{\mathbb{R}^d}
 \operatorname{Tr}\big( \sigma(s,x,\theta_s^n)\sigma(s,x,\theta_s^n)^\top
 D_\mu^2V^Y_{t,\varphi}(s,\theta_s^n)(x,x)\big)\theta_s^n (dx) ds.
\end{align*}
For the initial error, we give a direct proof using only
the second moment of \(X_0\); see also
\cite[Theorem 2.14(i)]{chassagen_22_aap}.
Let \(F=V_{t,\varphi}^Y(0,\cdot)\), \(\mu=\theta_0\), and let
\(\xi_1,\ldots,\xi_n\) be independent with law \(\mu\).
By \cite[Theorems 2.5--2.6 and Lemma 2.7]{chassagen_22_aap}, \(F\) has first and second linear functional derivatives with integrable growth. Using the notation of \autoref{def2,1}--\autoref{def2.2}, we have
\begin{align*}
 K\triangleq \sup_{\eta,v,w}\left\|D_vD_w\frac{\partial^2F}{\partial m^2}(\eta)(v,w)\right\|_{op}
 =\sup_{\eta,v,w}\|D_\mu^2F(\eta)(v,w)\|_{op}<\infty.
\end{align*}
By \autoref{phi} and \autoref{timeinhomogeneous},
\(K\) is bounded uniformly in \(t\) and \(\varphi\).
Set \(\nu_j=(1-j/n)\mu+n^{-1}\sum_{\ell=1}^j\delta_{\xi_\ell}\).
In the Taylor expansion of \(F(\nu_j)-F(\nu_{j-1})\), the first-order term
has expectation zero, since \(\xi_j\) is independent of \(\nu_{j-1}\).
Set \(K_\eta=\frac{\partial^2 F}{\partial m^2}(\eta)
 \).
 For every measure \(\eta\) on the segment joining
 \(\nu_{j-1}\) and \(\nu_j\), 
 we have
\begin{align*}
 \left|\iint_{\mathbb{R}^d \times \mathbb{R}^d} K_\eta(v,w)(\delta_{\xi_j}-\mu)(dv)(\delta_{\xi_j}-\mu)(dw)\right|
 \leq K\iint_{\mathbb{R}^d \times \mathbb{R}^d}
 |\xi_j-v|\,|\xi_j-w|\mu(dv)\mu(dw).
\end{align*}
This follows by expanding the signed measures and applying
the bound on \(D_vD_wK_\eta\).
The expectation of the right-hand side is at most
\(2K\mathbb E\left[ 
|\xi_1-\mathbb E\xi_1|^2 \right]\).
Since  \(\int_0^1(1-r)dr=1/2\), summing the Taylor remainders
over \(j=1,\cdots,n\) gives
\begin{align*}
 \big|\mathbb E\left[ F(\theta_0^n)-F(\theta_0)\right]\big|
 \leq\frac K n\mathbb E\left[ |X_0-\mathbb EX_0|^2\right].
\end{align*}
All terms are integrable under the second-moment assumption.
Indeed, the first linear functional derivative can be normalized
to have linear growth in its auxiliary variable, and the second
has at most quadratic growth by
\cite[Lemma 2.7]{chassagen_22_aap}. Consequently,
\begin{align*}
 \sup_{t\in[0,T]}\sup_{\varphi\in\mathcal F_{3,1}}
 \left|\mathbb E\left[V_{t,\varphi}^Y(0,\theta_0^n)-V_{t,\varphi}^Y(0,\theta_0)\right]\right|
 \leq\frac Cn.
\end{align*}
Moreover, by 
the boundedness of \(\sigma \) and \(D_\mu^2 V\),  we have
\begin{align*}
&\	\frac{1}{2n}\Big| \int_0^t\mathbb{E}\int_{\mathbb{R}^d}
	\operatorname{Tr}\big( \sigma(s,x,\theta_s^n)\sigma(s,x,\theta_s^n)^\top
	D_\mu^2V^Y_{t,\varphi}(s,\theta_s^n)(x,x)\big)\theta_s^n (dx) ds\Big|\\
&	\leq  \frac{C_d T}{2n}\|\sigma\|_{\infty}^2
\sup\limits_{(s,\mu,x)} |D_\mu^2 V_{t,\varphi}^Y (s,\mu) (x,x)|\\
&\leq 
 \frac{C}{n}.
\end{align*}
Therefore,
\begin{align}\label{5second}
	 \sup\limits_{t\in [0,T]}\sup\limits_{\varphi \in \mathcal{F}_{3,1}}\Big|\mathbb{E}\left[
	 \varphi(\widetilde{Y}_t^i)\right]-
	\mathbb{E}\left[\varphi(Y_t)  
	 \right] \Big| \leq 
	 \frac{C}{n}.
\end{align}
Applying \autoref{timeinhomogeneous} and the preceding argument with 
\((\Phi_{t,\varphi}^Y, V_{t,\varphi}^Y)\) replaced by 
\( (\Phi_{t,\psi}^Z, V_{t,\psi}^Z) \), we obtain
\begin{align}\label{5oversingleforz}
	\sup\limits_{t\in [0,T]}
	\sup\limits_{\psi \in \mathcal{G}_{3,1}}
	\Big|   \mathbb{E}\left[
	\psi(H_Z(t,X_t^i,\theta_t^n))\right]-
	\mathbb{E}\left[\psi(H_Z(t,X_t,\theta_t)  )
	\right] 
	\Big|\leq \frac{C}{n}.
\end{align}
Moreover, 
\begin{align*}
 \widetilde{Z}_t^{i,i}
 =H_Z(t,X_t^i,\theta_t^n)
 +\frac{1}{n}D_\mu U(t,X_t^i,\theta_t^n)(X_t^i)\sigma
 (t,X_t^i,\theta_t^n).
\end{align*}
Since \([\psi]_{Lip}\leq1\), the boundedness of \(D_\mu U\)
and \(\sigma\) gives
\begin{align}\label{5oversingeforzs2}
		\sup\limits_{t\in [0,T]}
	\sup\limits_{\psi \in \mathcal{G}_{3,1}}
	\Big|   \mathbb{E}\left[
	\psi(\widetilde{Z}_t^{i,i})\right]-
	\mathbb{E}\left[\psi(H_Z(t,X^i_t,\theta^n_t)  )
	\right] \Big| 
	\leq \frac{L_U \|\sigma\|_{\infty} }{n}.
\end{align}
Combining \eqref{5second}, \eqref{5oversingleforz},
and \eqref{5oversingeforzs2} proves \eqref{sharpfor2}.
\end{proof}

The preceding proof uses the \(\mathcal M_3\)-regularity assumptions
of Chassagneux, Szpruch, and Tse \cite{chassagen_22_aap}.
We next give sufficient conditions for applying the weak approximation
estimate of Frikha and Song \cite{frikhaandsong_26}.

\begin{assumption}\label{fandsong1}\sl 
Set \(A(t,x,\mu)=\sigma(t,x,\mu)\sigma(t,x,\mu)^\top\).
	  The following conditions hold:
\begin{enumerate}[(i)]
 \item  \(b, \sigma, A\) belong to \(\mathcal{M}_2([0,T]\times \mathbb{R}^d \times \mathcal{P}_2(\mathbb{R}^d))\).
 \item 
 \(\partial_t b\) and \(\partial_t A\) are jointly continuous in \((t,x,\mu)\) and satisfy
 \begin{align*}
\sup\limits_{(t,x,\mu)\in [0,T]\times \mathbb{R}^d \times \mathcal{P}_2(\mathbb{R}^d)}
\Big(
|\partial_t b(t,x,\mu)| +
|\partial_t A(t,x,\mu)|
 \Big) <\infty.
 \end{align*}
 \item 
The mixed derivatives \(D_\mu D_x b\) and \(D_\mu D_x \sigma\) exist and are jointly continuous. Moreover,
\begin{align*}
\sup\limits_{t\in[0,T]}\Big( \|D_\mu D_x b(t)\|_{\infty} +[D_\mu D_x b(t)]_{Lip}
+\|D_\mu D_x \sigma(t)\|_{\infty} +[D_\mu D_x \sigma(t)]_{Lip} \Big)<
\infty.
\end{align*}
\end{enumerate}
\end{assumption}

\begin{assumption}\label{fandsong2}\sl 
	For every \(t\in [0,T]\), \(\varphi\in \mathcal{F}_{3,1}\),
	\(\psi \in \mathcal{G}_{3,1} \),
the functionals	\( \Phi_{t,\varphi}^Y,  \Phi_{t,\psi}^Z\) belong to \( \mathcal{M}_2 (\mathcal{P}_2(\mathbb{R}^d))\),  and   there
exists a constant \(L_{\mathcal{M}_2}>0\) such that 
\begin{align*}
	\sup\limits_{t\in [0,T]}\Big\{ 
	\sup\limits_{\varphi \in \mathcal{F}_{3,1}}\|\Phi_{t,\varphi}^Y\|_{\mathcal{M}_2(\mathcal{P}_2)}
	+\sup\limits_{\psi \in \mathcal{G}_{3,1}} 
	\| \Phi_{t,\psi}^Z \|_{\mathcal{M}_2 (\mathcal{P}_2)}
	\Big\}\leq L_{\mathcal{M}_2}.
\end{align*}
\end{assumption}

Here \autoref{fandsong1} and \autoref{fandsong2} replace
\autoref{5forsde} and \autoref{phi}, respectively.
The standing assumptions and \autoref{asp5.2} are retained.

\begin{corollary}\label{frkandsong_cri}\sl 
	Under the standing assumptions of this section, 
	suppose additionally that
  \autoref{asp5.2}, \autoref{fandsong1} and \autoref{fandsong2} hold. Then for each \(i=1,\cdots,n\), 
 there exists a constant \(C>0\), independent of \(n, i, t\), such that
 	\begin{align}\label{coralymediate} 
 	\sup\limits_{t\in [0,T]} \Big\{  d_{3,1}^Y(\mathbb{P}_{\widetilde{Y}_t^i},
 	\mathbb{P}_{Y_t}) 
 	+  d_{3,1}^Z(\mathbb{P}_{\widetilde{Z}_t^{i,i}},
 	\mathbb{P}_{Z_t})\Big\}
 	\leq \frac{C}{n}.
 \end{align} 
\end{corollary}
\begin{proof}
	\autoref{fandsong1} implies \textbf{(HFR)} of
	Frikha and Song \cite{frikhaandsong_26}, with the factor
	\(\widetilde\sigma=\sigma\).
	 Indeed, the bounds for \(b, A\), including
	their time derivatives, give their \(\mathcal C_f^{1,2,2}\) condition;
the assumptions on mixed derivatives give the required flow regularity.
	  \autoref{fandsong2} gives the functional regularity required
	by \cite[Theorem 3]{frikhaandsong_26} and includes a global
	Lipschitz bound on \(D_\mu^2\Phi\).
	Neither \textbf{(HR)} nor \textbf{(HE)} is needed for that theorem.
	
	For \(h_l=T/l\), let \(k_l(s)=h_l\lfloor \frac{s}{h_l} \rfloor\) for \(s<T\).
	Using the same initial data and Brownian motions,
	 define the continuous
	Euler interpolation by
	\begin{align*}
		X_t^{i,n,h_l}=X_0^i
		&+\int_0^t b(k_l(s),X_{k_l(s)}^{i,n,h_l},
		\theta_{k_l(s)}^{n,h_l}) ds
		+\int_0^t\sigma(k_l(s),X_{k_l(s)}^{i,n,h_l},
		\theta_{k_l(s)}^{n,h_l}) dW_s^i,
	\end{align*}
	where 
	\begin{align*}
	  \theta_t^{n,h_l}\triangleq  \frac1n\sum_{j=1}^n\delta_{X_t^{j,n,h_l}}.
	\end{align*}
Let
 \(\mathscr A=\{\Phi_{t,\varphi}^Y,\Phi_{t,\psi}^Z:
	t\leq T,\ \varphi\in\mathcal F_{3,1},\ \psi\in\mathcal G_{3,1}\}\).
By \cite[Theorem 3]{frikhaandsong_26} and the uniform functional bounds, we obtain
	\begin{align*}
		\sup_{F\in\mathscr A}\sup_{t\in [0,T]}
		\Big|\mathbb E \left[ F(\theta_t^{n,h_l})-F(\theta_t) \right]\Big|
		\leq C\Big(\frac{1}{n}+h_l\Big).
	\end{align*}
	For each fixed \(n\),
	convergence of the Euler approximation gives
 \begin{align*}
	\mathbb{E}\left[\sup\limits_{t\in[0,T]}
	\mathcal{W}_2^2(\theta_t^{n,h_{l}},\theta_t^n)\right]
	\leq \frac{1}{n}\sum_{j=1}^{n}
	\mathbb{E}\left[\sup\limits_{t\in[0,T]}
	|X_t^{j,n,h_{l}}-X_t^{j,n}|^2\right]\longrightarrow 0.
\end{align*}
		The functionals in
		\( \mathscr A\) have a common \(\mathcal W_2\)-Lipschitz constant.
	By the triangle inequality, letting \(l\longrightarrow \infty\) with \(n\)  fixed 
	gives
	\begin{align*}
		\sup_{F\in\mathscr A}\sup_{t\in [0,T]}
		\Big|\mathbb E\left[F(\theta_t^n)-F(\theta_t)\right] \Big|\leq \frac{C}{n}.
	\end{align*}
	Exchangeability and  \autoref{z=hz} identify these functional errors
	with those of \(\widetilde Y^i\) and \(H_Z(t,X_t^i,\theta_t^n)\).
	Finally, the uniform second-moment estimate for the forward particles
	and linear growth of \(\sigma\) imply
 \begin{align*}
	\sup\limits_{n\geq 1}
	\sup\limits_{1\leq i\leq n}
	\sup\limits_{t\in [0,T]}
	\mathbb{E} \left[ 
	|\sigma(t,X_t^i,\theta_t^n)|^2\right]<\infty.
\end{align*}
	Consequently, for \(\psi\in\mathcal G_{3,1}\),
	\begin{align*}
		\sup_{t\in [0,T]}
	\Big| \mathbb{E}\left[\psi(\widetilde{Z}_t^{i,i})
	-\psi(H_Z(t,X_t^i,\theta_t^n))\right]\Big|
		\leq\frac{L_U}{n}\sup_{t \in [0,T]}\mathbb{E}
		\left[|\sigma(t,X_t^i,\theta_t^n)|\right]
		\leq \frac{C}{n}.
	\end{align*}
	This proves \eqref{coralymediate}.
\end{proof}

We next extend \autoref{empirical} to the diffusion coefficient
\(\sigma(t,x,\mu)\).
The proof differs only in the bounds for the consistency remainder
and the martingale correction.

\begin{lemma}\label{empirical2}\sl 
	Assume that    \autoref{asp5.2} holds and 
	that either \(\sigma\) is bounded
	or \autoref{fandsong1} holds.
	 Then, for every \(1\leq m\leq n,\)  
	\begin{align}\label{firstone2}
		\mathcal{W}_{2,\|\cdot\|_{\infty}}^2(\mathbb{P}_{(Y^1,\cdots,Y^m)},\mathbb{P}_{(\widetilde{Y}^1,\cdots,\widetilde{Y}^m)})
		+ \mathcal{W}_{2,L^2}^2(\mathbb{P}_{(Z^{1,1},\cdots,Z^{m,m})},\mathbb{P}_{(\widetilde{Z}^{1,1},\cdots,\widetilde{Z}^{m,m})})
		\leq C\frac{m}{n^2},
	\end{align}
	where the constant \(C>0\)  
	depends only on \(d,k,L_U,T\), 
	the parameters in  \autoref{lipp}, and either \(\|\sigma\|_{\infty}\)
	or the constants from \autoref{fandsong1}.
	The particlewise coupling
	estimate \eqref{empirical-coupling} also holds with the intermediate
	processes defined in this section.
\end{lemma}

\begin{proof}
For the general diffusion coefficient \(\sigma\), \eqref{widey} becomes
 \begin{align}\label{wideynew}
 	d\widetilde{Y}_t^i
 	=\Big(-f(t,X_t^i,\theta_t^n,\widetilde{Y}_t^i,m^U(t,\theta_t^n),H_Z(t,X_t^i,\theta_t^n),m^{\mathcal{Z}}(t,\theta_t^n))+R_t^{i,n}\Big) dt
 	+\sum_{j=1}^{n}\widetilde{Z}_t^{i,j}dW_t^j,
 \end{align}
 where
 \begin{align*}
 	&	m^U(t,\theta_t^n)=\frac{1}{n}\sum_{j=1}^{n}\delta_{U(t,X_t^j,\theta_t^n)}
 	\q\ \hbox{and} \q\ 
 	m^{\mathcal{Z}}(t,\theta_t^n)=
 	\frac{1}{n}\sum_{j=1}^{n}\delta_{D_x U(t,X_t^j,\theta_t^n)\sigma(t,X_t^j,\theta_t^n)},\\
 	&R_t^{i,n}=
 	\frac{1}{n}\operatorname{Tr}\Big(\sigma(t,X_t^i,\theta_t^n) \sigma(t,X_t^i,\theta_t^n)^\top D_x D_\mu U(t,X_t^i,\theta_t^n) (X_t^i)\Big)\\
 	&\q\ \q\  +\frac{1}{2n^2} \sum_{j=1}^{n}
 	\operatorname{Tr} \Big(\sigma(t,X_t^j,\theta_t^n) \sigma(t,X_t^j,\theta_t^n)^\top D_\mu^2 U(t,X_t^i,\theta_t^n)(X_t^j,X_t^j) \Big).
 \end{align*}
 The definitions of \(\xi_t^{i,n}\) and \(\eta_t^n\) in the proof of \autoref{empirical} are replaced by
 \begin{align*}
  \xi_t^{i,n}=H_Z(t,X_t^i,\theta_t^n)\q\ \hbox{and} \q\ 
  \eta_t^n=m^{\mathcal{Z}}(t,\theta_t^n).
 \end{align*}
 Since   the relevant derivatives of  \(U\) are bounded,
\begin{align*}
	 |R_t^{i,n}|
	 \leq 
	 \frac{C}{n}
	 \Bigg( \big|\sigma(t,X_t^i,\theta_t^n)\sigma(t,X_t^i,\theta_t^n)^\top\big|
	 +\frac{1}{n}
	 \sum_{j=1}^{n}\big|\sigma(t,X_t^j,\theta_t^n)\sigma(t,X_t^j,\theta_t^n)^\top\big| \Bigg).
\end{align*}
If \(\sigma\) is bounded, then
 \(|R_t^{i,n}|\leq 
 \frac{C}{n}\).  Under \autoref{fandsong1}, \(A=\sigma\sigma^\top\)
 has at most linear growth:
\begin{align*}
	 |A(t,x,\mu)|
	 \leq 
	 C \Bigg( 1+|x| +\Big( \int_{\mathbb{R}^d}
	 |v|^2 \mu (dv)  \Big)^{\frac{1}{2}} \Bigg).
\end{align*}
The uniform second-moment bound for the forward particles gives
\begin{align*}
	\sup\limits_{n\geq 1}
	\sup\limits_{1\leq i\leq n}
 \sup\limits_{t\in [0,T]} 
 \mathbb{E}\left[   |A(t,X_t^i,\theta_t^n)|^2 \right]<\infty.
\end{align*}
Hence Jensen's inequality  gives
\begin{align*}
 \sup\limits_{1\leq i\leq n}
 \mathbb{E}\int_0^T |R_t^{i,n}|^2 dt
 \leq 
 \frac{C}{n^2}.
\end{align*}
Since \(|\sigma(t,x,\mu)|^2=\operatorname{Tr}A(t,x,\mu)\leq\sqrt d\,|A(t,x,\mu)|\),
we obtain
 \begin{align*}
\mathbb{E}\int_0^T |\widetilde{Z}_t^{j,j}-H_Z(t,X_t^j,\theta_t^n)|^2dt
 &\leq \frac{C}{n^2}\mathbb{E}\int_0^T |
  \sigma (t,X_t^j,\theta_t^n)|^2dt\\
& \leq \frac{CT}{n^2}\sup\limits_{t\in [0,T]}
 \mathbb{E}\left[   |\sigma(t,X_t^j,\theta_t^n)|^2 \right]\\
& \leq 
  \frac{CT\sqrt{d}}{n^2}\sup\limits_{t\in [0,T]}\Bigg(
  \mathbb{E}\left[   |A(t,X_t^j,\theta_t^n)|^2 \right]\Bigg)^{\frac{1}{2}}\\
 &\leq \frac{C}{n^2}.
 \end{align*}
 Applying the averaged BSDE stability estimate from the proof of
 \autoref{empirical} and using exchangeability proves
 \eqref{empirical-coupling} and \eqref{firstone2}.
\end{proof}

The following proposition bounds the backward weak error
in terms of the forward functional errors.

\begin{proposition}[Weak-error transfer]\label{abstract-transfer}\sl
Assume the hypotheses of \autoref{empirical2} and set
\begin{align*}
 K_\sigma=\sup\limits_{n\geq 1}
 \sup\limits_{1\leq i\leq n}
 \sup\limits_{t\in [0,T]}
 \mathbb E\left[ |\sigma(t,X_t^i,\theta_t^n)|^2\right] <\infty.
\end{align*}
This quantity is finite under these hypotheses.
Define
\begin{align*}
 \varepsilon_n^Y
 &=\sup_{t\in[0,T]}\sup_{\varphi\in\mathcal F_{3,1}}
 \left|\mathbb E \left[ \Phi_{t,\varphi}^Y(\theta_t^n)
              -\Phi_{t,\varphi}^Y(\theta_t)\right] \right|\q\ \hbox{and} \q\ 
 \varepsilon_n^Z
 =\int_0^T\sup_{\psi\in\mathcal G_{3,1}}
 \left|\mathbb E\left[\Phi_{t,\psi}^Z(\theta_t^n)
              -\Phi_{t,\psi}^Z(\theta_t)\right]\right|dt.
\end{align*}
Then, for every \(i=1,\ldots,n\),
\begin{align}\label{abstract-transfer-bound}
 \sup_{t\in[0,T]}d_{3,1}^Y(\mathbb P_{Y_t^i},\mathbb P_{Y_t})
 +\int_0^T d_{3,1}^Z(\mathbb P_{Z_t^{i,i}},\mathbb P_{Z_t})dt
 \leq \varepsilon_n^Y+\varepsilon_n^Z+\frac Cn,
\end{align}
where \(C\) depends only on the constants in \autoref{empirical2},
\(L_U,T,K_\sigma\). 
No additional \(\mathcal M_2\)- or \(\mathcal M_3\)-regularity
of \(\Phi^Y\) and \(\Phi^Z\) is needed.
\end{proposition}

\begin{proof}
Set \(\widehat Z_t^i=H_Z(t,X_t^i,\theta_t^n)\).
Exchangeability and the limiting Markov representation give
\begin{align*}
 \sup_{t\in [0,T]}d_{3,1}^Y(\mathbb P_{\widetilde Y_t^i},\mathbb P_{Y_t})
 &=\varepsilon_n^Y\q\ \hbox{and} \q\ 
 \int_0^T d_{3,1}^Z(\mathbb P_{\widehat Z_t^i},\mathbb P_{Z_t})dt
 =\varepsilon_n^Z.
\end{align*}
Since the test functions are 1-Lipschitz, the particlewise estimate in
\autoref{empirical2} and Cauchy--Schwarz yield
\begin{align*}
 \sup_{t\in [0,T]}d_{3,1}^Y(\mathbb P_{Y_t^i},\mathbb P_{\widetilde Y_t^i})
 +\int_0^T d_{3,1}^Z(\mathbb P_{Z_t^{i,i}},\mathbb P_{\widetilde Z_t^{i,i}})dt
 \leq \frac Cn.
\end{align*}
Moreover,
\(\widetilde Z_t^{i,i}-\widehat Z_t^i
 =\frac{1}{n} D_\mu U(t,X_t^i,\theta_t^n)(X_t^i)\sigma(t,X_t^i,\theta_t^n)\),
so
\begin{align*}
 \int_0^T d_{3,1}^Z(\mathbb P_{\widetilde Z_t^{i,i}},\mathbb P_{\widehat Z_t^i})dt
 \leq \frac{L_UT\sqrt{K_\sigma}}n.
\end{align*}
The triangle inequality proves \eqref{abstract-transfer-bound}.
\end{proof}

The estimates above give
\(\varepsilon_n^Y+\varepsilon_n^Z\leq C/n\).
Other bounds for these functional errors can also be used in
\autoref{abstract-transfer}.
To obtain a rate faster than \(n^{-1}\) from this bound,
the consistency estimate must also be improved.

\begin{theorem}\label{5main}\sl 
 Assume that \autoref{asp5.2}  
 holds and that either
 \autoref{5forsde} and \autoref{phi} hold,
 or \autoref{fandsong1} and \autoref{fandsong2} hold.
Then there exists a constant \(C>0\), independent of \(n\) and \(i\),
such that, for every \(i=1,\cdots,n\),
 \begin{align}\label{5mainestimate}
 \sup\limits_{t\in [0,T]} d_{3,1}^Y(\mathbb{P}_{Y_t^i},
 \mathbb{P}_{Y_t} ) 
 +\int_0^T d_{3,1}^Z(  \mathbb{P}_{Z_t^{i,i}},
 \mathbb{P}_{Z_t})dt
 \leq \frac{C}{n}.
 \end{align}
 \end{theorem}
 
 \begin{proof}
The forward functional estimates established in the proofs of
\autoref{sharpsecodwithweak} and \autoref{frkandsong_cri}, respectively,
give \(\varepsilon_n^Y+\varepsilon_n^Z\leq C/n\) under the two sets of
assumptions. The required bound on \(K_\sigma\) follows from boundedness
of \(\sigma\) in the first case and from the uniform forward
second-moment estimate and linear growth of \(\sigma\) in the second.
Applying \autoref{abstract-transfer} gives \eqref{5mainestimate}.
\end{proof}

  \begin{remark}\label{twoassump}\rm
  	The first set of assumptions 
  	 in \autoref{5main}  requires third-order spatial and
  	measure regularity of \(b,\sigma\) and the functionals
  	\(\Phi^Y,\Phi^Z\), together with boundedness of \(\sigma\).
  	Only a second moment of the initial law is needed, as shown
  	by the initial bias estimate in the proof of
  	\autoref{sharpsecodwithweak}.
The second set requires second-order regularity of
\(b,\sigma,A=\sigma\sigma^\top\) and the functionals
\(\Phi^Y,\Phi^Z\), together with the bounds on
\(\partial_t b,\partial_t A,D_\mu D_xb,D_\mu D_x\sigma\)
in \autoref{fandsong1}.
It also requires only a second moment of the initial law
and allows \(\sigma\) to be unbounded.
However, the bounded first derivatives of \(A\) imply
 \begin{align*}
 |\sigma(t,x,\mu)|^2=\operatorname{Tr}\big(A(t,x,\mu)\big)
 \leq C\left(1+|x|+\Big(\int_{\mathbb R^d}|v|^2\mu(dv)\Big)^{1/2}\right).
 \end{align*}
Thus these assumptions do not cover arbitrary diffusion coefficients
with linear growth.
Neither set requires uniform ellipticity; \(A\) may be degenerate.
  \end{remark}
  
We next extend the estimate to \(m\)-particle marginals, with \(m\) fixed.
Equip \((\mathbb R^k)^m\) and \((\mathbb R^{k\times d})^m\)
with their Euclidean product norms.
   Define the following  two classes of test functions:
\begin{align*}
	\mathcal{F}_{3,1}^m
	=\{ \varphi \in C_b^{3,1}((\mathbb{R}^k)^m): \|\varphi\|_{ 3,1}\leq 1\}\q\ 
	\hbox{and} \q\ 
	\mathcal{G}_{3,1}^m
	=\{ \psi \in C_b^{3,1}((\mathbb{R}^{k\times d})^m): \|\psi\|_{ 3,1}\leq 1\}.
\end{align*}

For any \(\mu,\mu' \in \mathcal{P}_2((\mathbb{R}^k)^m)\), 
\(\nu, \nu' \in \mathcal{P}_2((\mathbb{R}^{k\times d})^m)\),
define the corresponding integral probability metrics
\begin{align*}
	&d_{3,1}^{Y,m} (\mu,\mu')=\sup\limits_{\varphi \in \mathcal{F}^m_{3,1}} \Big|\int_{(\mathbb{R}^k)^m} \varphi (y)\mu(dy)
	-\int_{(\mathbb{R}^k)^m} \varphi(y) \mu'(dy) \Big|,\\
	&d_{3,1}^{Z,m}(\nu,\nu')
	=\sup\limits_{\psi \in \mathcal{G}^m_{3,1}}
	\Big| \int_{(\mathbb{R}^{k\times d})^m} \psi(z) \nu(dz)
	-\int_{(\mathbb{R}^{k\times d})^m} \psi(z) \nu'(dz)    \Big|.
\end{align*}

For every \(t\in [0,T]\), \(\varphi\in \mathcal{F}^m_{3,1}\) and 
\(\psi \in \mathcal{G}^m_{3,1}\),  
let
\begin{align*}
	&\Phi_{t,\varphi}^{Y,m}(\mu) 
	=\int_{(\mathbb{R}^d)^m} \varphi\big( U(t,x_1,\mu),\cdots,
	U(t,x_m,\mu)
	\big) \mu^{\otimes m} (dx_1,\cdots,dx_m),\\
	&\Phi_{t,\psi}^{Z,m} (\mu) 
	=\int_{(\mathbb{R}^{ d})^m} 
	\psi\big( H_Z(t,x_1,\mu),\cdots,
	H_Z(t,x_m,\mu) \big) \mu^{\otimes m}(dx_1,\cdots,dx_m).
\end{align*}

The following assumptions are the \(m\)-particle counterparts of
\autoref{phi} and \autoref{fandsong2}, respectively.

\begin{assumption}\label{phim}\sl 
		For every \(t\in [0,T]\), \(\varphi\in \mathcal{F}^m_{3,1}\),
	\(\psi \in \mathcal{G}^m_{3,1} \),
	the functionals	\( \Phi_{t,\varphi}^{Y,m},  \Phi_{t,\psi}^{Z,m}\) belong to \( \mathcal{M}_3 (\mathcal{P}_2(\mathbb{R}^d))\), and   there
 exists a constant \(L^{(m)}_{\mathcal{M}_3}>0\) such that 
	\begin{align*}
		\sup\limits_{t\in [0,T]}\Big\{ 
		\sup\limits_{\varphi \in \mathcal{F}^m_{3,1}}\|\Phi_{t,\varphi}^{Y,m}\|_{\mathcal{M}_3(\mathcal{P}_2)}
		+\sup\limits_{\psi \in \mathcal{G}^m_{3,1}} 
		\| \Phi_{t,\psi}^{Z,m} \|_{\mathcal{M}_3 (\mathcal{P}_2)}
		\Big\}\leq L^{(m)}_{\mathcal{M}_3}.
	\end{align*}
\end{assumption}

\begin{assumption}\label{fandsong2m}\sl 
	For every \(t\in [0,T]\), \(\varphi\in \mathcal{F}^m_{3,1}\),
	\(\psi \in \mathcal{G}^m_{3,1} \),
the functionals \( \Phi_{t,\varphi}^{Y,m},  \Phi_{t,\psi}^{Z,m}\) belong to \( \mathcal{M}_2 (\mathcal{P}_2(\mathbb{R}^d))\), and    there
	exists a constant \(L^{(m)}_{\mathcal{M}_2}>0\) such that 
	\begin{align*}
		\sup\limits_{t\in [0,T]}\Big\{ 
		\sup\limits_{\varphi \in \mathcal{F}^m_{3,1}}\|\Phi_{t,\varphi}^{Y,m}\|_{\mathcal{M}_2(\mathcal{P}_2)}
		+\sup\limits_{\psi \in \mathcal{G}^m_{3,1}} 
		\| \Phi_{t,\psi}^{Z,m} \|_{\mathcal{M}_2 (\mathcal{P}_2)}
		\Big\}\leq L^{(m)}_{\mathcal{M}_2}.
	\end{align*}
\end{assumption}

To compare the product empirical measure with \(m\) distinct particles,
we also estimate the difference between sampling with and without replacement.
Define the forward functional errors by
\begin{align*}
	\varepsilon_{m,n}^Y
	&=\sup_{t\in[0,T]}\sup_{\varphi\in\mathcal F^m_{3,1}}
	\left|\mathbb E \left[ \Phi_{t,\varphi}^{Y,m}(\theta_t^n)
	-\Phi_{t,\varphi}^{Y,m}(\theta_t)\right] \right|\q\ \hbox{and} \q\ 
	\varepsilon_{m,n}^Z
	=\int_0^T\sup_{\psi\in\mathcal G^m_{3,1}}
	\left|\mathbb E\left[\Phi_{t,\psi}^{Z,m}(\theta_t^n)
	-\Phi_{t,\psi}^{Z,m}(\theta_t)\right]\right|dt.
\end{align*}
\begin{theorem}\label{5maintotal}\sl 
	Fix \(m\in \mathbb{N}\).
	 Assume that \autoref{asp5.2} holds
	 and that either
	\autoref{5forsde} and \autoref{phim} hold,
	or \autoref{fandsong1} and \autoref{fandsong2m} hold.
Then there exists a constant \(C_m>0\), independent of \(n\),
such that, for every \(n\geq m\),
	\begin{align}\label{5mainestimatetotal}
		\sup\limits_{t\in [0,T]} d_{3,1}^{Y,m}(\mathbb{P}_{(Y_t^1,\cdots,Y_t^m)},
		(\mathbb{P}_{Y_t})^{\otimes m} ) 
		+\int_0^T d_{3,1}^{Z,m}(  \mathbb{P}_{(Z_t^{1,1},\cdots,Z_t^{m,m})},
	(	\mathbb{P}_{Z_t})^{\otimes m})dt
		\leq \frac{C_m}{n},
	\end{align}
	where \(C_m\) depends only on \(d,k,m,T\), the second-moment bound on \(X_0\),
and the bounds in the applicable assumptions.
More precisely, the left-hand side of \eqref{5mainestimatetotal} is bounded by
\begin{align*}
 \varepsilon_{m,n}^Y+\varepsilon_{m,n}^Z
 +(1+T)\frac{m(m-1)}n+C_T\frac{\sqrt m}{n},
\end{align*}
where \(C_T\) is independent of \(m,n\).
\end{theorem}

\begin{proof}
We prove the result under \autoref{5forsde}, \autoref{asp5.2} and \autoref{phim}. Under \autoref{asp5.2}, \autoref{fandsong1}
and \autoref{fandsong2m}, the same argument applies, with the forward weak estimate below replaced by the 
Frikha and Song \cite{frikhaandsong_26} estimate used in the proof of \autoref{frkandsong_cri}.
By definition,
	\begin{align*}
	\Phi_{t,\varphi}^{Y,m}(\theta_t^n)
	=\frac{1}{n^m}\sum_{j_1,\cdots,j_m=1}^{n}
	\varphi (\widetilde{Y}_t^{j_1},\cdots,\widetilde{Y}_t^{j_m})\q\ \hbox{and}\q\ 
		\Phi_{t,\psi}^{Z,m}(\theta_t^n)
	=\frac{1}{n^m}\sum_{j_1,\cdots,j_m=1}^{n}
	\psi (\widehat{Z}_t^{j_1},\cdots,\widehat{Z}_t^{j_m}),
	\end{align*}
	where \(\widehat{Z}_t^j\triangleq H_Z(t,X_t^j,\theta_t^n)\).
Among the \(n^m\) ordered tuples, exactly
\((n)_m=n(n-1)\cdots(n-m+1)\) have distinct indices.
By exchangeability, each such tuple has the same expected
test-function value as the first \(m\) particles.
The difference in expectations is therefore bounded by twice
the supremum norm of the test function times the proportion
of tuples with repeated indices.
Since
\(1-\prod_{l=0}^{m-1}(1-\frac{l}{n}) \leq \sum_{l=0}^{m-1} \frac{l}{n}=\frac{m(m-1)}{2n}\),
\begin{align*}
\Big|\mathbb{E}\left[
\Phi_{t,\varphi}^{Y,m}(\theta_t^n)\right]-
\mathbb{E}\left[\varphi(\widetilde{Y}^1_t,\cdots,\widetilde{Y}_t^m)  
\right] \Big| \leq 
2\|\varphi\|_{\infty}
\Big( \frac{n^m-(n)_m}{n^m} \Big)
\leq \frac{m(m-1)}{n}.
\end{align*}
Similarly,
\begin{align*}
	\Big|\mathbb{E}\left[
	\Phi_{t,\psi}^{Z,m}(\theta_t^n)\right]-
	\mathbb{E}\left[\psi(\widehat{Z}^{1}_t,\cdots,\widehat{Z}_t^{m})  
	\right] \Big| \leq 
	2\|\psi\|_{\infty}
	\Big( \frac{n^m-(n)_m}{n^m} \Big)
	\leq \frac{m(m-1)}{n}.
\end{align*}
The Markov representation gives
\begin{align*}
	\Phi_{t,\varphi}^{Y,m} (\theta_t)
	&=\int_{(\mathbb{R}^k)^m} \varphi(\textbf{y}) (\mathbb{P}_{Y_t})^{\otimes m}(\textbf{dy}),
\\
	\Phi_{t,\psi}^{Z,m}(\theta_t)
	&=\int_{(\mathbb{R}^{k\times d})^m } \psi(\textbf{z}) (\mathbb{P}_{Z_t})^{\otimes m}(\textbf{dz}).
\end{align*}
Repeating the argument used in the proof of  \autoref{sharpsecodwithweak},
with \(\Phi_{t,\varphi}^Y\) and \(\Phi_{t,\psi}^Z\) replaced
by \(\Phi_{t,\varphi}^{Y,m}\) and \(\Phi_{t,\psi}^{Z,m}\), respectively, and using \autoref{phim}, we obtain
\begin{align*}
 \sup\limits_{t\in[0,T],\varphi\in \mathcal{F}_{3,1}^m}
 \Big|\mathbb{E}\left[  \Phi_{t,\varphi}^{Y,m}(\theta_t^n) 
 -\Phi_{t,\varphi}^{Y,m}(\theta_t)\right]\Big|
 + \sup\limits_{t\in[0,T],\psi\in \mathcal{G}_{3,1}^m}
 \Big|\mathbb{E}\left[  \Phi_{t,\psi}^{Z,m}(\theta_t^n) 
 -\Phi_{t,\psi}^{Z,m}(\theta_t)\right]\Big|
 \leq 
 \frac{C_m}{n}.
\end{align*}
Note that
\begin{align*}
\widetilde{Z}_t^{i,i}
-\widehat{Z}_t^{i}
=\frac{1}{n}D_\mu U(t,X_t^i,\theta_t^n)(X_t^i)\sigma (t,X_t^i,\theta_t^n).
\end{align*}
Since the test functions are 1-Lipschitz,
the boundedness of \(D_\mu U\)  gives
\begin{align*}
	\sup\limits_{t\in [0,T]}
	\sup\limits_{\psi\in \mathcal{G}_{3,1}^m}
 \Big|\mathbb{E}\left[ \psi(\widetilde{Z}_t^{1,1},\cdots,\widetilde{Z}_t^{m,m})
 -\psi(\widehat{Z}_t^{1},\cdots,\widehat{Z}_t^{m})
 \right]   \Big|
 \leq  \frac{L_U }{n} \sup\limits_{t\in [0,T]}
 \mathbb{E}\left[ 
 \Bigg( \sum_{j=1}^{m}|\sigma(t,X_t^{j},\theta_t^n)|^2
 \Bigg)^{\frac{1}{2}}\right]\leq 
 \frac{C \sqrt{m}}{n},
\end{align*}
where we use boundedness of \(\sigma\) under the first set of assumptions and its uniform second-moment bound under the second.
Combining these estimates gives the following bound for the intermediate particles:
\(\varepsilon_{m,n}^Y+\varepsilon_{m,n}^Z+(1+T)m(m-1)/n+C_T\sqrt m/n\).
Furthermore, by \autoref{empirical2}, we have
\begin{align*}
 \sup\limits_{t\in [0,T]}d_{3,1}^{Y,m}
 (\mathbb{P}_{(Y_t^1,\cdots,Y_t^m)}
 ,\mathbb{P}_{(\widetilde{Y}_t^1,\cdots,\widetilde{Y}_t^m)})
& \leq \frac{C\sqrt{m}}{n}, \\
 \int_0^T  d_{3,1}^{Z,m}
 (\mathbb{P}_{(Z_t^{1,1},\cdots,Z_t^{m,m})}
 ,\mathbb{P}_{(\widetilde{Z}_t^{1,1},\cdots,\widetilde{Z}_t^{m,m})})dt
& \leq \frac{C\sqrt{mT}}{n}.
\end{align*}
The triangle inequality gives the stated transfer estimate.
Under either set of
composite regularity assumptions, \(\varepsilon_{m,n}^Y+\varepsilon_{m,n}^Z\leq C_m/n\),
which proves \eqref{5mainestimatetotal} for fixed \(m\).
An estimate for \(m=m_n \longrightarrow \infty\) would also require control of
\(L_{\mathcal M_3}^{(m)}\) or \(L_{\mathcal M_2}^{(m)}\)
as \(m\) increases.
\end{proof}

The following example has Wasserstein errors of order \(n^{-1/2}\)
and weak errors of order \(n^{-1}\) for both the \(Y\)- and
\(Z\)-components.
The diffusion coefficient is time-dependent, and the cancellation
condition in \autoref{asp4.2} is not satisfied.

\begin{example}\label{nondege}\rm
		 Let \(d=k=1\), \(b=f= 0\), \(\tau>0.\)
	Let \((X_0^{i,n})_{i=1}^n\) be i.i.d. with distribution \(  \mathcal{N}(0,\tau^2)\), independent of the Brownian motions \((W^i)_{i=1}^n\).  Let \(\sigma \in C^1([0,T])\) be  nonconstant.  We
	consider the following \(n\)-particle system:
	\begin{align*}
		dX_t^{i,n}= 
		\sigma(t) dW_t^i,\q\ 0\leq t\leq T.
	\end{align*}
	The corresponding McKean--Vlasov equation  is 
	\begin{align*}
		dX_t=\sigma(t) dW_t,\q\ 0\leq 
		t\leq T. 
	\end{align*}
	Set
	\begin{align*}
	q_t=\int_t^T \sigma(s)^2ds,\q\ 
	v_t=\tau^2+\int_0^t \sigma(s)^2ds,\q\ 
s(\mu) =\int_{\mathbb{R}} \sin v \mu (dv).
	\end{align*}
	Let \(h(x,\mu)=\cos x s(\mu)\).
 Then the corresponding  decoupling field is 
 \begin{align*}
 U(t,x,\mu)=e^{-q_t}
 h(x,\mu)=e^{-q_t}\cos x s(\mu) \triangleq G(t,x,s(\mu)).
 \end{align*}
 A direct calculation gives
 \begin{align*}
  \partial_t U 
  +\frac{\sigma(t)^2}{2}D_x^2 U
  +\frac{\sigma(t)^2}{2}\int_{\mathbb{R}} D_v D_\mu U(t,x,\mu)(v)\mu (dv)=0,\q\ 
  U(T,x,\mu)=h(x,\mu).
 \end{align*}
Since
\(G(t,x,a)= e^{-q_t} \cos x a  ,\) 
\(D_a G(t,x,s(\theta_t))=e^{-q_t} \cos x  \not\equiv 0\).
Thus condition (ii) of \autoref{asp4.2} is not satisfied.
The coefficients satisfy \autoref{5forsde}, and \(U\) satisfies
\autoref{asp5.2}.
All positive-order spatial and Lions derivatives of \(U\) and
\(H_Z=-\sigma(t)e^{-q_t} \sin x\, s(\mu)\) are bounded and
Lipschitz, uniformly in time.
Repeated differentiation of the product-measure integrals gives
\autoref{phim} for every fixed \(m\).
Since \(\sigma\in C^1([0,T])\), \autoref{fandsong1} also holds.
Moreover, \autoref{fandsong2m} follows from
\(\mathcal M_3\subset\mathcal M_2\).
Thus both sets of assumptions are satisfied.
Since \(f=0\), \(Y_t^{i,n}=\mathbb{E}\left[ h(X_T^{i,n},\theta_T^n)
 \big| \mathscr{F}^n_t\right]\). 
Conditional expectation and Itô's formula give
\begin{align*}
	 Y_t^{i,n}
	& =\frac{e^{-2q_t}}{n}
	 \cos X_t^{i,n} \sin X_t^{i,n}
	 +\frac{ e^{-q_t} }{n}
	 	\cos X_t^{i,n} \sum_{j\neq i} \sin X_t^{j,n},\\
	 	Z_t^{i,i,n}
	 &	=\frac{\sigma(t)}{n} 
	 	\Big(  e^{-2q_t}
	 	\cos (2X_t^{i,n})-
	  e^{-q_t} 
	 	\sin X_t^{i,n} \sum_{j\neq i}^{} \sin X_t^{j,n} \Big),\\
	 	Z_t^{i,j,n}
	 &	=\frac{\sigma(t)}{n}
	 	e^{-q_t} 
	 	\cos X_t^{i,n} 	\cos X_t^{j,n},\q\ 
	 	i\neq j.
\end{align*}
Subtracting the empirical lift gives
\begin{align*}
 Y_t^{i,n}-\widetilde Y_t^i
 &=\frac{e^{-2q_t}-e^{-q_t}}{n}\cos X_t^{i,n}\sin X_t^{i,n},\\
 Z_t^{i,i,n}-\widetilde Z_t^{i,i}
 &=\frac{\sigma(t)(e^{-2q_t}-e^{-q_t})}{n}\cos(2X_t^{i,n}).
\end{align*}
For \(i\neq j\), \(Z_t^{i,j,n}=\widetilde Z_t^{i,j}\).
Since \(\theta_t=\mathcal N(0,v_t)\) is symmetric,
\(s(\theta_t)=0\). Hence \(Y=Z=0\).
Since \(Y_T=0\),
 \begin{align*}
  \mathcal{W}^2_2 (\mathbb{P}_{Y_T^{i,n}},\delta_0)
  =\mathbb{E}\left[|Y_T^{i,n}|^2\right]=
  \frac{n-1}{4n^2}
  (1-e^{-4v_T })
  +\frac{1}{8n^2}(1-e^{-8v_T}).
 \end{align*}
 Hence  
\begin{align*}
	\mathcal{W}_{2,\|\cdot\|_{\infty}} (\mathbb{P}_{Y^{i,n}},\delta_0)\geq 
		\mathcal{W}_{2 } (\mathbb{P}_{Y_T^{i,n}},\delta_0)
		=\Bigg( \frac{(1-e^{-4v_T})}{4n} 
		+O(\frac{1}{n^2})\Bigg)^{\frac{1}{2}}.
\end{align*}
Moreover, Doob's inequality gives 
\begin{align*}
	\mathcal{W}_{2,\|\cdot\|_{\infty}} (\mathbb{P}_{Y^{i,n}},\delta_0)
	\leq
	\Bigg( 4\mathbb{E}\left[ |Y_T^{i,n}|^2 \right] \Bigg)^{\frac{1}{2}}
	=\Bigg( \frac{(1-e^{-4v_T})}{n} 
	+O(\frac{1}{n^2})\Bigg)^{\frac{1}{2}}.
\end{align*}
Similarly,
\begin{align*}
	 \mathbb{E}\left[|Z_t^{i,i,n}|^2\right]=
	 \frac{\sigma(t)^2}{n^2}
	 \Bigg( \frac{e^{-4  q_t}}{2} (1+e^{-8v_t })
	 +(n-1) e^{-2  q_t} \Big( \frac{1-e^{-2v_t }}{2} \Big)^2
	 \Bigg).
\end{align*}
Therefore,
\begin{align*}
 	\mathcal{W}_{2, L^2} (\mathbb{P}_{Z^{i,i,n}},\delta_0)=
 	\Bigg( \mathbb{E} \int_0^T |Z_t^{i,i,n}|^2 dt \Bigg)^{\frac{1}{2}}
 =	\Bigg( \frac{\widetilde{C}_Z  }{n} 
 	+O(\frac{1}{n^2})\Bigg)^{\frac{1}{2}},
\end{align*}
where \(\widetilde{C}_Z=\int_0^T\sigma(t)^2e^{-2 q_t} \Big( \frac{1-e^{-2v_t }}{2} \Big)^2  dt>0\) because \(\tau>0\) and \( \sigma\not\equiv 0.\)
Let
\begin{align*}
   \varphi_Y(y)=\frac{\cos y}{\|\cos  \|_{ 3,1}}\in \mathcal{F}_{3,1}   
   \q\ \hbox{and} \q\ 
   \psi_Z(z)=\frac{\cos z}{\|\cos  \|_{ 3,1}}\in \mathcal{G}_{3,1} .
\end{align*}
Taylor's formula gives
\begin{align*}
 \Big|1-\cos x-\frac{x^2}{2} \Big|\leq
 \frac{x^4}{24}.
\end{align*}
Since the variables \(\sin X_t^{j,n}\) are independent,
centered, and bounded, the fourth-moment estimate for their sums gives
\begin{align*}
\mathbb{E}\left[ |Y_T^{i,n}|^4\right]
+\sup\limits_{t\in [0,T] }\mathbb{E}\left[ |Z_t^{i,i,n}|^4\right]
\leq 
\frac{C}{n^2}.
\end{align*}
Consequently, there exists some
    \(n_0\in \mathbb{N}\) such that, for  every  \(n\geq n_0,\)  
 \begin{align*}
 	d_{3,1}^Y(\mathbb{P}_{Y_T^{i,n}},\delta_0)& \geq
 	 \Big| \mathbb{E}\left[\varphi_Y (Y_T^{i,n})\right]-\mathbb{E}\left[\varphi_Y (0)\right]\Big|
 \geq \frac{1}{
 	 \|\cos \|_{ 3,1}} \Bigg(1-\mathbb{E}\left[\cos Y_T^{i,n}\right]\Bigg)
 	 \geq \frac{c_Y}{n},\\
 	 \int_0^T 
 	 d_{3,1}^Z(\mathbb{P}_{Z_t^{i,i,n}},\delta_0)dt & \geq
 	 \int_0^T  \Big| \mathbb{E}\left[\psi_Z (Z_t^{i,i,n})\right]-\mathbb{E}\left[\psi_Z (0)\right]\Big| dt
 	  \geq \frac{1}{
 	 	\|\cos \|_{ 3,1}} \int_0^T \Bigg(1-\mathbb{E}\left[\cos Z_t^{i,i,n}\right]\Bigg) dt
 	 	\geq \frac{c_Z}{n}.
 \end{align*}
 Define
 \begin{align*}
 	 \varphi_Y^m (y_1,\cdots,y_m)=
 	 \frac{\cos y_1}{\|\cos   \|_{ 3,1}}\in \mathcal{F}_{3,1}^m \q\ \hbox{and}\q\ 
 	   \psi_Z^m (z_1,\cdots,z_m)=
 	 \frac{\cos z_1}{\|\cos   \|_{ 3,1}}\in \mathcal{G}_{3,1}^m.
 \end{align*}
  These functions have the same \(C_b^{3,1}\)-norms as their
 one-dimensional counterparts.
 Then,
   for fixed \(m\in \mathbb{N}\) and  \(n \geq \max\{m, n_0\},\) we have
 \begin{align*}
  d_{3,1}^{Y,m}(\mathbb{P}_{(Y_T^{1,n},\cdots,Y_T^{m,n})},
  \delta_0^{\otimes m})
 & \geq d_{3,1}^Y(\mathbb{P}_{Y_T^{1,n}},\delta_0)\geq
  \frac{c_Y}{n},\\
 \int_0^T   d_{3,1}^{Z,m}(\mathbb{P}_{(Z_t^{1,1,n},\cdots,Z_t^{m,m,n})},
  \delta_0^{\otimes m})dt
 & \geq \int_0^T d_{3,1}^Z(\mathbb{P}_{Z_t^{1,1,n}},\delta_0)dt\geq
  \frac{c_Z}{n}.
 \end{align*}
Combining these lower bounds with \eqref{5mainestimatetotal}, we obtain
 \begin{align*}
\sup\limits_{t\in [0,T]} d_{3,1}^{Y,m}(\mathbb{P}_{(Y_t^{1,n},\cdots,Y_t^{m,n})},
\delta_0^{\otimes m})&\asymp \frac{1}{n} \q\ \hbox{and} \q\ 
 \int_0^T   d_{3,1}^{Z,m}(\mathbb{P}_{(Z_t^{1,1,n},\cdots,Z_t^{m,m,n})},
\delta_0^{\otimes m})dt \asymp \frac{1}{n}.
 \end{align*}
The comparison constants may depend on \(m,T,\tau,\sigma\),
but are independent of \(n\).
Thus the exponent in \(n\) is optimal for every fixed \(m\),
separately for the \(Y\)- and \(Z\)-components.
\end{example}

\appendix 
	\section{Proofs of
			\autoref{regular-proof-U} and  \autoref{timeinhomogeneous} 
	}\label{proofoflemma}
	
	\begin{proof}[Proof of \autoref{regular-proof-U}]
		We work on an auxiliary probability space with an atomless
		initial sigma-field, independent of the Brownian motion.
		Fix \(t\in[0,T]\) and \(\mu\in\mathcal P_2(\mathbb R^d)\).
		Choose
		\(\xi\in L^2(\mathscr F_t;\mathbb R^d)\) with law
		\(\mu\), independent of the Brownian increments after \(t\).
		Let \((X^{t,\xi},Y^{t,\xi},Z^{t,\xi})\) solve the limiting
		FBSDE, and let
		\((X^{t,a,\mu},Y^{t,a,\mu},Z^{t,a,\mu})\)
		be the decoupled system with initial state \(a\) and measure
		arguments given by the laws of the limiting system.
		By the argument of \autoref{exanduni}, these systems are
		uniquely solvable, and their law flows are determined by \(\mu\).
		Thus \(U(t,a,\mu)=Y_t^{t,a,\mu}\) is deterministic,
		and the usual \(L^2\) estimates give its linear growth in \autoref{asp5.2}(iii).
		
		For the spatial derivatives, \(D_\mu U\), and \(D_vD_\mu U\),
		 we adapt the
		variational arguments in \cite[Sections 5--8]{li_18}
		(see also \cite[Sections 4.3 and 8]{bensoussan_huang_tang_yam_25}
		for a related approach to coupled systems).
		The forward variations and the first backward variations
		have moments of every order, uniformly in the initial time,
		state, measure, and auxiliary variables.
		These estimates use only the positive-order coefficient
		derivative bounds and   also apply to the present
		time-dependent coefficients with linear growth.
		To obtain \(D_aD_\mu U\) and \(D_\mu^2U\), we differentiate
		the first measure variational equation in \(a\) and \(\mu\),
		respectively.
		We first construct the mixed variations.
		The backward equations are linear mean-field BSDEs with
		bounded coefficients.
		Their terminal values and inhomogeneous terms involve
		bounded coefficient derivatives and the variations already
		constructed.
		Let \((P,Q)\) denote either of these variations, with terminal
		value \(\eta\) and inhomogeneous term \(R\).
		For every
		\(r\geq2\),
		the linear estimates in
		\cite[Corollary A.1 and Theorem A.3]{li_18} yield
		\begin{align*}
			\mathbb E\left[
			\sup_{s\in[t,T]}|P_s|^r
			+\Bigg(\int_t^T|Q_s|^2ds\Bigg)^{\frac r2}\right]
			\leq C_r\mathbb E\left[
			|\eta|^r+\Bigg(\int_t^T|R_s|ds\Bigg)^r\right]
			\leq C_r.
		\end{align*}
		Here \(C_r\) is independent of \(t,a,\mu\) and the auxiliary
		variables.
		The same   estimate applies to the vector-valued system.
		The last inequality follows from the preceding moment
		estimates, using H\"older's inequality and Cauchy--Schwarz
		in time for products of martingale integrands.
		Applying the same linear stability estimate to the difference
		quotients, 
		with the truncation argument in
		\cite[Theorem 8.1]{li_18}, identifies these variations with
		the corresponding derivatives.
		The same argument gives joint continuity, including in the
		initial time.
		Differentiating the spatial variational equation in \(\mu\)
		gives the same mixed equation, 
		then  uniqueness yields
		\(D_aD_\mu U=D_\mu D_aU\) in \autoref{asp5.2}(i).
	 Evaluating the variational estimates at the initial time
	 gives the bounds in \autoref{asp5.2}(ii).

		Set \(\theta_s=\mathbb P_{X_s^{t,\xi}}\).
		The flow property and the representation argument in
		\cite[Corollary 9.1]{li_18} give
		\begin{align*}
			Y_s^{t,a,\mu}
			&=U(s,X_s^{t,a,\mu},\theta_s)\q\
			\hbox{and}\q\
			Z_s^{t,a,\mu}
			=D_xU(s,X_s^{t,a,\mu},\theta_s)
			\sigma(s,X_s^{t,a,\mu},\theta_s),
		\end{align*}
		where the second identity holds
		\(ds\otimes d\mathbb P\)-almost everywhere.
		The same identities hold for the limiting system.
		Its backward laws therefore have the
		\(\mathcal W_2\)-continuous versions
		\(m^U(s,\theta_s)\) and \(m^{\mathcal Z}(s,\theta_s)\).
		Using the joint continuity obtained above and the uniform
		integrability provided by the \(L^2\) estimates, the
		short-time argument in \cite[Theorems 9.1 and 9.2]{li_18}
		shows that \(U\) is continuously differentiable in time
		and satisfies \eqref{master2}, with one-sided time
		derivatives at the endpoints.
		The terminal condition in \autoref{asp5.2}(iv) follows
		from the construction.
		
		Finally, \(D_xU\) and \(\sigma\) are bounded.
		The master equation \eqref{master2}, the linear growth of \(U,b,f\), and
		the derivative bounds therefore give
		\begin{align*}
			|U(t,a,\mu)|+|\partial_tU(t,a,\mu)|
			\leq C\Bigg(1+|a|
			+\Big(\int_{\mathbb R^d}|v|^2\mu(dv)\Big)^{\frac12}\Bigg).
		\end{align*}
		This proves \autoref{asp5.2}.
	\end{proof}

	\begin{proof}[Proof of \autoref{timeinhomogeneous}]
		Let \(X^{s,\xi}\) solve \eqref{flowsde}, with 
		\(\mathbb{P}_\xi =\mu\), and set 
		\(\mu_r^{s,\mu} \triangleq \mathbf{P}_{s,r}\mu\).
 We also introduce the decoupled flow
 \begin{align*}
 X_r^{s,x,\mu}
 =x+ \int_s^r b(q, X_q^{s,x,\mu}, \mu_q^{s,\mu}) dq
 +\int_s^r \sigma (q, X_q^{s,x,\mu}, \mu_q^{s,\mu}) dW_q,\q\ 
 s\leq r\leq T.
 \end{align*}
When comparing two decoupled flows, we use the same Brownian motion.
 Spatial and Lions derivatives below are taken with respect to the deterministic parameters \(x\) and \(\mu\), respectively.
We follow the variational argument in
\cite[Theorem 2.18 and  
Theorem 3.4]{chassagen_22_aap}
to estimate the spatial and Lions derivatives of the flow.
The argument applies to time-dependent coefficients because
the derivative bounds are uniform in time.
For example, the spatial Jacobian
 \(J_t\triangleq D_x X_t^{s,x,\mu}\) satisfies
 \begin{align*}
 J_t=I_d+\int_s^t D_x b(u, X_u^{s,x,\mu}, \mathbf{P}_{s,u}\mu)J_udu
 +
 \sum_{j=1}^{d}\int_s^t   D_x \sigma_{\cdot j}(u,X_u^{s,x,\mu},\mathbf{P}_{s,u}\mu)
 J_udW_u^j,\q\ 0\leq s\leq t\leq 
 T.
 \end{align*}
The boundedness of the coefficient derivatives, together
with the Burkholder--Davis--Gundy inequality and Gronwall's lemma, gives
 \begin{align*}
  \mathbb{E}\left[ 
  \sup\limits_{t\in [s,T]} |J_t|^r\right]\leq C_r,\q\ \forall r\geq 2.
 \end{align*}
 We now give the equation for the first Lions derivative.
 Fix \(s,\mu\), and  set
\(X_u^x=X_u^{s,x,\mu}\), \(J_u^v=D_xX_u^{s,v,\mu}\), and
\(M_u^x(v)=D_\mu X_u^{s,x,\mu}(v)\).
Primes denote an independent copy of the Brownian motion
and an initial variable $\xi'$ with law $\mu$, independent
of that Brownian motion; $\mathbb E'$ denotes expectation
over this auxiliary copy only. 
For \(a=b\) or a column of \(\sigma\), set
\begin{align*}
 \mathcal A_{a,u}^x(v)
 =D_xa(u,X_u^x,\mu_u^{s,\mu})M_u^x(v)
 +\mathbb E'\left[D_\mu a(u,X_u^x,\mu_u^{s,\mu})(X_u^{\prime v})J_u^{\prime v}\right]
 +\mathbb E'\left[D_\mu a(u,X_u^x,\mu_u^{s,\mu})(X_u^{\prime\xi'})
 M_u^{\prime\xi'}(v)\right].
\end{align*}
Here
\(M_u^{\prime\xi'}(v)
=\left.D_\mu X_u^{\prime s,x,\mu}(v)\right|_{x=\xi'}\):
we first differentiate with respect to \(\mu\), keeping \(x\) fixed,
and then set \(x=\xi'\).
The Lions chain rule gives the linear equation
\begin{align*}
 M_t^x(v)=\int_s^t\mathcal A_{b,u}^x(v)du
 +\sum_{j=1}^d\int_s^t\mathcal A_{\sigma_{\cdot j},u}^x(v)dW_u^j,
 \q\ M_s^x(v)=0.
\end{align*}
By the boundedness of the coefficient derivatives, the estimates for \(J\),
the Burkholder--Davis--Gundy inequality, and Gronwall's lemma, we obtain
\begin{align*}
 \sup\limits_{s,x,\mu,v}\mathbb E \left[
 \sup\limits_{t\in[s,T]}|M_t^x(v)|^r \right]
 \leq C_r.
\end{align*}

For \(\ell=2,3\), differentiating these equations gives a linear
system for the flow derivatives of total order \(\ell\), with
bounded coefficients and terms involving expectations over
independent copies.
Written componentwise, the inhomogeneous terms are finite sums
of products of bounded coefficient derivatives and flow derivatives
of lower order.
Each product contains at most \(\ell\) flow derivatives.
By Hölder's inequality, for \(1\leq q\leq3\) and \(r\geq2\), we have
\begin{align*}
	\Bigg(
	\mathbb E\left[
	\sup_{u\in[s,T]}
	\prod_{j=1}^{q}|H_u^j|^r
	\right]\Bigg)^{\frac{1}{r}}
	\leq
	\prod_{j=1}^{q}
	\Bigg(
	\mathbb E\left[
	\sup_{u\in[s,T]}|H_u^j|^{qr}
	\right]\Bigg)^{\frac{1}{qr}}.
\end{align*}
Jensen's inequality gives the same bounds for terms involving
expectations over independent copies.
The moment bounds for the derivatives of lower order, together
with the Burkholder--Davis--Gundy inequality and Gronwall's lemma,
give the bounds at order \(\ell\).
To obtain the Lipschitz estimates, subtract the corresponding
equations and write each difference of products as a sum of terms
containing one difference.
Then,
the Lipschitz continuity of the coefficient derivatives and
the estimates for derivatives of lower order  give
the required bounds.
For terms involving expectations over independent copies,
choose \((\xi,\widetilde\xi)\) to be an optimal coupling of
\(\mu,\mu'\), independently of the Brownian motions.
We first condition on the initial variables and then integrate
the estimates with respect to this coupling, using
$\mathbb E\left[ |\xi-\widetilde\xi| \right]
\leq \mathcal W_2(\mu,\mu')$.
The derivatives are obtained by passing to the limit in the
corresponding difference quotients.
All bounds on coefficient derivatives and their Lipschitz
constants are uniform in time, these estimates do not use time derivatives of \(b\) or \(\sigma\).
Thus, for every \(r\geq 2\) and every multi-index
 \((w,\alpha)\) with \(0<w+|\alpha|\leq 3,\)
 we get
  \begin{align*}
& \mathbb{E}
\left[ \sup\limits_{t\in [s,T]}  |\mathcal{D}^{w,\alpha}X_t^{s,x,\mu}(\textbf{v})|^r \right] \leq C_r,\\
&
\mathbb{E}\left[ \sup\limits_{t\in [s,T]}  
|\mathcal{D}^{w,\alpha}X_t^{s,x,\mu}(\textbf{v})
-\mathcal{D}^{w,\alpha} X_t^{s,x',\mu'}(\textbf{v}')|^r\right]
\leq C_r \Big( |x-x'| +\mathcal{W}_2(\mu,\mu')
+  \sum_{j=1}^w|v_j-v_j'|\Big)^r,
 \end{align*}
 where \(C_r>0\) is  independent of \(s,x,x',\mu,\mu'\) and the auxiliary variables. 
 These estimates require only \(\mu, \mu' \in \mathcal{P}_2\).
 Indeed, the law flow satisfies
 \begin{align*}
 \sup\limits_{t\in [s,T]}
 \mathcal{W}_2(\mu_t^{s,\mu},\mu_t^{s,\mu'})
 \leq C \mathcal{W}_2 ( \mu,\mu').
 \end{align*}
  The   decoupled flows satisfy
  \begin{align*}
 &\mathbb{E}\left[
 \sup\limits_{t\in [s,T]} |X_t^{s,x,\mu}-X_t^{s,x',\mu'}|^r \right]
 \leq C_r\Big(    |x-x'|+ \mathcal{W}_2(\mu,\mu')
 \Big)^r.
  \end{align*}
The higher moment bounds above are for the derivatives of the
decoupled flow. Their equations involve bounded coefficient derivatives,
so these bounds require no moments of the initial law beyond order two.
  The stability estimates and the joint continuity of the coefficient
  derivatives imply that
  \begin{align*}
  	   (s,t,x,\mu,\mathbf v)
  	 \longmapsto
  	 \mathcal D^{w,\alpha}X_t^{s,x,\mu}(\mathbf v)
  \end{align*}
  is jointly continuous in \(L^r\).

     We next estimate the Lions derivatives of
     \(V_t^\Phi(s,\mu)=\Phi(\mathbf P_{s,t}\mu)\).
     Let $\xi$ have law $\mu$ and be independent of the Brownian
     increments after time $s$. 
     We regard Lions derivatives of scalar-valued functionals as row vectors.
     The chain rule gives
     \begin{align*}
     	D_\mu V_t^\Phi(s,\mu)(v)
     	=
     	\mathbb E \left[
     	D_\mu\Phi(\mu_t^{s,\mu})(X_t^{s,v,\mu})J_t^v
     	\right]
     	+\mathbb E \left[
     	D_\mu\Phi(\mu_t^{s,\mu})(X_t^{s,\xi,\mu})
     	M_t^\xi(v)
     	\right].
     \end{align*}
     Here $M_t^\xi(v)$ means
     $\left.D_\mu X_t^{s,x,\mu}(v)\right|_{x=\xi}$.
    The derivative with respect to \(\mu\) is taken before setting \(x=\xi\).
    The first term comes from perturbing the initial variable.
Further differentiation in \(\mu\) and the auxiliary variables
gives all derivatives of total order at most three.
Each derivative is a finite sum of expectations of products
of derivatives of \(\Phi\) and derivatives of the flow.
The preceding moment bounds and H\"older's inequality give
uniform bounds for these terms.
     The Lipschitz estimates follow by subtracting these formulas
     and estimating the differences of products.
     For terms evaluated at random initial points, we first condition
     on the initial variables and apply the estimates for deterministic
     initial points. We then integrate with respect to an optimal coupling,
     using Minkowski's inequality.
   Thus only second moments of the initial
   laws are required.
  We obtain
   \begin{align*}
     \sup_{0\leq s\leq t\leq T}
    \max_{0< w+|\alpha|\leq3}
\Big( 
    \|\mathcal D^{w,\alpha}
    V_t^\Phi (s,\cdot)  \|_\infty
    +
    [\mathcal D^{w,\alpha}
    V_t^\Phi (s,\cdot) ]_{Lip}
 \Big)
    \leq C\|\Phi\|_{\mathcal M_3}.
   \end{align*}
   The same representations imply joint continuity of these derivatives, including at \(s=t.\)

  To estimate \(V_t^\Phi(s,\delta_0)\), note that the boundedness
  of \(D_\mu\Phi\) implies that \(\Phi\) is globally
  \(\mathcal W_2\)-Lipschitz.
  The second-moment estimate for the McKean--Vlasov SDE gives
 \begin{align*}
  |V_t^{\Phi}(s,\delta_0)|
  \leq |\Phi(\delta_0)|+C \|\Phi\|_{\mathcal{M}_3} \mathcal{W}_2(\mathbf{P}_{s,t}\delta_0, \delta_0)
  \leq C\|\Phi\|_{\mathcal{M}_3}, \q\  0\leq s\leq t\leq 
  T.
 \end{align*} 

Fix \(t>0\). We now prove differentiability with respect to
the initial time \(s\).
For \(F\in\mathcal M_2(\mathcal P_2(\mathbb R^d))\), define
 \begin{align*}
 	\mathcal L_t F(\mu)
 	=&\int_{\mathbb R^d}
 	\langle D_\mu F(\mu)(v),b(t,v,\mu)\rangle \mu(dv)
 	+\frac12\int_{\mathbb R^d}
 	\operatorname{Tr}\!\left[
 	A(t,v,\mu)D_vD_\mu F(\mu)(v)
 	\right] \mu(dv),
 \end{align*}
 where \(A=\sigma\sigma^\top\).
 For fixed  \(0\leq s <t\leq T\) and \(0<h< t-s\),
 the nonlinear flow property gives
 \(
 V_t^\Phi(s,\mu)
 =V_t^\Phi(s+h,\mathbf P_{s,s+h}\mu).
\)
Set \(F_h(\eta)=V_t^\Phi(s+h,\eta)\).
Applying the Lions chain rule to \(F_h(\mathbf P_{s,r}\mu)\)
on \([s,s+h]\), we obtain
 \begin{align*}
   V_t^\Phi(s,\mu)-V_t^\Phi(s+h,\mu)
  =\int_s^{s+h}
  \mathcal L_r F_h(\mathbf P_{s,r}\mu) dr.
 \end{align*}
The integrands have at most quadratic growth.
The \(\mathcal W_2\)-continuity of the law flow gives the
uniform integrability needed to pass to the limit.
Dividing by \(h\), using the joint continuity established above,
and letting \(h\downarrow0\), we obtain
 \begin{align*} 
 	\lim_{h\downarrow0}
 	\frac{V_t^\Phi(s+h,\mu)-V_t^\Phi(s,\mu)}{h}
 	=-\mathcal L_sV_t^\Phi(s,\cdot)(\mu).
 \end{align*}
 For \(0<s<t,\)  set \(F_0=V_t^\Phi(s,\cdot)\).  Applying the same argument on \([s-h,s]\),   we have
  \begin{align*}
 	V_t^\Phi(s-h,\mu)-V_t^\Phi(s,\mu)
 	=\int_{s-h}^{s}
 	\mathcal L_r F_0(\mathbf P_{s-h,r}\mu) dr.
 \end{align*}
Thus, for \(0<s<t\), the left and right derivatives coincide:
 \begin{align*}
  	\lim_{h\downarrow0}
  \frac{V_t^\Phi(s,\mu)-V_t^\Phi(s-h,\mu)}{h}
  =-\mathcal L_sV_t^\Phi(s,\cdot)(\mu).
 \end{align*}
 At $s=t>0$, the same argument with $F_0=\Phi$ yields
 \begin{align*}
   \lim\limits_{h\downarrow0}
  \frac{V_t^\Phi(t,\mu)-V_t^\Phi(t-h,\mu)}{h}
  =-\mathcal L_t\Phi(\mu).
 \end{align*}
 For fixed \(\mu\), the map
 \(s\longmapsto-\mathcal L_sV_t^\Phi(s,\cdot)(\mu)\)
 is continuous on \([0,t]\).
 Thus \(V_t^\Phi(\cdot,\mu)\) is continuously differentiable,
 with one-sided derivatives at the endpoints, and satisfies
 the backward equation in \autoref{timeinhomogeneous}.
 The terminal condition follows from \(\mathbf P_{t,t}\mu=\mu\).
Combining the derivative estimates with the bound on
\(V_t^\Phi(s,\delta_0)\), we obtain
\begin{align*}
 \sup\limits_{t\in [0,T]}
 \|V_t^\Phi\|_{\mathcal{M}_3([0,t]\times \mathcal{P}_2(\mathbb{R}^d))}
 \leq C\|\Phi\|_{\mathcal{M}_3(\mathcal{P}_2(\mathbb{R}^d))}.
\end{align*}
This completes the proof.
	\end{proof}
	
 	\section*{Acknowledgments}
Part of this work was done while Shuxian Gao was visiting the Beijing International Center for Mathematical Research (BICMR), Peking University. The authors thank  BICMR and the group of Zhenfu Wang for their hospitality, and Zhenfu Wang for helpful discussions and valuable feedback.

\section*{Declaration of competing interest}
The authors declare that they have no competing interests.

\section*{Data availability}
Data sharing is not applicable to this article as no datasets
were generated or analysed during the current study.

\section*{Declaration of generative AI and AI-assisted technologies in the manuscript preparation process}
During the preparation of this work, the authors used ChatGPT (OpenAI) to assist with language editing and background literature searches. All mathematical arguments and references were independently verified by the authors, who reviewed and edited the manuscript and take full responsibility for its final content.

\end{document}